\documentclass[11pt,reqno]{amsart}
\usepackage[margin=1.1in]{geometry}

\usepackage{amsmath,amsthm,amssymb}
\usepackage{mathrsfs}
\usepackage{array} 
 
\usepackage{enumitem}

\usepackage{pgf,tikz,pgfplots}

\usetikzlibrary{arrows}
\usepackage[unicode]{hyperref}

\hypersetup{colorlinks=true, linkcolor=blue, citecolor=blue, filecolor=blue, urlcolor=blue}
\usepackage{comment}
\usepackage{chngcntr}
\usepackage{mathabx}
\usepackage{bm}
\usepackage{etoolbox}
\usepackage{subcaption}
\usepackage[justification=centering]{caption}
\numberwithin{equation}{section}

\usepackage[comma,sort&compress, numbers]{natbib}

\newcommand{\prob}{\stackrel{P}{\rightarrow}}

\newcommand{\one}{{\bf 1}}

\newcommand{\R}{{\mathbb R}}
\newcommand{\E}{\mathbb{E}}
\renewcommand{\P}{\mathbb{P}}

\newcommand{\cF}{{\mathcal F}}
\newcommand{\cH}{{\mathcal H}}

\newcommand{\ra}{\rightarrow}

\newcommand{\Spec}{\mathrm{Spec}}

\newcommand\Aut{\mathrm{Aut}}

\newcommand\dto{\overset{D}{\rightarrow}}

\newtheorem{theorem}{Theorem}[section]

\newtheorem{corollary}{Corollary}[section]

\newtheorem{prop}{Proposition}[section]

\newtheorem{lemma}{Lemma}[section]

\theoremstyle{definition} 
\newtheorem{definition}{Definition}[section]
\newtheorem{remark}{Remark}[section]
\newtheorem{example}{Example}[section]

\def\Cov{{\rm Cov}}

\def\Var{{\rm Var}}

\allowdisplaybreaks

\begin{document}

\title[ Higher-Order Graphon Theory ]{ Higher-Order Graphon Theory: Fluctuations, Degeneracies, and Inference }
\author[Chatterjee, Dan, and Bhattacharya ]{Anirban Chatterjee, Soham Dan, and Bhaswar B. Bhattacharya}
\address{Department of Statistics\\ University of Pennsylvania\\ Philadelphia\\ PA 19104\\ United States}
\email{anirbanc@wharton.upenn.edu}
\address{IBM Research\\ Yorktown Heights\\ NY 10598\\ United States}
\email{Soham.Dan@ibm.com} 
\address{Department of Statistics\\ University of Pennsylvania\\ Philadelphia\\ PA 19104\\ United States}
\email{bhaswar@wharton.upenn.edu}

\keywords{Inhomogeneous random graphs, network analysis, generalized $U$-statistics, subgraph counts.} 

\subjclass[2010]{05C80, 60F05, 05C60}

\begin{abstract} 
Exchangeable random graphs, which include some of the most widely studied network models, have emerged as the mainstay of statistical network analysis in recent years. Graphons, which are the central objects in graph limit theory, provide a natural way to sample exchangeable random graphs. It is well known that network moments (motif/subgraph counts) identify a graphon (up to an isomorphism), hence,  understanding the sampling distribution of subgraph counts in random graphs sampled from a graphon is pivotal for nonparametric network inference. Although there are a few results regarding the asymptotic normality of subgraph counts in graphon models, for many commonly appearing graphons this distribution is degenerate.  This degeneracy phenomenon was overlooked until very recently and its consequences in network inference have remained unexplored. Towards this, we obtain the following results: 

\begin{itemize}

\item We derive the joint asymptotic distribution of any finite collection of network moments in random graphs sampled from a graphon, that includes both the non-degenerate case (where the distribution is Gaussian) as well as the degenerate case (where the distribution has both Gaussian or non-Gaussian components). This provides the higher-order fluctuation theory for subgraph counts in the graphon model. 

\item We develop a novel multiplier bootstrap for graphons that consistently  approximates the limiting distribution of the network moments (both in the Gaussian and non-Gaussian regimes). Using this and a procedure for testing degeneracy, we construct joint confidence sets for any finite collection of motif densities. This provides a general framework for statistical inference based on network moments in the graphon model. 

\end{itemize}
Examples and simulations are provided to validate the general theory.  To illustrate the broad scope of our results we also consider the problem of detecting global structure (that is, testing whether the graphon is a constant function) based on small subgraphs. We propose a consistent test for this problem, invoking celebrated results on quasi-random graphs, and derive its limiting distribution both under the null and the alternative. 
\end{abstract}

\maketitle

\section{Introduction} 
Networks provide a convenient way to represent complex relational data. 
The ubiquitous presence of network data in recent years has led to the  development of several probabilistic models for random graphs that aim to capture  various features of real-world networks. 
One of the most extensively studied models for network data are exchangeable random graphs \cite{aldous1981representations,bickel2009nonparametric,bickel2011method,crane2018probabilistic,diaconisjanson,hoover1982row,lovasz2012large}, where the distribution of the network, given the location of the nodes, remains
unchanged under permutations of the node labels. The celebrated Aldous-Hoover theorem \cite{aldous1981representations,hoover1982row} shows that any exchangeable random graph of infinite size can be generated by first sampling independent node variables $\{U_i\}_{i \geq 1}$
uniformly on $[0, 1]$, and then connecting each pair of nodes $(i, j)$ independently with probability $W(U_i, U_j)$, for some measurable function $W: [0,1]^2\rightarrow[0,1]$ which is symmetric, that is, $W(x,y)=W(y,x)$, for all $x,y\in[0,1]$. 
The function $W$ is commonly referred to as a {\it graphon}. Graphons arise as limits of sequences of dense graphs and is the fundamental object in graph limit theory \cite{borgs2008convergent,borgs2012convergent,lovasz2012large}. 
The theory of graph limits has been extensively studied since its inception and is the backbone of several beautiful results in combinatorics, probability, statistics and related areas (see \cite{lovasz2012large} for a book length treatment). 
As mentioned before, graphons provide a natural way for sampling finite exchangeable random graphs, a concept that has appeared independently in various 
contexts (see \cite{bollobas2007phase,diaconis1981statistics,lovasz2006limits,boguna2003class,bickel2011method,hoff2002latent} among others). We describe this formally in the following definition: 

\begin{definition}[Graphon random graph model]\label{definition:W} Given a graphon $W:[0,1]^{2}\rightarrow[0,1]$, a $W$-{\it random graph} on the set of vertices $[n]:=\{1,2, \ldots, n\}$, hereafter denoted by $G(n, W)$, is obtained by connecting the vertices $i$ and $j$ with probability $W(U_{i},U_{j})$ independently for all $1 \leq i < j \leq n$, where $\{U_{i}: 1 \leq i \leq n\}$ is an i.i.d. sequence of $U[0,1]$ random variables. An alternative way to achieve this sampling is to generate i.i.d. sequences $\{U_{i}: 1 \leq i \leq n\}$ and $\{Y_{ij}: 1\leq i<j \leq n\}$ of $U[0,1]$ random variables and then assigning the edge $(i, j)$ whenever $\{Y_{ij}\leq W(U_{i},U_{j})\}$, for $1 \leq i < j \leq n$. 
\end{definition} 

The model in Definition \ref{definition:W} will be referred to as the $W$-random model or the graphon random graph model. This includes many well-known network models such as, the classical Erd\H{o}s-R\'enyi random graph model (where $W = W_p \equiv p\in [0,1]$ is the constant function), the stochastic block model \citep{bickel2009nonparametric,holland1983stochastic} (and its many variations), smooth graphons \cite{gao2015rate}, random dot-product graphs \citep{dotproduct2018statistical,rubin2022statistical} (see also \citet{lei2021}), and random geometric graphs \citep{penrose2003}, among others.

Network moments or motif counts are the frequencies of particular patterns (subgraphs) in a network, such as the number/density of edges, triangles, or stars in a network \cite{experimental2007network,milo2002network,shen2002network}. Motif counts encode structural information about the geometry of a network and are important summary statistics for potentially large networks. They are the building blocks of network models, such as Exponential Random Graph Models (ERGMs) \cite{hunter2008ergm,chatterjee2013estimating,mukherjee2020degeneracy,mukherjee2023statistics,shalizi2013consistency,xu2021signal,xu2021stein}, and many features of a network of practical interest can be derived from the motif counts, such as clustering coefficient \cite{watts1998collective}, degree distribution \cite{networkdegree}, and transitivity \cite{holland1971transitivity} (see \cite{rubinov2010complex} for others).  
This has propelled the fast growing literature on counting and estimating network motifs under various sampling models (see \cite{BDM,triangletripartite,eden2017approximately,gonen2011counting,klusowski2018counting,lyu2023sampling} and the references therein).

In the framework of the graphon model, network method of moments, introduced in the seminal papers \cite{bickel2011method,borgs2010moments}, is an important tool for inferring properties of the underlying graphon based on the motif counts of the observed network. 
This makes understanding the asymptotic properties of subgraph counts in $W$-random graphs a problem of central importance in network analysis. To this end, suppose $G_n$ is the observed graph sampled from the $W$-random model $G(n, W)$. Then for a finite simple graph\footnote{ A graph is said to be {\it simple} if it has no self-loops and does not contain more than one edge between a pair of vertices. } (motif) $H=(V(H), E(H))$, with $V(H) = \{1,2, \ldots, |V(H)|\}$ such that $|V(H)| \geq 2$, the {\it $H$-th empirical network moment} is the number of copies of $H$ in $G_n$. This will be denoted by $X(H, G_n)$, which can be expressed more formally as: 
\begin{align}\label{eq:XHW}
X(H, G_n)=\sum_{1\leq i_{1}<\cdots<i_{|V(H)|}\leq n}\sum_{H'\in \mathscr{G}_H(\{i_{1},\ldots, i_{|V(H)|} \}) } \prod_{(i_{s}, i_{t}) \in E(H')}\one\left\{Y_{i_{a}i_{b}}\leq W(U_{i_{a}}U_{i_{b}})\right\}, 
\end{align}
where, for any set $S \subseteq [n]$, $\mathscr G_H(S)$ denotes the collection of all subgraphs of the complete graph $K_{|S|}$ on the vertex set $S$ which are isomorphic to $H$.\footnote{Note that we count unlabelled copies of $H$. Several
other authors count labelled copies, which multiplies $X(H, G_n)$ by $|\text{Aut}(H)|$.} Note that
$$| \mathscr{G}_H(\{1,\ldots, |V(H)| \}) | = \dfrac{|V(H)|!}{|\text{Aut}(H)|}, $$
where $\text{Aut}(H)$ is the set of all automorphisms of the graph $H$, that is, the collection of permutations of the vertex set $V (H)$ such that $(x, y) \in E(H)$ if and only if $(\sigma(x), \sigma(y)) \in E(H)$. Therefore, by exchangeability, 
\begin{align}\label{eq:EXHW}
\mathbb E[X(H, G_n)] & =\sum_{1\leq i_{1}<\cdots<i_{|V(H)|}\leq n}\sum_{H'\in \mathscr{G}_H(\{i_{1},\ldots, i_{|V(H)|} \}) } \prod_{(i_{s}, i_{t}) \in E(H')} \P(Y_{i_{a}i_{b}}\leq W(U_{i_{a}}U_{i_{b}}) ) \nonumber \\ 
&= \dfrac{(n)_{|V(H)|}}{|\text{Aut}(H)|}t(H,W) , 
\end{align}
where $(n)_{|V(H)|}:=n(n-1)\cdots (n-|V(H)|+1)$ and 
\begin{align}\label{eq:tHW}
t(H,W)=\int_{[0,1]^{|V(H)|}}\prod_{(a, b) \in E(H)}W(x_{a},x_{b})\prod_{a=1}^{|V(H)|}\mathrm{d} x_{a}  
\end{align}
is the {\it homomorphism density} of the graph $H$ in the graphon $W$. The homomorphism density $t(H,W)$ can be interpreted as the probability that a $W$-random graph on $|V(H)|$ vertices contains the graph $H$, that is,
\begin{align*}
    t(H,W) = \mathbb{P}\left[H\subseteq G(|V(H)|,W)\right].
\end{align*} 
One of the fundamental results in graph limit theory is that the homomorphism densities identify a graphon up to a measure-preserving transformation. The computation in \eqref{eq:EXHW} shows that 
\begin{align}\label{eq:tHGncount}
\hat t(H, G_n) := \frac{|\text{Aut}(H)|}{(n)_{|V(H)|}} X(H, G_n) 
\end{align}
is an unbiased estimate of the homomorphism density $t(H, W)$. To assess the uncertainty and confidence of this estimate it is essential to understand the fluctuations (asymptotic distribution) of $\hat t(H, G_n)$ (equivalently that of $X(H, G_n)$). In fact, many inferential tasks in network analysis, such as estimating the clustering coefficient or testing for global structure, require understanding the joint distribution of multiple (more than 1) subgraph counts. This raises following natural questions:

\begin{enumerate}[label=\textrm{(Q\arabic*)}]\label{enum: Over-Dim Settings}
    \item\label{item:jointdistribution} Given a collection of $r$ graphs $\cH= \{H_1, H_2, \ldots, H_r\}$, what is limiting joint distribution of $\bm X(\cH, G_n) := (X(H_1, G_n), X(H_2, G_n), \ldots, X(H_r, G_n))^\top$?

\item \label{item:confidenceinterval} How can one construct asymptotically valid joint confidence sets for the homomorphism densities 
$\bm t(\cH, W) = (t(H_1, W), t(H_2, W), \ldots, t(H_r, W))^\top$ based on a single realization of the sampled graph $G_n$?  

\end{enumerate}

Despite the growing interest in the random graphon model and the network method of moments, existing results provide only a limited understanding of  these questions. In this paper we develop a framework for studying the asymptotic properties of network moments, which resolves the questions above in its full generality and closes several gaps in the existing literature. We summarize our results in the following sections.

\subsection{Joint Distribution of Network Moments}
\label{sec:distribution}


The asymptotic distribution of subgraph counts in the Erd\H{o}s R\'enyi  model, where $W = W_p  \equiv p$ is a constant function, has been classically studied, using various tools such as $U$-statistics \cite{Nowicki1989,nowicki1988subgraph}, 
method of moments \cite{rucinski1988small},  Stein's method \cite{barbour1989central},  and martingales 
\cite{SJ79, SJ94} (see also \cite[Chapter 6]{JLR}). 
In particular, when $p \in (0, 1)$ is fixed and $G_n \sim G(n, W_p)$, $\bm{X}(\cH, G_n)$ is known to be asymptotically  jointly normal for any finite collection $\cH$ of non-empty graphs (see \cite[Section 9]{janson1991asymptotic}).  For general graphons $W$, the fluctuations of $X(H, G_n)$ (or that of the empirical homomorphism density $t(H, G_n)$ (see \eqref{eq:tFG} for the definition) has received significant attention recently. This began with the work of \citet{bickel2011method}, where the asymptotic Gaussian distribution for subgraph counts was established, under certain sparsity assumptions. Later, using the framework of mod-Gaussian convergence,  F{\'e}ray, M{\'e}liot, and Nikeghbali~\cite{feray2020graphons} derived the asymptotic normality, moderate deviations, and local limit theorems for the empirical homomorphism density. 
The joint Gaussian convergence of a finite collection of empirical homomorphism densities was established in \citet{delmas2021asymptotic}. Recently, \citet{zhang2021berryesseen} derived rates of convergence to
normality (Berry--Esseen type bounds), \citet{zhang2022edgeworth} obtained Edgeworth expansions, and \citet{invariant2022limit} studied connections to exchangeability, for $X(H, G_n)$ (or its related variations). Other related results include central limit theorems with rates of convergence for centered subgraph counts \cite{kaur2021higher}, analysis of localized subgraph counts \cite{maugis2020central}, and motif counts in bipartite exchangeable networks \cite{le2023networkstatistics}.

One interesting feature that has escaped attention is that the limiting normal distribution of the subgraph counts obtained in the aforementioned works can be degenerate depending on the structure of the graphon $W$. For instance, in a planted bisection model \cite{mossel2016consistency} (a stochastic block model with two equal-sized communities and 
 connection probabilities $p$ and $q$ within and between blocks, respectively),  the limiting distribution of network moments such as edges and triangles are degenerate (see Case 4 in Example \ref{example:edgetriangle}). 
This degeneracy phenomenon was noted in \citet{feray2020graphons}, and first systemically studied by \citet{hladky2019limit} when $H=K_R$ is the $R$-clique (the complete graph on $R$ vertices), for some $R \geq 2$. This was extended to general subgraphs $H$ by \citet{BCJ}. Here, it was shown that the usual Gaussian limit of $X(H, G_n)$ is degenerate when a certain regularity function, which encodes the homomorphism density of $H$ incident to a given `vertex' of $W$, is constant almost everywhere. In this case, the graphon $H$ is said to be $H$-regular (see Definition \ref{defn:tabxyHW}) and the asymptotic distribution of $X(H, G_n)$ (with another normalization, differing by a factor $n^{\frac{1}{2}}$) can have two components: a Gaussian component and another independent (non-Gaussian) component which is a (possibly) infinite weighted sum of centered chi-squared random variables. This degeneracy phenomenon also appears in the subsequent work of \citet{chatterjee2024fluctuation} on the fluctuations of the largest eigenvalue. Very recently, \citet{huang2024gaussian} established an invariance principle for $X(H, G_n)$ that encompasses higher-order degeneracies.

In this paper we generalize the above results, which only considers the marginal distribution of a single subgraph count, to joint distributions (recall \ref{item:jointdistribution}). Specifically, we derive the limiting joint distribution of $\bm X(\cH, G_n) := (X(H_1, G_n), X(H_2, G_n), \ldots, X(H_r, G_n))^\top$ (appropriately centered and scaled), when $W$ is irregular with respect to $H_{1},\cdots, H_{q}$ for some $1\leq q \leq r$, and regular with respect to $H_{q+1}, H_{q+2}, \ldots, H_r$.  This is significantly more delicate than marginal convergence, because of the non-Gaussian dependencies between and within the irregular  and regular marginals. Towards this, using the asymptotic theory of generalized $U$-statistics developed by \citet{janson1991asymptotic} and the framework of multiple stochastic integrals we show the following (see Theorem \ref{thm:asymp-joint-dist} for the formal statement):  
\begin{itemize}

\item The limiting distribution of $((X(H_i, G_n)))_{1 \leq i \leq q}$ (the irregular marginals) is a linear stochastic integral in terms of the regularity function. 

\item The limiting distribution of $((X(H_i, G_n)))_{q+1 \leq i \leq r}$ (the regular marginals) is the sum of two independent components; one of which is a multivariate Gaussian and the other is a bivariate stochastic integral in terms of the 2-point conditional kernel of $H_i$ in $W$. 
\end{itemize}
The stochastic integrals are with respect to the same underlying Brownian motion on $[0, 1]$, which captures the dependence between the different marginals. This result goes beyond the well-known sampling convergence (law of large numbers) for subgraph densities (see \cite[Corollary 10.4]{lovasz2012large}) and also the first-order  Gaussian fluctuations. Hence, our results can be thought of as the higher-order fluctuation theory for subgraph counts in the random graphon model. 
The formal statement of the results are given in Section \ref{sec:distributionH}. In Section \ref{sec:examples} we illustrate the general theory in some examples.

\subsection{Joint Confidence Sets}

To use the results described in the previous section for statistical inference (recall \ref{item:confidenceinterval}), one needs to estimate the quantiles of the limiting distribution of the subgraph counts (which depend on the unknown graphon $W$). This is particularly relevant because network moments commonly appear in inferential tasks such as goodness-of-fit and two-sample problems (see \cite{ghoshdastidar2017networkstatistics,levin2019bootstrapping,shao2022higher,lei2016,tang2017nonparametric,ouadah2022motif,bravo2023quantifying,maugis2017statistical,wegner2018identifying} among several others), which requires one to approximate the quantiles of the sampling distribution of the subgraph counts. Towards this different network bootstrap and subsampling methods have been proposed (see \cite{subsamplingnetwork,levin2019bootstrapping,lin2020bootstrap,lunde2023subsampling,green2022bootstrapping,zhang2022edgeworth} and the references therein). However, most of the  existing results on bootstrap consistency are restricted to the regime where the subgraph count has a non-degenerate Gaussian distribution (and some of them also require the network to be sparse). The literature is surprisingly silent in the case where the Gaussian distribution is degenerate. The recent paper \cite{shao2022higher} appears to be the only one that directly address the degeneracy issue in the context of network two-sample testing. However, their result requires the network to be sparse (in addition to other technical conditions) and, hence, does not directly apply to the dense regime.

In this paper we develop a multiplier bootstrap method for approximating the limiting joint distribution of the network moments that remains valid even if the Gaussian distribution is degenerate. On a high level, this entails replacing the graphon $W$ in the limiting distribution with its empirical counterpart (obtained from the observed graph $G_n$) and introducing random Gaussian multipliers (which are independent of $G_n$). For the irregular marginals (where the limiting distribution is Gaussian), the estimate takes the form a linear combination of Gaussians with weights given by an empirical estimate of the regularity function. On the other hand, for the regular marginals, the estimate is a quadratic form in Gaussians in terms of an empirical estimate of the 2-point conditional kernel (see \eqref{eq:ZHestimate} for the formal definition). We show that this estimate, interestingly, converges to the joint distribution of the network moments, conditional on the observed network $G_n$, with no additional assumptions on the graphon $W$ (Theorem \ref{thm:ZnHGn}). We refer to this as the {\it graphon multiplier bootstrap}. Details are given in Section \ref{sec:estimatedistribution}. 

The graphon multiplier bootstrap, however, cannot be directly used for constructing confidence sets for the homomorphism densities, because we do not know which of the subgraphs in $\cH$ are regular with respect to $W$. For this we develop a test for regularity based on a consistent estimate of the variance of the limiting Gaussian distribution (Proposition \ref{prop:H01}). Combining this with the graphon multiplier bootstrap we construct joint confidence sets for the homomorphism densities that are asymptotically valid for any finite collection of subgraphs (Theorem \ref{thm:LHW}). 
To validate the theoretical results, we also study the finite sample performance of the proposed method in simulations. Details are given in Section \ref{sec:LHW}.

\subsection{Testing for Global Structure}

The framework for analyzing the asymptotic properties network moments discussed above, readily applies to many problems in network inference.  To illustrate, here we consider the problem of detecting global structure based on small subgraphs.  
Different variations of this problem have appeared in the literature. For instance, \citet{gao2017testing} considered testing whether a degree-corrected block model has any structure, that is, whether it has a single community (which corresponds to no structure) versus it has more than 1 community (see  also \cite{gao2017subgraph} for related results). In the graphon framework, detecting global structure corresponds to testing the null hypothesis: 
\begin{align}\label{eq:H0Wp}
    H_0: W = p \text{ almost everywhere for some } p \in (0, 1) , 
\end{align}
based on a single observed network $G_n$ from the $W$-random model. For this problem, \citet{fangandrollin2015} proposed a universally consistent test based on the densities of the edge and the 4-cycle, invoking the celebrated result of Chung, Graham, and Wilson \cite{chung1989cycle} about quasi-random graphs.  In this paper, using the same quasi-randomness result, we propose a simpler test statistic which also gives a universally consistent test. Our proposal relies on the observation that $H_0$ in \eqref{eq:H0Wp} holds if and only if $f(W):=t(K_{2},W)^{4} - t(C_{4},W) =0$, where $K_2$ denotes the edge and $C_4$ denotes the 4-cycle. Consequently, a test which rejects for large values of $\hat f(G_n) := \hat t(K_{2},G_n)^{4} - \hat t(C_{4}, G_n)$ (recall \eqref{eq:tHGncount}) will be universally consistent. In Section \ref{sec:structure} we derive the limiting distribution of $\hat f(G_n)$ under both the null and the alternative, using the techniques employed in Section \ref{sec:distribution}. This allows us to obtain a test with precise asymptotic level (unlike the test in \cite{fangandrollin2015} which is conservative) and also understand its fluctuations under the alternative. 
\section{Asymptotic Joint Distribution of Network Moments} 
\label{sec:distributionH}

We begin by introducing the notion of regularity, the conditional 2-point kernel, and other related concepts in Section \ref{subsection: Hom_and_Cond_Hom}. In Section \ref{sec:vertexedgeab} we define the graph join operations. The asymptotic joint distribution of the subgraph counts are given in Section \ref{subsection: Joint_Distribution}.

\subsection{Conditional Homomorphism Density}\label{subsection: Hom_and_Cond_Hom}

Recall the definition of homomorphism density for a simple graph from \eqref{eq:tHW}. This extends easily to multigraphs as follows: The homomorphism density of a fixed multigraph $F=(V(F), E(F))$ (without loops) in a graphon $W$ is defined as: 
\begin{align}\label{eq:tFW}
t(F,W)=\int_{[0,1]^{|V(F)|}}\prod_{(a, b) \in E(F)}W(x_{a},x_{b})\prod_{a=1}^{|V(F)|}\mathrm{d} x_{a}. 
\end{align}
Note that \eqref{eq:tFW} is a natural continuum extension of the homomorphism density of a  fixed graph $F=(V(F), E(F))$ into  finite (unweighted) graph $G=(V(G), E(G))$ defined as: 
\begin{align}\label{eq:tFG}
t(F, G) :=\frac{|\hom(F,G)|}{|V (G)|^{|V (F)|}},
\end{align}
where  $|\hom(F,G)|$ denotes the number of homomorphisms of $F$ into $G$. In fact, $t(F, G)$ is the proportion of maps $\phi: V (F) \rightarrow V (G)$ which define a graph homomorphism. Defining the {\it empirical graphon} associated with the graph $G$ as: 
\begin{align}\label{eq:emp_graph}
	W^G(x, y) :=\boldsymbol 1\{(\lceil |V(G)|x \rceil, \lceil |V(G)|y \rceil)\in E(G)\} , 
\end{align} 
it can be easily verified that $t(F,G) = t(F, W^{G})$. 
(In other words, to obtain the empirical graphon $W^G$ from the graph $G$, partition $[0, 1]^2$ into $|V(G)|^2$ squares of side length $1/|V(G)|$, and let $W^G(x, y)=1$ in the $(i, j)$-th square if $(i, j)\in E(G)$, and 0 otherwise.)

We now introduce the notion of conditional homomorphism densities and $H$-regularity of graphons.

\begin{definition}\label{defn:tabxyHW}(1-point conditional homomorphism density and $H$-regularity)
Fix $a \in V(H)$ and $x \in [0, 1]$. Then 1-{\it point conditional homomorphism density function} of $H$ in a graphon $W$ given the vertex $a$ is defined as: 
    \begin{align*}
        t_a(x,H,W) &:=\mathbb{E}\left[\prod_{(a,b)\in E(H)}W(U_{a},U_{b})\Bigm|U_a = x \right] .
    \end{align*}
   In other words,  $t_a(x, H, W)$ is the homomorphism density of $H$ in the graphon $W$ when the vertex $a \in V(H)$ is marked with the value $x \in [0, 1]$. A graphon $W$ is said to be $H$-\emph{regular} if
\begin{align}\label{eq:H_regular}
 \overline{t}(x,H,W) :=   \frac{1}{|V(H)|}\sum_{a=1}^{|V(H)|}t_a(x,H,W)=t(H,W),
\end{align}
for almost every $x \in [0, 1]$. We say $W$ is $H$-{\it irregular} if it is not $H$-regular. 
\end{definition}

To illustrate the notion of regularity, we consider the following 3 examples: (1) $H=K_2$ is the edge, (2) $H=K_3$ is the triangle, and (3) $H=K_{1, 2}$ is the 2-star. These 3 choices of $H$ will be the running examples throughout the paper. 

\begin{itemize} 

\item {\it $H=K_2$ is the edge}: In this case, for any $a \in \{1, 2\}$, by symmetry, 
$$t_a(x, K_2, W) = \mathbb E[W(U_1, U_2)|U_a=x] = \int_0^1 W(x, y) \mathrm d y := d_W(x) , $$
is the {\it degree function} of $W$. Hence, a graphon $W$ is $K_2$-regular if and only the degree function $d_W$ is constant almost everywhere, that is, $W$ is {\it degree-regular}.

\item {\it $H=K_3$ is the triangle}: Again, by symmetry, for all $1 \leq a \leq 3$, 
$$t_a(x, K_3, W) = \int_0^1 \int_0^1 W(x, y) W(y, z) W(x, z) \mathrm d y \mathrm d z,$$
which is the homomorphism density of triangles incident at the point $x \in [0, 1]$.

\item {\it $H=K_{1, 2}$ is 2-star}: Suppose the vertices of $K_{1, 2}$ are labeled $\{ 1, 2, 3\}$ with the central vertex labeled 1. Then we have the following: 

\begin{itemize}

\item For $a=1$, $t_1(x, K_{1, 2}, W) = \int_0^1 \int_0^1 W(x, y) W(x, z) \mathrm d y \mathrm d z = d_W(x)^2$. 

\item For $a \in \{2, 3\}$, $t_a(x, K_{1, 2}, W) = \int_0^1 \int_0^1 W(x, y) W(y, z) \mathrm d y \mathrm d z$. 

\end{itemize} 
Hence, 
$$\bar t(x, K_{1, 2}, W) = \frac{1}{3} \left( d_W(x)^2 + 2 \int_0^1 \int_0^1 W(x, y) W(y, z) \mathrm d y \mathrm d z \right) . $$
 
\end{itemize}

Next, we define the 2-point conditional homomorphism density and the kernel derived from it. This kernel will arise in the non-Gaussian component of the limiting distribution of $X(H, G_n)$ in the regular regime.

\begin{definition}\label{defn:WH}(2-point conditional homomorphism density) 
Fix $a \ne b \in V(H)$ and $x \in [0, 1]$. Then the 2-{\it point conditional homomorphism density function} of $H$ in a graphon $W$ given the vertices $a$ and $b$ is defined as: 
    \begin{align*}
        t_{a, b}(x, y,H,W) &:=\mathbb{E}\left[\prod_{(a,b)\in E(H)}W(U_{a},U_{b})\Bigm|U_a = x, U_b = y \right] .
    \end{align*}
Further, the {\it 2-point conditional
      kernel of $H$} is defined as:  
  \begin{align}\label{eq:WH}
  W_{H}(x,y)=\frac{1}{2|\Aut(H)|}\sum_{1\leq a\neq b\leq |V(H)|}t_{a, b}(x,y,H,W).
  \end{align}
   \end{definition}

For illustration, as before, we consider the following examples: 
\begin{itemize} 
\item {\it $H=K_2$ is the edge}: In this case, $t_{1, 2}(x, y, K_2, G_n) = t_{2, 1}(x, y, K_2, G_n) = W(x, y)$.  Hence, $W_H(x, y) = \frac{W(x,y)}{2}$, that is, the 2-point conditional kernel is the scaled graphon $W$. 
\item {\it $H=K_3$ is the triangle}: By symmetry, $t_{a,b}(x,y,K_{3},W) = t_{1,2}(x,y,K_{3},W)$ for all $1\leq a\neq b\leq 3$. Hence, the 2-point conditional kernel is given by,
\begin{align*} 
    W_{K_{3}}(x,y) = \frac{1}{2}t_{1,2}(x,y,K_{3},W) = \frac{1}{2}W(x,y)\int_{0}^{1}W(x,z)W(z,y)\mathrm{d}z, 
\end{align*} 
since $|\Aut(K_3)| = 3! =6$. 
\item {\it $H=K_{1, 2}$ is 2-star}: Suppose the vertices of $K_{1, 2}$ are labeled $\{ 1, 2, 3\}$ with the central vertex labeled 1. Then we have the following: 
\begin{itemize}
    \item For $a=1$ and $b\in \{2,3\}$, $$t_{a,b}(x,y,K_{1,2},W) = W(x,y)d_W(x) \text{ and } t_{b,a}(x,y,K_{1,2},W) = W(x,y)d_W(y).$$
    \item For the remaining vertex pairs $(2,3)$ and $(3,2)$, $$t_{2,3}(x,y,K_{1,2},W) = t_{3,2}(x,y,K_{1,2},W) = \int_{0}^{1}W(x,z)W(z,y)\mathrm d z.$$
\end{itemize}
Hence, the 2-point conditional kernel is given by,
\begin{align*}
    W_{K_{1,2}}(x,y) = \frac{1}{2}\left[\int_{0}^{1}W(x,z)W(z,y)\mathrm d z + W(x,y)(d_W(x) + d_{W}(y))\right]. 
\end{align*}
\end{itemize}

\begin{remark}\label{remark:regular} 
Note that a graphon $W$ is $H$-regular (see Definition \ref{defn:tabxyHW}) 
if and only if the 2-point conditional kernel $W_H$ is degree regular. 
This is because, for all $x \in [0, 1]$,  
\begin{align}\label{eq:degreeWH}
\int_0^1 W_{H}(x,y) \mathrm d y  
& = \frac{|V(H)|-1}{2|\Aut(H)|}\sum_{a=1}^{|V(H)|}t_{a}(x,H,W), 
\end{align}
and the RHS of \eqref{eq:degreeWH} is a constant if and only if $W$ is $H$-regular. In fact, if $W$ is $H$-regular, then $\frac{1}{|V(H)|} \sum_{a=1}^{|V(H)|}t_{a}(x,H,W) = t(H, W)$ almost everywhere. Hence, the degree function of $W_H$ becomes 
\begin{align}\label{eq:degreeH}
\int_0^1 W_{H}(x,y) \mathrm d y = \frac{|V(H)|(|V(H)|-1)}{2|\Aut(H)|} \cdot t(H, W) & := d_{W_H}, 
\end{align} 
for almost every $x \in [0, 1]$. 
\end{remark}

Note that $|W_H| \leq \frac{|V(H)|(|V(H)|-1)}{ 2 |\mathrm{Aut}(H)| } := K$. Hence,  $W_H$ defines an operator $T_{W_H}:L^{2}[0,1]\rightarrow L^{2}[0, K]$ as follows:  
\begin{equation}
(T_{W_H}f)(x)=\int_0^1W_H(x, y)f(y) \mathrm d y, 
\label{eq:TW}
\end{equation} 
for each $f\in L^{2}[0,1]$. $T_{W_H}$ is a symmetric Hilbert--Schmidt operator;
thus it  is compact and has a discrete spectrum, that is, it has a countable
multiset of non-zero real eigenvalues, which we denote by 
$\mathrm{Spec}(W_H)$, such that 
\begin{align*}
\sum_{\lambda\in \Spec(W_H)}\lambda^2=\iint W_H(x,y)^2 \mathrm d x \mathrm d y<\infty. 
\end{align*} 

Note that if $W$ is $H$-regular,
then $d_{W_H}$ is an eigenvalue of the operator $T_{W_H}$ (recall
\eqref{eq:TW}) and $\phi \equiv 1 $ is a corresponding eigenvector. In this
case, we will use $\mathrm{Spec}^{-}(W_{H})$ to denote the collection
$\mathrm{Spec}(W_H)$ with the multiplicity of the eigenvalue $d_{W_H}$
decreased by $1$.

\begin{figure}
    \centering
    \includegraphics[scale=0.65]{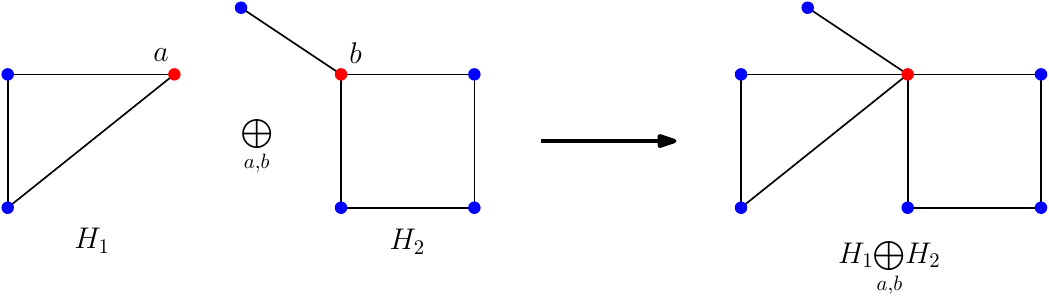}
    \caption{The $(a, b)$-{\it vertex join} of the graphs $H_1$ and $H_2$.}
    \label{fig:v_join}
\end{figure}

\subsection{Graph Join Operations} 
\label{sec:vertexedgeab}

The variance of the subgraph count $X(H, G_n)$ involves different graphs obtained by joining 2 isomorphic copies of $H$. To describe the  asymptotic variance succinctly it is convenient to define some basic graph join operations (as in \cite{BCJ}). To this end, suppose $H=(V(H),E(H))$ 
is a graph with vertex set $V(H) = \{1, 2, \ldots, |V(H)|\}$. Denote by $E^+(H)$ the ordered pairs of edges in $H$, that is, $E^{+}(H)=\{(a,b):1\leq a\neq b\leq r, (a,b) \text{ or } (b,a)\in E(H)\}$.

\begin{definition}\label{defn:H1H2ab}
Suppose $H_1 = (V(H_1), E(H_1))$ and $H_2 = (V(H_2), E(H_2))$ be two graphs with vertex sets $V(H_1) = \{1, 2, \ldots, |V(H_1)|\}$ and $V(H_2) = \{1, 2, \ldots, |V(H_2)|\}$ and edge sets $E(H_1)$ and $E(H_2)$, respectively.

\begin{itemize}

\item {\it Vertex Join}: For $a \in V(H_1)$ and $b \in V(H_2)$, the $(a, b)$-{\it vertex join} of $H_1$ and $H_2$, denoted by 
\begin{align*}
H_{1}\bigoplus_{a,b}H_{2}, 
\end{align*}
is the graph obtained by identifying the $a$-th vertex of $H_1$ with the $b$-th vertex of $H_2$ (see Figure \ref{fig:v_join}).

\item {\it Weak Edge Join}: For $(a,b) \in E^{+}(H_1)$ and $(c,d) \in E^{+}(H_2)$, with $1 \leq a\neq b\leq r$ and $1 \leq c\neq d \leq r$, the $(a, b), (c, d)$-{\it weak edge join} of $H_1$ and $H_2$, denote by 
\begin{align*}
    H_{1}\bigominus_{(a,b),(c,d)}  H_{2} , 
\end{align*}
is the graph obtained identifying the vertices $a$ and $c$ and the vertices $b$ and $d$ and keeping a single edge between the two identified vertices (see Figure \ref{fig:e}).

\item {\it Strong Edge Join}: For $(a,b) \in E^{+}(H_1)$ and $(c,d) \in E^{+}(H_2)$, with $1 \leq a\neq b\leq r$ and $1 \leq c\neq d \leq r$, the $(a, b), (c, d)$-{\it strong edge join} of $H_1$ and $H_2$, 
\begin{align*}
H_{1}\bigoplus_{(a,b),(c,d)}  H_{2} , 
\end{align*} 
is the multi-graph obtained identifying the vertices $a$ and $c$ and the vertices $b$ and $d$ and keeping both the edges between the two identified vertices (see Figure \ref{fig:e}). 
\end{itemize}
\end{definition}

\begin{figure}
    \centering
    \includegraphics[scale=0.65]{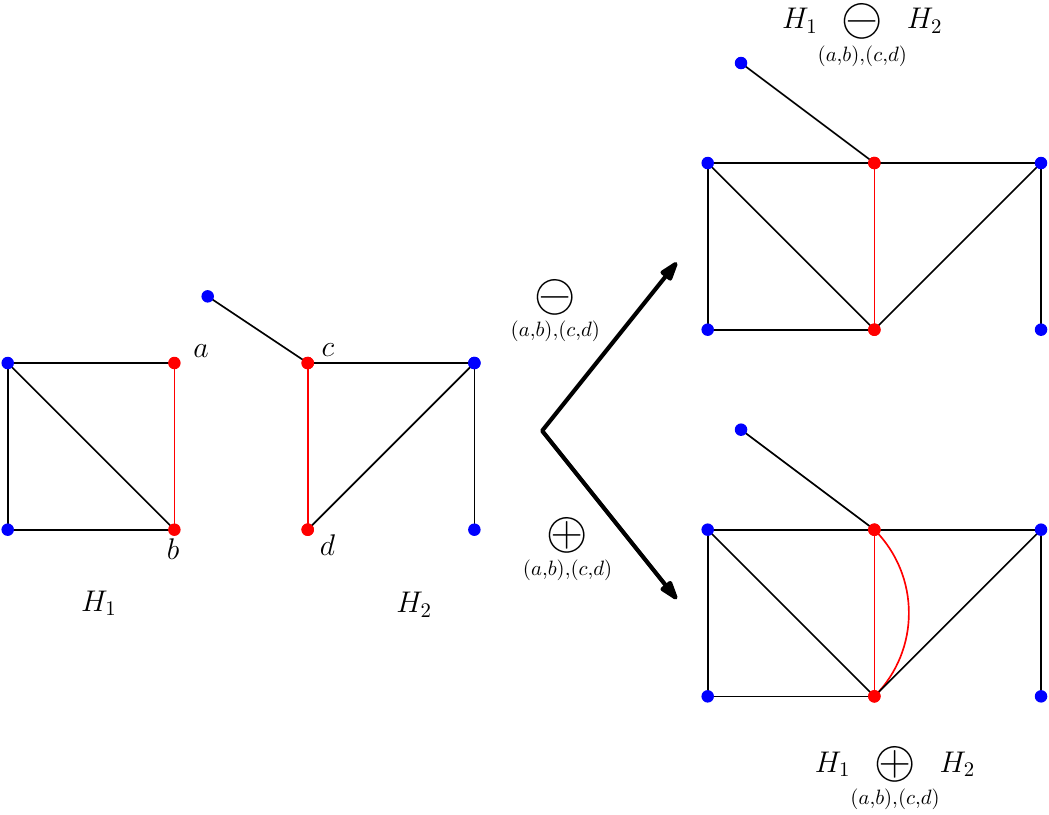}
    \caption{ \small{ The weak and strong edge joins of the graphs $H_1$ and $H_2$. } } 
    \label{fig:e}
\end{figure}

\subsection{Joint Distribution of Subgraph Counts}\label{subsection: Joint_Distribution}

Suppose $\cH=\{H_{1},\cdots, H_{r}\}$ is a collection of finite simple graphs, where $H_i= (V(H_i), E(H_i))$ with vertices labeled $V(H_{i}) = \left\{1,2,\cdots,|V(H_{i})|\right\}$ and $|V(H_{i})|\geq 2$, for $1 \leq i \leq r$. To begin with, for any finite simple graph $H= (V(H), E(H))$, with $|V(H)| \geq 2$, define 
\begin{align}\label{eq:ZnHW}
Z(H, G_n) = 
\left\{
\begin{array}{cc}
  \dfrac{X(H, G_n) - \dfrac{(n)_{|V(H)|}t(H,W)}{|\mathrm{Aut}(H)|}}{n^{|V(H)|-\frac{1}{2}}} & \text{ if } W \text{ is } H\text{-irregular} , \\ \\ 
 \dfrac{X(H, G_n) - \dfrac{(n)_{|V(H)|}t(H,W)}{|\mathrm{Aut}(H)|}}{n^{|V(H)|-1}} & \text{ if } W\text{ is } H\text{-regular} . 
\end{array}
\right. 
\end{align}
Our goal is to derive the limiting distribution of 
\begin{align}\label{eq:graphW}
\bm Z(\mathcal H, G_n) = (Z(H_1, G_n), Z(H_2, G_n), \ldots, Z(H_r, G_n))^\top . 
\end{align}
For this we need to define the following covariance matrix: 

\begin{definition}\label{def:Sigma}
Given a graphon $W$ and finite collection of graphs $\{F_1, F_2, \ldots, F_p\}$, such that $W$ is regular with respect to $F_1, F_2, \ldots, F_p$. Then define a $p \times p$ matrix $\Sigma:=(\sigma_{ij})_{1 \leq i, j \leq p}$ as follows: 

\begin{align}\label{eq:Gvariance}
    \sigma_{ij} = \dfrac{1}{2|\mathrm{Aut}(F_{i})||\mathrm{Aut}(F_{j})|}\sum_{\substack{(a,b)\in E^{+}(F_{i})\\ (c,d)\in E^{+}(F_{j})}}\left[t\left(F_{i}\bigominus_{(a,b),(c,d)}F_{j},W\right) - t\left(F_{i}\bigoplus_{(a,b),(c,d)}F_{j},W\right)\right] , 
\end{align}
for all $1\leq i, j \leq p$.  
\end{definition}

We are now ready to state our result about the limiting distribution of subgraph counts. To this end, denote by $\{B_{t}: t\in [0,1]\}$ the standard Brownian motion on $[0,1]$ and recall the framework of multiple Weiner-It\^{o} stochastic integrals from Section \ref{sec:stochasticintegral}.

\begin{theorem}\label{thm:asymp-joint-dist}
Fix a graphon $W$ and a finite collection of non-empty graphs 
$\mathcal H = \{ H_1, H_2, \ldots, H_r \}$,
such that $W$ is irregular with respect to $H_{1},\cdots, H_{q}$ for some $1\leq q \leq r$ and regular with respect to $H_{q+1}, H_{q+2}, \ldots, H_r$. Then 
\begin{align}\label{eq:ZnHWlimit}
\bm Z(\mathcal H, G_n) \stackrel{D} \rightarrow  \bm{Z}\left(\mathcal H, W\right):=(Z(H_1, W), Z(H_2, W), \ldots, Z(H_r, W))^\top , 
\end{align}
such that 
\begin{itemize}

\item for $1\leq i\leq  q$,
\begin{align*}
  Z(H_i, W):= \int_{0}^{1} \left\{ \dfrac{1}{|\mathrm{Aut}(H_{i})|}\sum_{a=1}^{|V(H_{i})|}t_{a}(x,H_{i},W) - \dfrac{|V(H_{i})|}{|\mathrm{Aut}(H_{i})|}t(H_{i},W) \right\} \mathrm d B_{x} , 
\end{align*} 

\item for  $q+1 \leq i \leq r$,  
\begin{align*}
    Z(H_i, W)    & := G_{i}
    + \int_{0}^{1}\int_{0}^{1} \left\{ W_{H_{i}}(x,y) - \dfrac{|V(H_{i})|\left(|V(H_{i})-1|\right)}{2|\mathrm{Aut}(H_{i})|} t(H_{i},W) \right\} \mathrm d B_{x}\mathrm d B_{y} , 
\end{align*}
where $\bm G= \left(G_{q+1},\cdots, G_{r}\right)\sim N_{r-q} \left(\bm{0}, \Sigma\right)$, with $\Sigma = ((\sigma_{ij}))_{q+ 1 \leq i, j \leq r}$ as in \eqref{eq:Gvariance},  is independent of $\{B_{t}\}_{t\in [0,1]}$.

\end{itemize} 

\end{theorem}

The proof of Theorem \ref{thm:asymp-joint-dist} uses the asymptotic theory of generalized $U$-statistics developed in Janson and Nowicki \cite{janson1991asymptotic}. This allows us to decompose $\bm{X}(\cH, G_n)$ over sums of increasing complexity using a projection method (see also \cite[Chapter 11]{gaussianhilbert}). The terms in the expansion are indexed by the vertices and edges subgraphs of the complete graph of increasing sizes, and the asymptotic behavior of $\bm{X}(\cH, G_n)$ is determined by the joint distribution of non-zero terms indexed by the smallest size graphs. Then the machinery of multiple stochastic integral provides a convenient way to express the dependence among the irregular and regular marginals. The proof is given in Section \ref{sec:distributionpf}.

Theorem \ref{thm:asymp-joint-dist} recovers as special cases a number of existing results. For instance, when $\cH= \{H\}$ is a singleton, we get the marginal distribution of $Z(H, G_n)$, which was proved for cliques in \cite{hladky2019limit} and for general subgraphs in \cite{BCJ}. In this case the limiting distribution can be alternately expressed as in the following corollary, in terms of the graph join operations and the eigenvalues of the kernel $W_H$ (recall the discussion following Remark \ref{remark:regular}).  We show how to derive Corollary \ref{cor:marginaldist} from Theorem \ref{thm:asymp-joint-dist} in Section \ref{sec:proofofmarginaldist}.

\begin{corollary}[{\cite[Theorem 2.9]{BCJ}}] \label{cor:marginaldist}
    Fix a graphon $W$ and a non-empty graph $H=(V(H), E(H))$. Then as $n\rightarrow \infty$, the following hold:
    \begin{itemize}
        \item If $W$ is $H$-irregular,
        \begin{align}\label{eq:ZHGvariance}
            Z(H, G_n)\dto N(0, \tau_{H,W}^2),
        \end{align}
        where 
        \begin{align}\label{eq:deftHW2}
            \tau_{H,W}^2 = \frac{1}{|\Aut(H)|^{2}}\left[\sum_{1\leq a, b\leq
            |V(H)|}t\left(H\bigoplus_{a,b}H,W\right) - |V(H)|^{2}t(H,W)^{2}\right].
        \end{align}
        
        \item If $W$ is $H$-regular, that is, $\tau_{H,W}^2 = 0$, 
        \begin{align}\label{eq:Hregularmarginal}
            Z(H, G_n)\dto \sigma_{H,W} \cdot Z+\sum_{\lambda\in \mathrm{Spec}^{-}(W_{H})}\lambda(Z_{\lambda}^{2}-1)
        \end{align}\label{eq:defsigmaHW2}
        where $Z, \{Z_{\lambda}:\lambda\in \mathrm{Spec}^{-}(W_{H})\}$ are independent $N(0,1)$, 
        \begin{align*}
            \sigma_{H,W}^{2}:=\frac{1}{2|\Aut(H)|^{2}}\sum_{(a,b),(c,d)\in
      E^{+}(H)}\left[t\left(H\bigominus_{(a,b),(c,d)}
      H,W\right)-t\left(H\bigoplus_{(a,b),(c,d)} H,W\right)\right],
        \end{align*}
        and  $\mathrm{Spec}^{-}(W_{H})$ is the multiset $\mathrm{Spec}(W_H)$ with multiplicity of the eigenvalue $d_{W_H}$ decreased by $1$.
    \end{itemize}
\end{corollary}

\begin{remark} 

An interesting question that arises from Corollary \ref{cor:marginaldist}, is whether the distribution in \eqref{eq:Hregularmarginal} always non-degenerate? This is known to be true when $H$ is the clique \cite{hladky2019limit} and if $H$ is the 2-star or the 4-cycle \cite{BCJ}. However,  there are non-trivial cases where the limit in \eqref{eq:Hregularmarginal} is degenerate (see \cite[Example 4.6]{BCJ}). Instances where one (but not both) of the two
components of the distribution in \eqref{eq:Hregularmarginal} is degenerate also has interesting combinatorial properties (see \cite[Section 4]{BCJ}). Additional degeneracies appear in the multivariate case. For instance, the matrix $\Sigma$ in Theorem \ref{thm:asymp-joint-dist} can be singular. This is the case, for example, in the Erd\H{o}s-R\'{e}nyi model where the matrix $\Sigma$ has rank 1 for any finite collections of graphs (see Example \ref{example:randomgraph}). 
\end{remark}

Another case which has appeared in prior work is when all the graph in $\cH$ are irregular with respect to $W$ (see \cite{feray2020graphons} for the univariate case and \cite{delmas2021asymptotic} for the multivariate case). In this case, since a linear stochastic integral has a Gaussian distribution, the limiting distribution of $\bm{Z}(\cH, G_n)$ is multivariate Gaussian (see Theorem 1.5 in \cite{gaussianhilbert}). The covariance matrix of this Gaussian distribution can be expressed in terms of the graph join operations as follows: 

\begin{corollary} [{\cite[Corollary 7.6]{delmas2021asymptotic}}] \label{cor:irregjoint} 
    Fix a graphon $W$ and a finite collection of non-empty graphs $\cH = \{H_{1},\ldots, H_{r}\}$, such that $W$ is $H_i$-irregular for all $1 \leq i \leq r$. Then 
    \begin{align*}
        \bm Z(\cH, G_n)\dto N_{r}(\bm 0, \Gamma) , 
    \end{align*} 
    where $\Gamma = ((\tau_{ij}))_{1 \leq i,j \leq r}$, with 
    \begin{align*}
        \tau_{ij} = \frac{1}{|\Aut(H_i)||\Aut(H_j)|}\left[\sum_{a=1}^{|V(H_i)|}\sum_{b=1}^{|V(H_j)|}t\left(H_i\bigoplus_{a,b}H_j,W\right) - |V(H_i)||V(H_j)|t(H_i,W)t(H_j,W)\right] .
    \end{align*}  
 $($Note that $\tau_{ii} = \tau^2_{H_i, W}$, for $\tau^2_{H_i, W}$ as defined in \eqref{eq:deftHW2} with $H$ replaced by $H_i$.$)$
 \end{corollary} 

\section{Examples} 
\label{sec:examples}

In this section we compute the limiting distribution of $\bm{Z}(\mathcal{H}, G_n)$ in a few examples. We begin with the joint distribution of the counts of edges and triangles.  

\begin{example}\label{example:edgetriangle} (Edges and triangles) Fix a graphon $W$ and suppose $\cH= \{ K_2, K_3\}$ be the edge and the triangle. There are 4-cases depending on whether or not $W$ is $K_2$ or $K_3$-regular.

        \begin{figure}
            \centering
            \includegraphics[scale=0.825]{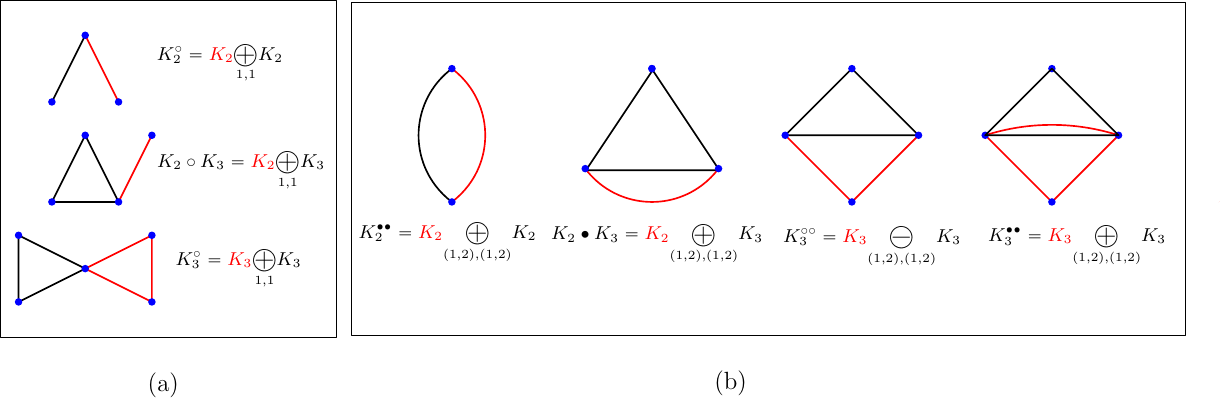}
            \caption{\small{ Graphs obtained from (a) vertex join operations and (b) edge join operations, between a copy of $K_2$ and a copy $K_3$. }} 
            \label{fig:ETEjoinirr}
        \end{figure}

       \begin{itemize} 

\item[{\it Case} 1:] {\it $W$ is irregular with respect to $K_2$ and $K_3$}: In this case, Corollary \ref{cor:irregjoint} applies. To this end, as shown in Figure \ref{fig:ETEjoinirr} (a), denote by $K_2^\circ$, $K_3^\circ$, and $K_2 \circ K_3$ the graphs obtained by the vertex joins of 2 copies of $K_2$, 2 copies of $K_3$, and one copy $K_2$ and one copy of $K_3$, respectively.  Then by Corollary \ref{cor:irregjoint},  
    \begin{align*}
        \begin{pmatrix}
            Z(K_2, G_n)\\
            Z(K_3, G_n)
        \end{pmatrix}\dto
        N_2\left(\bm 0, \begin{pmatrix} 
        \tau_{11} & \tau_{12} \\ 
        \tau_{21} & \tau_{22} 
        \end{pmatrix} \right) , 
    \end{align*}
    where $\tau_{11} := t(K_2^\circ,W) - t(K_2,W)^2$, $\tau_{22} := \frac{1}{4}[t\left(K_3^\circ, W\right) - t(K_3,W)^2]$, and $$\tau_{12} = \tau_{21} := \frac{1}{2}[t( K_2 \circ K_3, W ) - t(K_2,W)t(K_3,W)].$$   For a specific example of a graphon which is irregular with respect to $K_2$ and $K_3$, consider 
    \begin{align}\label{eq:K2irregular}
    \tilde{W}_1(x, y) := \tfrac{1}{2}(x+y),
    \end{align}  for $x , y \in  [0, 1]$. In this case, $t_a(x, K_2, \tilde{W}_1)= d_{\tilde{W}_1}(x) = \frac{1}{2}(x+ \frac{1}{2})$, for $a \in \{1, 2\}$, and $t_b(x, K_3, \tilde{W}_1) = \frac{1}{8}(x^2 + \frac{7x}{6} + \frac{1}{3})$, for $b \in \{1, 2, 3\}$, are both non-constant functions, hence, $\tilde{W_1}$ is $K_2$ and $K_3$-irregular.

\item[{\it Case} 2:] {\it $W$ is regular with respect to $K_2$ and  irregular with respect to $K_3$}: In this case, Theorem \ref{thm:asymp-joint-dist} shows that,
    \begin{align*}  
        \begin{pmatrix} 
        Z(K_2, G_n) \\ 
        Z(K_3, G_n)
        \end{pmatrix}
         \stackrel{D} \rightarrow \begin{pmatrix}
            G + \frac{1}{2} \int_{0}^{1}\int_{0}^{1} \left(W(x,y) - t(K_2,W)\right) \mathrm d B_{x}\mathrm d B_{y}\\
            \frac{1}{2}\int_{0}^{1}\left(\int_0^1\int_0^1W(x,y)W(y,z)W(z,x)\mathrm d y\mathrm d z - t(K_3,W)\right)\mathrm d B_{x}
        \end{pmatrix} , 
        \end{align*}
        where $G\sim N(0,\sigma^2)$ is independent of the Brownian motion $\{B_{t}\}_{t\in [0,1]}$ and
        $$\sigma^2 = \frac{1}{2}\left\{t\left(K_2,W\right) - t\left(K_2^{\bullet \bullet},W\right)\right\} . $$ 
        Here, $K_2^{\bullet \bullet}$ is the graph obtained by the strong edge join of 2 copies of $K_2$, as shown in Figure \ref{fig:ETEjoinirr}(b). For a concrete example of a graphon which is $K_2$-regular and $K_3$-irregular, consider the graphon $\tilde{W}_2$ shown in Figure \ref{fig:K2regularK3irregular}(a). This can be expressed more formally as: 
    \begin{align}\label{eq:K2regular}
        \tilde{W}_2(x, y)=
            \begin{cases}
            1 & \text{ if }(x,y)\in \left[0, \frac{1}{3}\right] \times  \left[\frac{2}{3}, 1 \right] \bigcup  \left[\frac{2}{3}, 1\right] \times  \left[0, \frac{1}{3} \right] , \\
            1 & \text{ if }(x,y)\in \left[\frac{1}{3}, \frac{2}{3}\right] \times  \left[\frac{1}{3}, \frac{2}{3}\right]   , \\ 
            0 & \text{ otherwise} .
            \end{cases}
    \end{align}
        The `graph' representation of this graphon is shown in Figure \ref{fig:K2regularK3irregular}(b), which corresponds to a clique and a disjoint complete bipartite graph of equal block sizes. In this case, the degree function $d_{\tilde{W}_2}(x) = \frac{1}{3}$, for all $x \in [0, 1]$, hence, $\tilde{W}_2$ is $K_2$-regular. Further, for $b \in \{ 1, 2, 3 \}$,   
                \begin{align*}
        t_b(x, K_3, \tilde{W} )=
            \begin{cases}
            0 & \text{ if }(x,y)\in \left[0, \frac{1}{3}\right]  \bigcup  \left[\frac{2}{3}, 1\right] , \\
            \frac{1}{9} & \text{ if }(x,y)\in \left[\frac{1}{3}, \frac{2}{3}\right] , 
            \end{cases} 
    \end{align*} 
    which means $\tilde{W}_2$ is $K_3$-irregular.
    
\begin{figure}[!ht] 
    \centering
    \begin{subfigure}[c]{0.45\textwidth}
        \centering
       \includegraphics[width=0.75\linewidth]{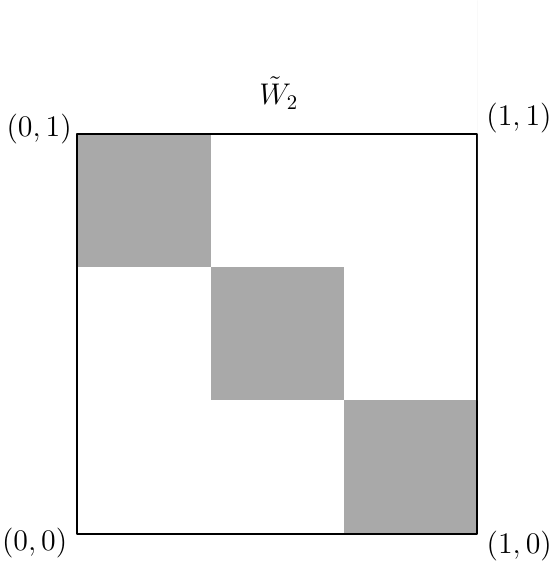} 
       \caption*{(a)}
    \end{subfigure}
    \begin{subfigure}[c]{0.45\textwidth}
        \centering
       \includegraphics[width=0.75\linewidth]{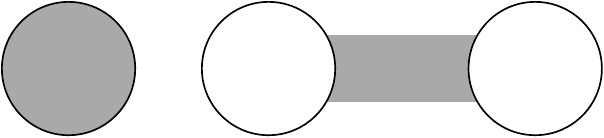} 
          \caption*{(b)}
    \end{subfigure} 
    \caption{ \small{ (a) A $K_2$-regular and $K_3$-irregular graphon $\tilde{W}_2$ and (b) its `graph' representation. } }
    \label{fig:K2regularK3irregular}
\end{figure}

\item[{\it Case} 3:] {\it $W$ is irregular with respect to $K_2$ and regular with respect to $K_3$:} In this case, from Theorem \ref{thm:asymp-joint-dist} we have, 
    \begin{align*}  
        \begin{pmatrix} 
        Z(K_2, G_n) \\ 
        Z(K_3, G_n)
        \end{pmatrix}
         \stackrel{D} \rightarrow \begin{pmatrix}
        \int_{0}^{1} \left(\int_0^1 W(x,y)\mathrm dy - t(K_2,W)\right)\mathrm d B_{x} \\ 
         G + \frac{1}{2} \int_{0}^{1}\int_{0}^{1} \left(\int_0^1 W(x,y)W(y,z)W(z,x)\mathrm d z - t(K_3,W)\right) \mathrm d B_{x}\mathrm d B_{y}
        \end{pmatrix} , 
        \end{align*}
        where $G\sim N(0,\sigma^2)$ is independent of the Brownian motion $\{B_{t}\}_{t\in [0,1]}$ and
$$\sigma^2 = \frac{1}{2}\left\{t\left(K_3^{\circ \circ}, W\right) - t\left(K_3^{\bullet \bullet}, W\right)\right\}.$$ 
Here, $K_3^{\circ \circ}$ and $K_3^{\bullet \bullet}$ are the graphs obtained by the weak and strong edge joins of 2 copies of $K_3$, as shown in Figure \ref{fig:ETEjoinirr}(b), respectively.
For an example of a graphon which is $K_2$-irregular and $K_3$-regular,  consider the graphon $\tilde W_3$ in Figure \ref{fig:K2irregularK3regular}(a). This is a $6 \times 6$ block graphon taking values 1, $\frac{1}{2}$, and 0 in the gray, green, and white blocks, respectively. The `graph' representation of this graphon is shown in Figure \ref{fig:K2irregularK3regular}(b), which corresponds to 2 disjoint complete tri-partite graphs with equal block sizes and a random bipartite graph with edge probability $\frac{1}{2}$ between 2 blocks of the tri-partite graphs. The bipartite connections change the degrees of the corresponding vertices, but do not change their 1-point triangle densities, hence, $\tilde W_3$ is $K_2$-irregular but $K_3$-regular. 

\begin{figure}[!ht] 
    \centering
    \begin{subfigure}[c]{0.45\textwidth}
        \centering
       \includegraphics[width=0.85\linewidth]{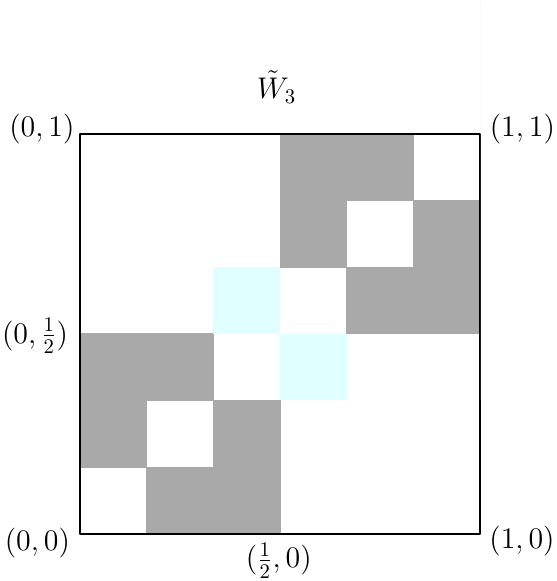} 
       \caption*{(a)}
    \end{subfigure}
    \begin{subfigure}[c]{0.45\textwidth}
        \centering
       \includegraphics[width=0.95\linewidth]{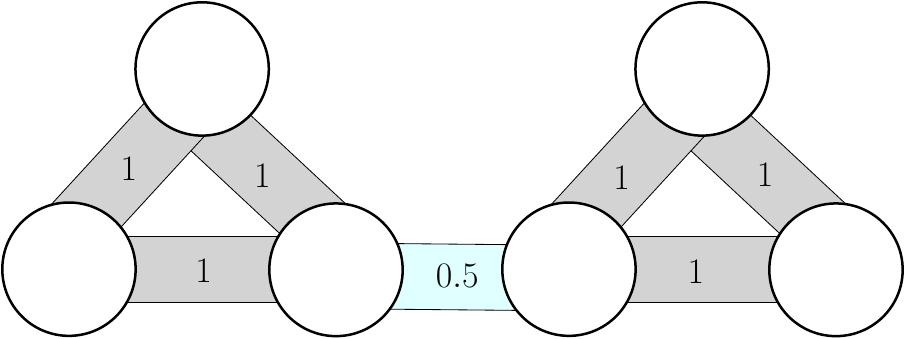} 
          \caption*{(b)}
    \end{subfigure} 
    \caption{ \small{ (a) A $K_2$-irregular and $K_3$-regular graphon $\tilde{W}_3$ and (b) its `graph' representation. } }
    \label{fig:K2irregularK3regular}
\end{figure}

\item[{\it Case} 4:] {\it $W$ is regular with respect to $K_2$ and $K_3$}: Once again, an application of Theorem \ref{thm:asymp-joint-dist} gives, 
        \begin{align*}
            \begin{pmatrix} 
                Z(K_2, G_n) \\ 
                Z(K_3, G_n)
                \end{pmatrix}
                 \stackrel{D} \rightarrow \begin{pmatrix}
                    G_1 + \frac{1}{2} \int_{0}^{1}\int_{0}^{1} \left(W(x,y) - t(K_2,W)\right) \mathrm d B_{x}\mathrm d B_{y}\\
                    G_2 + \frac{1}{2} \int_{0}^{1}\int_{0}^{1} \left(\int_0^1 W(x,y)W(y,z)W(z,x)\mathrm d z - t(K_3,W)\right) \mathrm d B_{x}\mathrm d B_{y}
                \end{pmatrix} .  
        \end{align*}
        Here, $\{B_{t}\}_{t\in [0,1]}$ is the standard Brownian motion and independently 
        $$\begin{pmatrix} 
        G_1 \\ G_2 
        \end{pmatrix} \sim N_2 \left( \bm 0, \begin{pmatrix} 
        \sigma_{11} & \sigma_{12} \\ 
        \sigma_{21} & \sigma_{22} 
        \end{pmatrix} \right) , $$ 
        where $\sigma_{11} := \frac{1}{2}\{t\left(K_2,W\right) - t\left(K_2^{\bullet \bullet},W\right)\}$, $\sigma_{22} := \frac{1}{2}\left\{t\left( K_3^{\circ \circ},W\right) - t\left(K_3^{\bullet \bullet},W\right)\right\}$, and $$\sigma_{12} = \sigma_{21} := \frac{1}{2}\left\{t\left( K_2 \circ K_3 \right) - t\left(K_2 \bullet K_3\right)\right\},$$ 
           with $K_2^{\bullet \bullet}$, $K_3^{\circ \circ}$, $K_3^{\bullet \bullet}$, $K_2 \circ K_3$, and $K_2 \bullet K_3$ as shown in Figure \ref{fig:ETEjoinirr}(b). 
    A simple example of a graphon which is $K_2$ and $K_3$ regular is the constant graphon $W_p \equiv p$, which, incidentally, is $H$-regular, for all finite graphs $H$. More generally, consider the $R$-block graphon, for some $R \geq 1$, with equal block sizes, taking values $a \in [0, 1]$ in the diagonal blocks and $b \in [0, 1]$ in the off-diagonal blocks. This graphon is also $K_2$ and $K_3$ regular. 
    \end{itemize}

\end{example}

Next, we consider the case when $W_p \equiv p$ is the constant function $p \in (0, 1)$, that is, $G_n \sim G(n, p)$ is the Erd\H{o}s-R\'enyi random graph. In this case, it is well-known that the joint asymptotic distribution of the subgraph counts is a multivariate Gaussian (see \cite[Section 9]{janson1991asymptotic}). In the following example we show how to obtain this classical result from our general theorem.

\begin{example}\label{example:randomgraph}(Erd\H{o}s-R\'enyi random graph) 
Suppose $W = W_p \equiv p$, that is, $G_n \sim G(n, p)$ is an Erd\H{o}s-R\'enyi random graph with edge probability $p$. In this case, for any  collection of finite subgraphs $\mathcal{H} = \{H_{1},H_{2}, \cdots, H_{r}\}$, the limiting joint distribution of $Z(\cH, G_n)$ is known to be a multivariate Gaussian. Moreover, the covariance matrix of the Gaussian has rank 1 \cite[Section 9]{janson1991asymptotic}. Here, we show how to derive this result from Theorem \ref{thm:asymp-joint-dist}. Note that $W_p$ is regular with respect to $H_i$, for all $1\leq i\leq r$ (recall \eqref{eq:H_regular}). Also, 
\begin{align*}
    W_{H_{i}}(x,y) = \frac{|V(H_{i})|(|V(H_{i})| - 1)}{2|\Aut{(H_{i})}|}p^{|E(H_{i})|} =\frac{|V(H_{i})|(|V(H_{i})| - 1)}{2|\Aut{(H_{i})}|} t(H_{i},W) , 
\end{align*}
for $1\leq i\leq r$. Hence, the bivariate stochastic integral in Theorem \ref{thm:asymp-joint-dist} vanishes, and the limiting distribution is given by,
\begin{align}\label{eq:ErdosRenyidistconvg}
    \bm{Z}(\mathcal{H}, G_n) \dto N_r(\bm 0, \Sigma) , 
\end{align}
where $\Sigma = (\sigma_{ij})_{1\leq i,j\leq r}$ with 
\begin{align}\label{eq:ErdosRenyicovmat}
    \sigma_{ij} = \dfrac{2 |E(H_{i})||E(H_{j})|}{|\Aut{(H_{i})}||\Aut{(H_{j})}|}p^{|E(H_{i})| + |E(H_{j})| - 1}(1-p).
\end{align} 
Now, for every $2 \leq i \leq r$ and $1 \leq j \leq r$ observe that 
\begin{align*}
    \frac{\sigma_{1j}}{\sigma_{ij}} = \dfrac{|E(H_{1})|}{|E(H_{i})|}\dfrac{|\Aut(H_{i})|}{|\Aut{(H_{1})}|}p^{|E(H_{1})| - |E(H_{i})|} .
\end{align*}
Hence, the $i$-th column of $\Sigma$ is a multiple of the first column of $\Sigma$, for $2 \leq i \leq r$, that is, the matrix $\Sigma$ has rank 1. 

\end{example}

In the next example we discuss the global clustering coefficient, which can be expressed in terms of the counts of 2-stars and triangles.

\begin{example}(Global clustering coefficient/transitivity) 
The {\it global clustering coefficient} of a graph $G$ is defined as (see \cite{luce1949method}): 
\begin{align}\label{eq:tG}
\eta(G) := \frac{3 \times \text{number of triangles in } G}{ \text{the number of 2-stars in } G } = \dfrac{3 X(K_{3}, G)}{X(K_{1, 2}, G)} = \frac{\hat{t}(K_{3}, G)}{\hat{t}(K_{1, 2},G)},
\end{align} 
where $\hat{t}(\cdot, G)$ is defined in \eqref{eq:tHGncount}.
    This is a measure of clustering in the graph $G$ and is also known as the {\it transitivity ratio} (see \cite[Page 243]{wasserman1994social}). Extending \eqref{eq:tG}, one can define the {\it global clustering coefficient} of a graphon $W$ as follows: 
\begin{align*}
    \eta(W) := \dfrac{t(K_{3},W)}{t(K_{1, 2},W)} = \P\left(\text{the nodes (1, 2, 3) are connected} \mid \text{(1, 2) and (1, 3) are connected} \right), 
    \end{align*}
    assuming $t(K_{1, 2},W) > 0$. Clearly, when $G_n \sim G(n, W)$, then $\eta(G_n)$ is a consistent estimate of $\eta(W)$. Using the asymptotic joint distribution of $(X(K_{1, 2}, G_n), X(K_3, G_n) ) $ from Theorem \ref{thm:asymp-joint-dist} and the delta method, we can derive the asymptotic distribution of $\eta(G_n)$. The limit depends on whether or not the graphon $W$ is $K_{1, 2}$ and $K_3$ regular, hence, 4 cases can arise, similar to Example \ref{example:edgetriangle}. We can also quantify the uncertainty of $\eta(G_n)$ in estimating $\eta(W)$, using the results on joint confidence sets in Section \ref{sec:estimatedistribution}.   
        \end{example}

\section{Graphon Multiplier Bootstrap}
\label{sec:estimatedistribution}

Note that the asymptotic distribution of the subgraph counts obtained in Theorem \ref{thm:asymp-joint-dist} depends on the graphon $W$. Hence, to use this result for statistical inference of the homomorphism densities, one needs to estimate quantiles of the asymptotic distribution. When the limit is Gaussian, that is, $W$ is $H$-irregular, this entails estimating the asymptotic variance consistently. However, if the limit is non-Gaussian, which is the case when $W$ is $H$-regular,  this is more delicate. This becomes even more challenging in the multivariate regime, when there is a combination of irregular and regular components.

In this section, we introduce the {\it graphon multiplier bootstrap}, a method for estimating the quantiles of the limiting distribution $\bm{Z}(\cH, G_n)$ (recall \eqref{eq:ZnHW}), based on the observed network $G_n$ itself and additional external randomness. To begin with,
denote by $\bm A_{G_n} = ((w_{st}))_{s,t = 1}^{n}$ the adjacency matrix of $G_n$ and, as before, let $W^{G_n}$ be the empirical graphon corresponding to $G_n$ (recall \eqref{eq:emp_graph}).    
Then the empirical homomorphism density of a graph $H = (V(H), E(H))$ in $G_n$
can be expressed as (recall \eqref{eq:tFW}):  
\begin{align}\label{eq:tHWGn}
    t(H, W^{G_n})=\frac{1}{n^{|V(H)|}}\sum_{\bm s\in [n]^{|V(H)|}} \prod_{(i,j)\in E(H)} w_{s_is_j}.
\end{align}
Moreover, the number of copies of $H$ in the observed in $G_n$ as defined in \eqref{eq:XHW} can be alternatively expressed as: 
\begin{align}\label{eq:XHWGn}
X(H, G_n):=\frac{1}{|\mathrm{Aut}(H)|}\sum_{\bm s\in ([n])_{|V(H)|}} \prod_{(i,j)\in E(H)} w_{s_is_j} ,
\end{align}
where $([n])_{|V(H)|}$ is the set of all $|V(H)|$-tuples ${\bm s}=(s_1,\ldots, s_{|V(H)|})\in  [n]^{|V(H)|}$ with distinct indices.\footnote{For a set $S$, the set $S^N$ denotes the $N$-fold cartesian product $S\times S \times \ldots \times S$.} Note that the cardinality of $(n)_{|V(H)|}$ is $\frac{n!}{(n-|V(H)|)!} = (n)_{|V(H)|}$. To obtain the bootstrap estimate of the asymptotic distribution of $X(H, G_n)$ we need to define the empirical counterparts of the 
1-point and 2-point conditional homomorphism densities (recall \eqref{defn:tabxyHW}). (Hereafter, for simplicity, we will assume $H$ has no isolated vertex.)

\begin{definition}\label{def:Xav} (Empirical 1-point subgraph density) 

Fix $a \in V(H)$ and $v \in V(G_n)$. Denote by $X_a(v, H, G_n)$  the number of injective homomorphism $\phi: V(H) \rightarrow V(G_n)$ such that $\phi(a)=v$. More formally, 
$$X_a(v, H, G_n)= \sum_{\substack{\bm s_{\{a\}^c}}} \prod_{y \in N_H(a)} w_{ v s_y}(G_n) \prod_{(x, y) \in E(H\backslash \{a\})} w_{s_x s_y}(G_n),$$ 
where the sum is over tuples $\bm s_{\{a\}^c} := (s_x)_{x\in V(H)\setminus \{a\} } \in ([n]\setminus \{v\})_{|V(H)|-1}$ and $N_H(a)$ denotes the neighbors of $a$ in the graph $H$. Then the empirical 1-point  subgraph density function is defined as: 
\begin{align}\label{eq:tHGnestimate}
\hat t(v, H, G_n) := \frac{1}{|\Aut(H)|} \sum_{a=1}^{|V(H)|} \frac{X_a(v, H, G_n)}{n^{|V(H)| -1 }} . 
\end{align}
\end{definition} 

Note that \eqref{eq:tHGnestimate} counts (up to constant factors depending on the automorphisms of $H$) the fraction of copies of $H$ in $G_n$ passing through the vertex $v \in V(G_n)$. To illustrate we consider the following examples: 

\begin{itemize} 

\item {\it $H=K_2$ is the edge}: In this case, $\hat t(v, H, G_n) = \frac{d_v}{n}$, where $d_v$ is the degree of the vertex $v$ in $G_n$. 

\item {\it $H=K_3$ is the triangle}: Suppose the vertices of $K_3$ are 
labeled $\{ 1, 2, 3\}$. By symmetry, for all $1 \leq a \leq 3$, $$X_a(v, K_3, G_n) = \sum_{1 \leq s_1 \ne s_2 \leq n} w_{v s_1}  w_{v s_2}w_{s_1 s_2},$$ which is twice the number of triangles in $G_n$ with $v$ as one of the vertex. Therefore,  
$$\hat t(v, K_3, G_n)  = \frac{1}{2n^2} \sum_{1 \leq s_1 \ne s_2 \leq n} w_{v s_1}  w_{v s_2}w_{s_1 s_2}.$$

\item {\it $H=K_{1, 2}$ is the 2-star}: Suppose the vertices of $K_{1, 2}$ are labeled $\{ 1, 2, 3\}$ with the central vertex labeled 1. Then we have the following: 

\begin{itemize}

\item For $a=1$, $X_1(v, K_{1, 2}, G_n) = \sum_{1 \leq s_1 \neq s_2 \leq n} w_{v s_1}  w_{v s_2}$, is twice the number of 2-stars in $G_n$ with $v$ as the central vertex. 

\item For $a \in \{2, 3\}$, $X_a(v, K_{1, 2}, G_n) = \sum_{s_1=1}^n w_{v s_1} (d_{s_1} - 1)$,  is the number of 2-star in $G_n$ where $v$ is a leaf vertex. 

\end{itemize}
Hence, 
$$\hat t(v, K_{1, 2}, G_n)  = \frac{1}{2n^2} \left(\sum_{1 \leq s_1 \ne s_2 \leq n} w_{v s_1}  w_{v s_2} + 2\sum_{s_1=1}^n w_{v s_1} (d_{s_1} - 1) \right).$$ 

\end{itemize}

Next, we define the 2-point subgraph density of $G_n$, which is the empirical analogue of 2-point conditional kernel \eqref{eq:WH}.

\begin{definition}\label{def:Xabuv} (Empirical 2-point subgraph density) Fix $a \ne b \in V(H)$ and $u, v \in V(G_n)$. Denote by $X_{a, b}(u, v, H, G_n)$  the number of injective homomorphism $\phi: V(H) \rightarrow V(G_n)$ such that $\phi(a)=u$ and $\phi(b) = v$. More formally, 
$$X_{a, b}(u, v, H, G_n)= w_{uv}^+ \sum_{\substack{\bm s_{\{a, b\}^c}}}  \prod_{y \in N_H(a) \backslash\{b\}} w_{ u s_y}(G_n) \prod_{y \in N_H(b) \backslash\{a\}} w_{ v s_y}(G_n) \prod_{(x, y) \in E(H\backslash \{a, b\})} w_{s_x s_y}(G_n),$$
where the sum is over tuples $\bm s_{\{a, b\}^c} := (s_x)_{x\in V(H)\setminus \{a, b\} } \in ([n]\setminus \{ u, v \})_{|V(H)|-2}$ and $w_{uv}^+ = w_{uv}$ if $(a, b) \in E(H)$ and $w_{uv} = 1$ otherwise. The 2-point subgraph density is then defined as: 
\begin{align}\label{eq:WGnuv}
\hat W^{G_n}_{H}(u, v) = \frac{1}{2 |\mathrm{Aut}(H)|} \sum_{ 1 \leq a \ne b \leq |V(H)|} \frac{ X_{a, b}(u, v, H, G_n) }{ n^{|V(H)| - 2} } . 
\end{align}
By convention we define $\hat W^{G_n}_{H}(u,u) = 0$ for all $1\leq u\leq n$.
\end{definition} 

Note that $((\hat W^{G_n}_{H}(u, v)))_{1 \leq u, v \leq n}$ is a $n \times n$ matrix which counts (up to constant factors depending on the automorphisms of $H$) the fraction of copies of $H$ in $G_n$ passing through the vertices $u, v \in V(G_n)$.  To illustrate we consider the following examples: 

\begin{itemize} 

\item {\it $H=K_2$ is the edge}: In this case, $\hat W^{G_n}_{H}(u, v) = \frac{1}{2} w_{uv} $, is the scaled adjacency matrix of $G_n$. 

\item {\it $H=K_3$ is the triangle}: Suppose the vertices of $K_3$ are 
labeled $\{ 1, 2, 3\}$. By symmetry, for all $1 \leq a, b \leq 3$, $$X_{a, b}(u, v, K_3, G_n) = \sum_{s_1=1}^n w_{u v}  w_{u s_1}w_{v s_1},$$ which is the number of triangles in $G_n$ with $u$ and $v$ as vertices. Therefore,  
$$\hat W^{G_n}_{K_3}(u, v)  = \frac{1}{2 n} \sum_{1 \leq s_1 \leq n} w_{u v}  w_{u s_1}w_{v s_1} .$$

\item {\it $H=K_{1, 2}$ is the 2-star}: Suppose the vertices of $K_{1, 2}$ are labeled $\{ 1, 2, 3\}$ with the central vertex labeled 1. Then we have the following for $1 \leq u \ne v \leq n$: 

\begin{itemize}

\item For $a=1$ and $b \in \{2, 3\}$, $X_{1, b}(u, v, K_{1, 2}, G_n) = w_{u v} \sum_{s_1 = 1}^n  w_{u s_1} - w_{uv} = w_{uv}(d_u-1)$, is the number of 2-stars in $G_n$ with $u$ as the central vertex and $v$ as the leaf vertex. Similarly, $X_{1, b}(v, u, K_{1, 2}, G_n) = w_{v u} \sum_{s_1 = 1}^n  w_{v s_1} - w_{vu} = w_{vu}(d_v-1)$, is the number of 2-stars in $G_n$ with $v$ as the central vertex and $u$ as the leaf vertex. Also, note that $X_{b, 1}(u, v, K_{1, 2}, G_n) = X_{1, b}(v, u, K_{1, 2}, G_n)$.

\item For $a, b \in \{2, 3\}$, $X_{a, b}(u, v, K_{1, 2}, G_n) = \sum_{s_1=1}^n w_{u s_1} w_{v s_1}$,  is the number of 2-star in $G_n$ with $u, v$ as leaf vertices. 

\end{itemize} 
Hence, 
$$\hat W^{G_n}_{K_{1, 2}} (u, v) = \frac{1}{2n} \left[ w_{u v} \left(d_u + d_v - 2\right) + \sum_{s_1=1}^n w_{u s_1} w_{v s_1}  \right].$$ 

\end{itemize}

With the above definitions we can now describe multiplier bootstrap estimates of the limiting distribution $\bm Z(\cH, W)$. For this, recall that $\mathcal{H} = \{ H_1, H_2, \ldots, H_r \}$ is such that $W$ is irregular with respect to $H_1, H_2, \ldots, H_q$ and $W$ is regular with respect to $H_{q+1}, \ldots, H_{r}$. Suppose 
 $Z_1, Z_2, \ldots, Z_n$ are i.i.d. $N(0, 1)$ independent of the graph $G_n$. Then define 
\begin{align}\label{eq:ZHestimate}
    \hat Z(H_{i}, G_{n}) = 
    \begin{cases}
        \frac{1}{\sqrt n}\sum_{v=1}^{n} ( \hat t(v, H_i, G_n) -  \bar t(H_i, G_n) ) Z_v  & \text{ if }  1 \leq i \leq q, \\ \\ 
   \frac{1}{n} \sum_{1 \leq u, v \leq n} ( \hat W^{G_n}_{H_i}(u, v) - \bar W^{G_n}_{H_i} ) \left(Z_{u}Z_{v} - \delta_{u, v}\right) & \text{ if } q+1 \leq i \leq  r , 
    \end{cases}
\end{align} 
where $\delta_{u, v} = \bm{1}\{u=v\}$, $ \bar t(H_i, G_n) = \frac{1}{n} \sum_{v=1}^n \hat t(v, H_i, G_n)$, and 
\begin{align}\label{eq:averageW}
\bar W^{G_n}_{H_i} = \frac{1}{n^2} \sum_{1 \leq u, v \leq n} \hat W^{G_n}_{H_i}(u, v) . 
\end{align}
Denote 
\begin{align}\label{eq:ZHGnestimate}
\hat {\bm Z}(\cH, G_n) = ( \hat{Z}(H_1, G_n), \hat{Z}(H_2, G_n), \ldots, \hat{Z}(H_r, G_n) )^\top . 
\end{align} 
Note that $\hat {\bm Z}(\cH, G_n) $ depends only on the observed graph $G_n$ and the Gaussian multipliers $Z_1, Z_2, \ldots, Z_n$, but not on the graphon $W$. In the following theorem we show that $\hat {\bm Z}(\cH, G_n) $, conditional on the graph $G_n$, converges to $\bm{Z}(\mathcal{H},W)$ as in Theorem \ref{thm:asymp-joint-dist}.

\begin{theorem}\label{thm:ZnHGn} 
Fix a graphon $W$ and a finite collection of non-empty graphs $\mathcal H = \{ H_1, H_2, \ldots, H_r \}$ such that $W$ is irregular with respect to $H_{1},\cdots, H_{q}$, for some $1\leq q \leq r$, and regular with respect to $H_{q+1}, H_{q+2}, \ldots, H_r$. Suppose $G_n$ is a realization from $G(n, W)$ and $\hat{\bm Z}_n(\mathcal H, G_n)$ be as defined in \eqref{eq:ZHGnestimate}. Then, almost surely as $n\rightarrow\infty$,
\begin{align}\label{eq:ZHGnestimatedistribution}
\hat{\bm Z}_n(\mathcal H, G_n)|G_n \stackrel{D} \rightarrow \bm{Z}(\mathcal{H},W) , 
\end{align} 
where $\bm{Z}(\mathcal{H},W)$ is as in \eqref{eq:ZnHWlimit}. 
\end{theorem}

The proof of Theorem \ref{thm:ZnHGn} is given in Section \ref{sec:limitWGnpf}. It shows the asymptotic distribution of $\hat{\bm Z}_n(\mathcal H, G_n)|G_n$ is the same as that of the subgraph counts $\bm{Z}(H, G_n)$ (recall \eqref{eq:ZnHW}). Hence, we can use the distribution $\hat{\bm Z}_n(\mathcal H, G_n)|G_n$, which depends only on the observed graph $G_n$, to approximate the quantiles of the limiting distribution $\bm{Z}(\mathcal{H},W)$. This allow us to construct joint confidence sets for the homomorphism densities as described in Section \ref{sec:LHW}.

\begin{remark} 
Recently, \citet{lin2020bootstrap} proposed a bootstrap method for approximating the sampling distribution of a network moment in the sparse regime (where the networks have $o(n^2)$ edges), which bears some similarity to our approach. Specifically, the authors use a multiplier bootstrap to estimate the terms in the Hoeffding decomposition of a network moment and also approximates the local subgraph counts based on sampling for fast computation. However, as in most prior work, the bootstrap consistency essentially requires the network moment to be have a non-degenerate Gaussian limit. Moreover, the result only applies in the sparse regime and for the marginal distribution a single network moment that is either acyclic or a cycle. 
\end{remark}

\section{Joint Confidence Sets} 
\label{sec:LHW}
Suppose $\mathcal H = \{ H_1, H_2, \ldots, H_r \}$ is a collection of non-empty graphs, with $H_i= (V(H_i), E(H_i))$ and $|V(H_{i})|\geq 2$, for $1 \leq i \leq r$. In this section, we construct a joint confidence set for the collection of homomorphism densities 
$$\bm t(\cH, W) = (t(H_1, W), t(H_2, W), \ldots, t(H_r, W))^\top, $$ 
given a sample $G_n$ from $G(n, W)$. Note that, although Theorem \ref{thm:ZnHGn} provides a way to estimate the quantiles of the limiting distribution $\bm Z(\cH, W)$, this result cannot be directly applied for constructing a confidence set, because it is a-priori unknown whether or not $W$ is $H_i$-regular for some $1 \leq i \leq r$. For this, we propose a {\it testimation} strategy for constructing joint confidence sets, which first tests for $H_i$-regularity based on the observed graph $G_n$, for $1 \leq i \leq r$, and then uses Theorem \ref{thm:ZnHGn} to estimate the appropriate quantiles. The rest of this section is organized as follows: In Section \ref{sec:HWR} we discuss the test for regularity. Using this and the graphon multiplier bootstrap from the previous section we provide an algorithm for constructing confidence sets in Section \ref{sec:methodLHW}. We illustrate the performance of the algorithm in simulations in Section \ref{sec:simulations}.

\subsection{Testing for Regularity} 
\label{sec:HWR}

Given a graphon $W$ and finite simple graph $H = (V(H), E(H))$, with $|V(H)| \geq 2$, the regularity testing problem for the pair $(H, W)$ can be formulated as follows: 
\begin{align}\label{eq:secondtest}
	\mathrm{H}_{0} :W \text{ is } H \text{-regular} \quad \text{ versus } \quad \mathrm{H}_{1}: W \text{ is not } H\text{-regular} . 
\end{align}
Recall that $W$ is $H$-regular if and only if the asymptotic variance $\tau_{H, W} = 0$ (recall \eqref{eq:deftHW2}). For notational convinience define, 
\begin{align}\label{eq:defRHW}
	R(H,W) = |\text{Aut}(H)|^2 \tau_{H,W}^2 = \sum_{1\leq a,b\leq |V(H)|}t\left(H\bigoplus_{a,b}H,W\right)-|V(H)|^2t(H,W)^2 . 
\end{align}
Clearly, $W$ is $H$-regular if and only if $R(H,W) = 0$. Note that, since the vertex joins to 2 simple graphs is another simple graph, $R(H,W)$ is a function of homomorphism densities of simple graphs. 
Hence,  $R(H,W)$  can be consistently estimated from $G_n$ based on the following simple estimate: 
\begin{align}\label{eq:RGn}
	R(H,G_{n})=\sum_{1\leq a,b\leq |V(H)|}\hat{t}\left(H\bigoplus_{a,b}H,G_{n}\right)-|V(H)|^2\hat{t}(H,G_{n})^2.
\end{align}
The following lemma shows that $R(H,G_{n})$ converges to zero at rate faster than $\sqrt n$ when $W$ is $H$-regular. 
\begin{prop}\label{prop:H01} Suppose $R(H,G_{n})$ be as defined in \eqref{eq:RGn}. Then the following hold: 
\begin{itemize}
 
\item[$(1)$] When $W$ is $H$-regular, $\sqrt{n}R(H,G_{n})\overset{P}{\rightarrow} 0$. 

\item[$(2)$] When $W$ is not $H$-regular, $\sqrt{n}R(H,G_{n})\overset{P}{\rightarrow} \infty$. 

\end{itemize}
\end{prop}

The proof of Proposition \ref{prop:H01} is given in Section \ref{sec:H01pf}. Now, consider the test function
\begin{align}\label{eq:testHW}
	\phi(H, G_{n})= \bm 1 \left\{\sqrt{n}R(H,G_{n})> 1 \right\} . 
\end{align} 
Proposition \ref{prop:H01} implies that under $\mathrm H_{0}$ as in \eqref{eq:secondtest}, $\mathbb{P}(\phi(H, G_{n}))\rightarrow 0$, and under $\mathrm H_{1}$, $\mathbb{P}(\phi_{n}(H,G_{n}))\rightarrow 1$. Hence, the test \eqref{eq:testHW} is consistent for the regularity testing problem \eqref{eq:secondtest}. 

\subsection{Constructing Confidence Sets} 
\label{sec:methodLHW}
Using the test for regularity, we can now describe our algorithm for constructing a joint confidence set for $\bm t(\cH, W)$ as follows: 
\begin{itemize} 
\item For each $1 \leq i \leq r$, consider the hypothesis testing problem: 
\begin{align*}
	\mathrm{H}_{0, i} :W \text{ is } H_i \text{-regular} \quad \text{ versus } \quad \mathrm{H}_{1, i}: W \text{ is not } H_i \text{-regular} . 
\end{align*}
Let $S(\cH, G_n) := \{ 1 \leq i \leq r: \sqrt n R(H_i, G_n)  > 1 \} $, 
be the set of indices where the hypothesis of $H_i$-regularity is rejected.
\item Define 
\begin{align}\label{eq:def-QWGn}
    \bm{Q}(\mathcal{H}, G_{n}) = (Q(H_{1}, G_{n}),\cdots, Q(H_{r}, G_{n})) , 
\end{align}
where (recall \eqref{eq:ZHestimate}) 
\begin{align*}
    Q (H_i, G_{n}) = 
    \begin{cases}
        \frac{1}{\sqrt n}\sum_{v=1}^{n} ( \hat t(v, H_i, G_n) -  \bar t(H_i, G_n) ) Z_v  & \text{ if }  i \in S(\cH, G_n), \\ \\ 
   \frac{1}{n} \sum_{1 \leq u, v \leq n} ( \hat W^{G_n}_{H_i}(u, v) - \bar W^{G_n}_{H_i} ) \left(Z_{u}Z_{v} - \delta_{u, v}\right) & \text{ if } i \notin S(\cH, G_n) , 
    \end{cases} 
\end{align*} 
    with $Z_1, Z_2, \ldots, Z_n$ are i.i.d. $N(0, 1)$ independent of the graph $G_n$.  Denote by $\hat{q}_{1-\alpha, \mathcal H, G_n}$ the $(1-\alpha)$-th quantile of distribution of $\| \bm{Q}(\mathcal{H}, G_{n}) \|_2|G_n$. (Note that the distribution of $\bm{Q}(\mathcal{H}, G_{n})$ given $G_n$ does not depend on the graphon $W$, it only depends on the randomness of the Gaussian multipliers $\{Z_u\}_{1 \leq u \leq n}$, Hence, in practice, $\hat{q}_{1-\alpha, \mathcal H, G_n}$ will be computed from the empirical quantiles of $\| \bm{Q}(\mathcal{H}, G_{n}) \|_2|G_n$ obtained by repeatedly sampling the Gaussian multipliers.) 
\item Define 
\begin{align}\label{eq:jointZHGn}
\tilde{\bm Z}(\mathcal H, G_n) = ( \tilde{Z}(H_1, G_n), \tilde{Z}(H_2, G_n), \ldots, \tilde{Z}(H_r, G_n))^\top , 
\end{align}
with 
\begin{align}\label{eq:ZHconfidenceset}
\tilde Z(H_i, G_n) = 
\left\{
\begin{array}{cc}
  \dfrac{X(H_i, G_n) - \dfrac{(n)_{|V(H_i)|}t(H_i,W)}{|\mathrm{Aut}(H_i)|}}{n^{|V(H_i)|-\frac{1}{2}}} & \text{ if } i \in  S(\cH, G_n), \\ \\ 
 \dfrac{X(H_i, G_n) - \dfrac{(n)_{|V(H_i)|}t(H_i,W)}{|\mathrm{Aut}(H_i)|}}{n^{|V(H_i)|-1}} & \text{ if } i \in S(\cH, G_n).
\end{array}
\right. 
\end{align}
Report the confidence set 
\begin{align}\label{eq:Hconfidenceset}
\mathcal C(\cH, G_n) = \{ \bm t(\cH, W)  : \| \tilde{\bm Z}(\mathcal H, G_n) \|_2 \leq \hat{q}_{1-\alpha, \mathcal H, G_n} \} , 
\end{align}
where $\bm t(\cH, W) = (t(H_1, W), t(H_2, W), \ldots, t(H_r, W))^\top$. 
\end{itemize}

The following theorem shows that the set $C(\cH, G_n)$ is a confidence set for the vector of homomorphism densities $\bm t(\cH, W)$ with asymptotically $\alpha$ coverage probability.

\begin{theorem}\label{thm:LHW} Let $C(\cH, G_n)$ be as defined in \eqref{eq:Hconfidenceset}. Then $\lim_{n \rightarrow \infty} \mathbb P( C(\cH, G_n) ) = 1-\alpha$.
\end{theorem}

The proof of Theorem \ref{thm:LHW} is given in Section \ref{sec:confidencesetpf}. The proof involves showing that $\bm{Q}(\mathcal{H}, G_{n})|G_n$ and $\tilde{\bm Z}(\mathcal H, G_n)$ (recall \eqref{eq:def-QWGn} and \eqref{eq:jointZHGn}, respectively), both converge to the distribution of $\bm Z(\cH, W)$ asymptotically.

\begin{remark}\label{remark:intervalH} (Marginal Confidence Intervals) 
The algorithm for constructing joint confidence sets described above takes a simpler form when $\cH= \{H\}$ is a singleton. In other words, suppose, we want to construct a (marginal) confidence interval for $t(H,W)$. Then, recalling Corollary \ref{cor:marginaldist},  we proceed as follows: 

\begin{itemize}
    
    \item If $\sqrt{n}R(H,G_{n}) > 1$ (that is, $\mathrm{H}_0$ in \eqref{eq:secondtest} is rejected), then define 
    \begin{align*}
        L(H,G_n) = \left[\hat t(H, G_n) - z_{\alpha/2}\frac{|\Aut(H)|\hat{\tau}_{H,G_n}}{\sqrt{n}}, \hat t(H, G_n) + z_{\alpha/2}\frac{|\Aut(H)|\hat{\tau}_{H,G_n}}{\sqrt{n}}\right] ,  
    \end{align*}
    where $\hat t(H, G_n) = \frac{|\Aut(H)|}{(n)_{|V(H)|}}X(H, G_n)$ (as defined in \eqref{eq:tHGncount}), $$\hat{\tau}_{H,G_n}^2 = \frac{1}{n}\sum_{v=1}^n\left( \hat t(v, H, G_n) -  \bar t(H, G_n) \right)^2,$$
    and $z_{\alpha}$ is the $(1-\alpha)$-th quantile of standard Gaussian distribution.
   
    \item If $\sqrt{n}R(H,G_{n}) \leq 1$ (that is, $\mathrm{H}_0$ in \eqref{eq:secondtest} is accepted), then define 
    \begin{align*}
        L(H,G_n) = \left[\hat t(H, G_n) - \hat q_{\alpha/2, H, G_n}\frac{|\Aut(H)|}{n}, \hat t(H, G_n) - \hat q_{1-\alpha/2, H, G_n}\frac{|\Aut(H)|}{n}\right].
    \end{align*}
    Here, $\hat q_{1-\alpha,H,G_n}$ is the $\alpha$-th quantile of the random variable $\frac{1}{n}\sum_{i=1}^{n} \lambda_i(H,G_n)(Z_i^2-1)|G_n$, where $\{\lambda_{i}(H,G_n)\}_{1\leq i\leq n}$ are the eigenvalues of the matrix $((\hat W_{H}^{G_n}(u, v) - \bar W_{H}^{G_n}))_{1 \leq u, v \leq n}$ (recall \eqref{eq:WGnuv} and \eqref{eq:averageW}) and $\{Z_i\}_{1\leq i\leq n}$ are i.i.d. $N(0, 1)$. 
\end{itemize}
From Corollary \ref{cor:marginaldist} and the proof of Theorem \ref{thm:LHW}, it easily follows that $\lim_{n\ra\infty}\P\left(L(H,G_n)\right) = 1-\alpha$, that is, $L(H,G_n)$ is an asymptotically valid confidence interval for $t(H, W)$. 
\end{remark}

\subsection{Simulations} 
\label{sec:simulations} 

In this section we evaluate the performance of the algorithm for constructing joint confidence sets in simulations. 

\subsubsection{Confidence Interval for the Edge Density} 

For the confidence interval of the edge density $t(K_2,W)$ we consider  the following 2 choices of $W$:

\begin{itemize}

\item $W = W_{-}(x,y) := xy$, for $x, y \in [0, 1]$. This graphon is $K_2$-irregular (the degree function $d_{W_-}(x) = \frac{x}{2}$ ) and $t(K_2, W_{-}) = \frac{1}{4}$. 

\item Next, we  consider the $K_2$-regular graphon 
$$W= W_{+}(x,y) :=  \begin{cases}
            \frac{1}{2} & \text{ if }(x,y)\in \left[0, \frac{1}{2}\right] \times  \left[\frac{1}{2}, 1\right] \bigcup  \left[\frac{1}{2}, 1\right] \times  \left[0, \frac{1}{2} \right] , \\
           0 & \text{ otherwise}.   
            \end{cases} 
            $$
            Note that this graphon corresponds to the random bipartite graph with equal block sizes and edge probability $\frac{1}{2}$.  Note that $t(K_2, W_{+}) = \frac{1}{4}$. 
\end{itemize}

\begin{figure}[!ht] 
    \centering
    \begin{subfigure}[c]{0.45\textwidth}
       \includegraphics[width=1\linewidth]{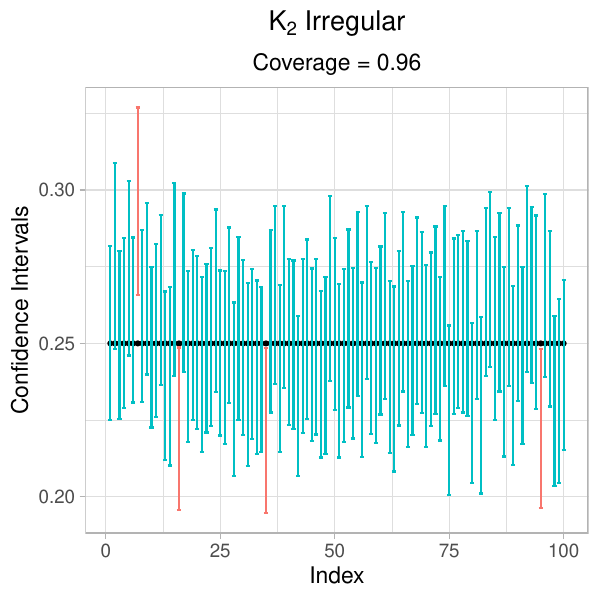} 
    \end{subfigure}
    \begin{subfigure}[c]{0.45\textwidth}
       \includegraphics[width=1\linewidth]{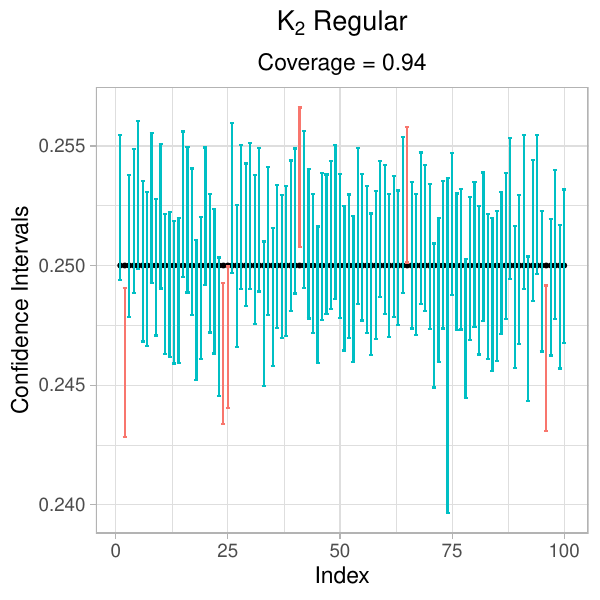} 
    \end{subfigure} 
    \caption{\small{ 100 instances of the $95\%$ confidence interval for edge density $t(K_2, W)$ with (a) $W=W_{-}$ and (b) $W= W_{+}$. } }
    \label{fig:K2CI}
\end{figure}
Using the method described in Remark \ref{remark:intervalH} we construct 100 instances of the 95\% confidence interval for $t(K_2, W)$, when $W = W_{-}$ (Figure \ref{fig:K2CI}(a)) and $W = W_+$ (Figure \ref{fig:K2CI}(b)).  Each interval is computed based on a graph of size $n=400$ sampled from the model $G(n, W)$, for $W  \in \{ W_{-}, W_{+} \}$, and the quantiles are estimated using $1000$ resamples from the conditional distribution. The black horizontal line represents the population edge density $t(K_2, W) = \frac{1}{4}$ (in both cases) and the intervals not containing $\frac{1}{4}$ are shown in red. Observe that in both cases the fraction of intervals containing $\frac{1}{4}$ (the empirical coverage probability) is very close to 0.95, as predicted by the asymptotic theory.

\subsubsection{Joint Confidence Sets for Edge and Triangle Densities}

We now use our algorithm to construct the joint confidence set for the edge and the triangle densities $(t(K_2, W), t(K_3, W))$. Here, 4 possible cases can arise depending on whether or not $W$ is $K_2$ or $K_3$-regular (recall Example \ref{example:edgetriangle}). For each of the 4 graphons considered in Example \ref{example:edgetriangle}, we show below the heatmap of 100 instances of the 95\% confidence ellipsoid  (recall \eqref{eq:Hconfidenceset}). In all the simulations, the confidence sets are computed based on graphs of size $n=400$ and the quantiles are estimated using $1000$ resamples from the conditional distribution. The empirical coverage is given by the fraction of confidence ellipsoids that contain the true homomorphism densities (which is marked by the black point).

\begin{itemize}

\item Figure \ref{fig:K2K3irregK2regK3irreg}(a) shows the joint confidence sets for $(t(K_2, W), t(K_3, W))$ when $W= \tilde{W}_1(x, y) = \frac{1}{2}(x+y)$ (recall \eqref{eq:K2irregular}). This graphon is both $K_2$ and $K_3$-irregular. Also, for this graphon $(t(K_2, \tilde{W}_1), t(K_3, \tilde{W}_1)) =  (\frac{1}{2}, \frac{5}{32})$. In this case, the empirical coverage is 94\%. 

    \begin{figure}[h]
        \centering 
        \begin{subfigure}[c]{0.45\textwidth}
           \includegraphics[width=1\linewidth]{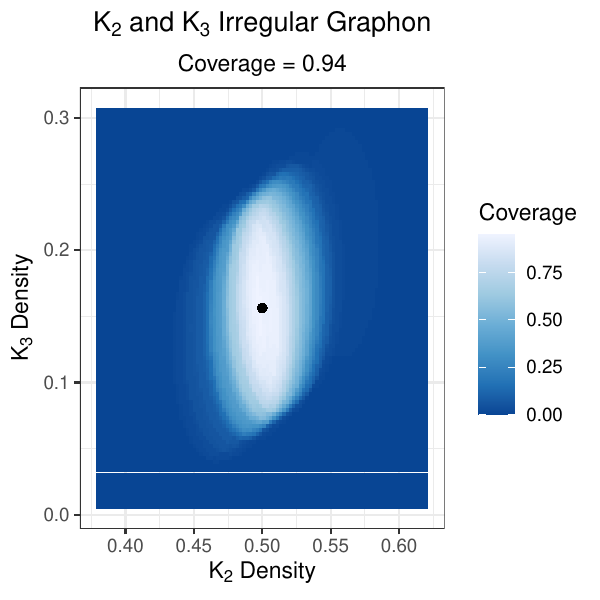} 
           \caption*{\small{(a)}}
        \end{subfigure} 
        \begin{subfigure}[c]{0.45\textwidth}
           \includegraphics[width=1\linewidth]{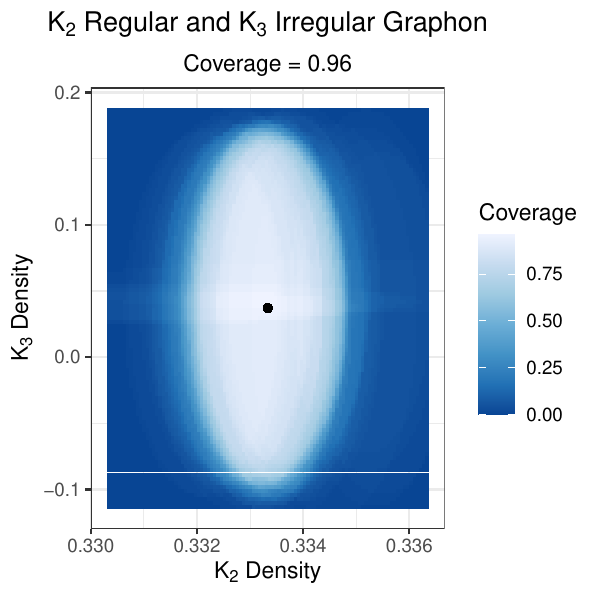}
              \caption*{\small{(b)}} 
        \end{subfigure} 
        \caption{\small{100 instances of the $95\%$ confidence sets for $(t(K_2, W), t(K_3, W))$ with (a) $W= \tilde{W}_1$ and (b) $W= \tilde{W}_2$. }}
        \label{fig:K2K3irregK2regK3irreg}
        \end{figure}

\item Figure \ref{fig:K2K3irregK2regK3irreg}(b) shows the joint confidence sets for $(t(K_2, W), t(K_3, W))$ when $W= \tilde{W}_2(x, y)$ is the graphon in \eqref{eq:K2regular}. This graphon is $K_2$-regular and $K_3$-irregular.  
Furthermore, $(t(K_2,W_3), t(K_3,W_3)) = (\frac{1}{3}, \frac{1}{27})$. In this case, the empirical coverage is 96\%. 

\item Figure \ref{fig:K2K3regK2irregK3reg}(a) shows the joint confidence sets for $(t(K_2, W), t(K_3, W))$ when $W= \tilde{W}_3(x, y)$ is the graphon shown in Figure \ref{fig:K2irregularK3regular}. This graphon is $K_2$-irregular and $K_3$-regular.
Furthermore, a direct computation shows that $(t(K_2, \tilde{W}_3), t(K_3, \tilde{W}_3)) = (\frac{13}{36}, \frac{1}{18})$. In this case, the empirical coverage is 95\%.

     \begin{figure}[h]
    \centering 
    \begin{subfigure}[c]{0.45\textwidth}
       \includegraphics[width=1\linewidth]{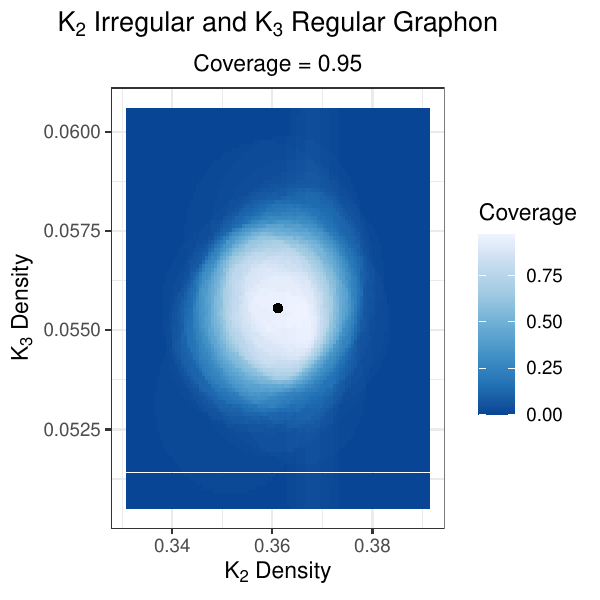} 
       \caption*{\small{(a)}}
    \end{subfigure} 
    \begin{subfigure}[c]{0.45\textwidth}
       \includegraphics[width=1\linewidth]{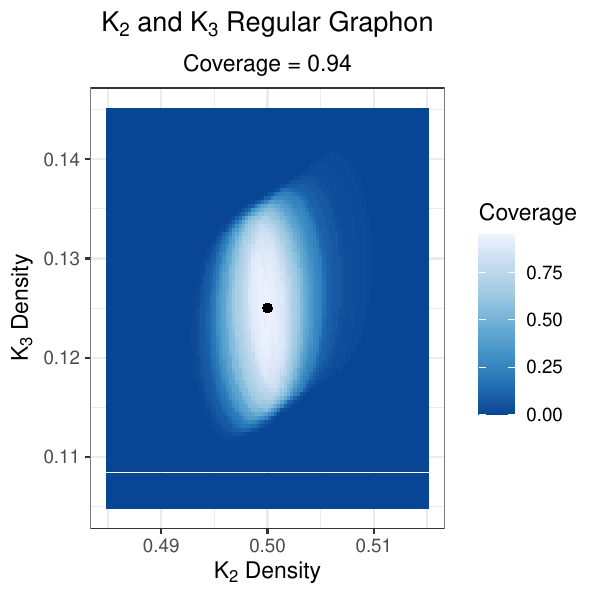}
          \caption*{\small{(b)}} 
    \end{subfigure} 
    \caption{\small{100 instances of the $95\%$ confidence sets for $(t(K_2, W), t(K_3, W))$ with (a) $W= \tilde{W} _3$ the graphon in Figure \ref{fig:K2irregularK3regular} and (b) $W = W_{\frac{1}{2}} \equiv \frac{1}{2}$ the constant graphon $\frac{1}{2}$. }}
    \label{fig:K2K3regK2irregK3reg}
    \end{figure} 
    
\item Figure \ref{fig:K2K3regK2irregK3reg}(b) shows the joint confidence sets for $(t(K_2, W), t(K_3, W))$ when $W= W_{\frac{1}{2}} \equiv \frac{1}{2}$ is the constant function $\frac{1}{2}$ (which corresponds to the Erd\H{o}s-R\'enyi graph $G(n, \frac{1}{2})$. This graphon is both $K_2$ and $K_3$-regular. Also, $(t(K_2,W), t(K_3,W) )  = ( \frac{1}{2}, \frac{1}{8} )$. In this case, the empirical coverage is 94\%. 

\end{itemize}

The results above show that the proposed method achieves the desired coverage in different simulation settings. It is worth recalling our method does have any prior knowledge about whether or not $W$ is $K_2$ or $K_3$-regular. We first test for the presence of regularity and construct the confidence sets depending on the outcome of the tests as described in Section \ref{sec:methodLHW}. 

\section{Testing for Global Structure} 
\label{sec:structure}

Testing global network properties based on counts of subgraphs is a central theme in many statistical network analysis problems.  
A basic problem in this direction is to test whether the network is generated completely at random or whether it has some additional structure. In the context of stochastic models this entails testing whether or not the network has any community structure \cite{gao2017subgraph,gao2017testing}. For graphon models, global structure testing can be formulated as the following hypothesis (recall \eqref{eq:H0Wp}): 
\begin{align}\label{eq:structureH0Wp}
    H_0: W = p \text{ almost everywhere for some } p \in (0, 1) \text{ versus } H_1: W \text{ is non-constant} . 
\end{align}
To find a consistent test for this hypothesis, we need to find a functional $f: \mathcal W \rightarrow \mathbb{R}$ which has the property that $f(W) = 0$ if and only if $W$ is a constant function almost everywhere. A classical result of Chung, Graham, and Wilson about quasi-random graphs \cite{chung1989cycle} implies that the function $f(W) = t(K_{2},W)^{4} - t(C_{4},W)$ satisfies this property (see \cite[Claim 11.53]{lovasz2012large}). Hence, one can construct a consistent test for \eqref{eq:structureH0Wp} by estimating this functional based on the observed graph $G_n$. 
To this end, define,
\begin{align}\label{eq:definitionfGn}
    \hat f(G_n) = \hat t(K_{2}, G_n)^{4} - \hat t(C_{4},G_n),
\end{align}
where, for any graph $H$, $\hat t(H, G_n) = \frac{|\mathrm{Aut}(H)|}{(n)_{|V(H)|}}X(H, G_n)$ (as defined in \eqref{eq:tHGncount}).

In the following theorem we derive the asymptotic distribution of $\hat f(G_n)$ under the null hypothesis.  

\begin{prop}\label{ppn:C4}
Under $H_0$ as in \eqref{eq:H0Wp}, 
\begin{align}\label{eq:cltfGn}
n^\frac{3}{2} \hat f(G_n) \overset{D} \rightarrow N\left(0, \vartheta^2 \right) . 
\end{align} 
where $\vartheta^2 := 32t(K_2, W)^{6}(1-t(K_2, W))^{2}$. 
\end{prop}

The proof is given in Section \ref{sec:structureH0pf}. As in the proof of Theorem \ref{thm:asymp-joint-dist}, it uses the method of orthogonal projections. 
One interesting feature of the statistic $\hat f(G_n)$ is that it has fluctuations of order $O(n^{-\frac{3}{2}})$ under $H_0$, even though 
we know from  Example \ref{example:randomgraph} that both $\hat t(K_{2}, G_n)$ and  $\hat t(C_{4},G_n)$ have fluctuations of $O(1/n)$. This means $\hat f(G_n)$ cancels the contributions from $\hat t(K_{2}, G_n)$ and  $\hat t(C_{4},G_n)$ in the $O(1/n)$ scale and the leading asymptotic contribution of $\hat f(G_n)$ is determined from the third-order projection.  The same scaling appears in the Erd\H{o}s-Zuckerberg (EZ) statistic considered in \cite{gao2017testing}, for testing the presence of community structure in degree-corrected block models. 

To apply Proposition \ref{ppn:C4} to test the hypothesis \eqref{eq:structureH0Wp}, we need to consistently estimate the asymptotic variance in \eqref{eq:cltfGn}. Towards this, note that, since $\hat{t}(K_2, G_n) \stackrel{P} \rightarrow t(K_2, W)$ (follows from \eqref{eq:tGntWGn} and Corollary 10.4 from \cite{lovasz2012large}), by Slutsky's lemma: 
\begin{align}\label{eq:defTnstructure}
 T_n :=    n^\frac{3}{2}\dfrac{ \hat f (G_n)}{4\sqrt{2}\ \hat{t}(K_{2},G_n)^{3}(1- \hat{t}(K_{2},G_n))}\overset{D}{\rightarrow} N(0,1) , 
\end{align}
under $H_0$. Hence, the test which rejects when $|T_{n}| > z_{\alpha/2}$ is asymptotically level $\alpha$. 
In the following proposition we show that this test consistent, that is, it can detect any non-constant graphon with probability going to 1 (see Section \ref{sec:structureH1pf} for the proof).

\begin{prop}\label{prop:consistency}
For any graphon $W$ such that $|t(K_{2},W)^{4} - t(C_{4},W)|>0$ 
we have,
\begin{align*}
    \mathbb{P}\left(|T_{n}|> z_{\alpha/2} \right)\rightarrow 1. 
\end{align*} 
\end{prop} 

Proposition \ref{prop:consistency} provides a test for  \eqref{eq:structureH0Wp} that has precise asymptotic level and is consistent in detecting all non-constant graphons. In comparison, the asymptotic null distribution of the test statistic in \citet{fangandrollin2015} is unknown and the resulting test is conservative (see \cite[Remark 3.3]{fangandrollin2015}). The framework of orthogonal projections and the results obtained in Section \ref{sec:distributionH} allow us derive the asymptotic distribution of $\hat f(G_n)$ both under the null (as in Proposition \ref{ppn:C4}) and the alternative (see Proposition \ref{prop:H0alternative} in Section \ref{sec:structureGn}). This will allow us to approximate the asymptotic power of the test based on $T_n$ (recall \eqref{eq:defTnstructure}), and also obtain a confidence interval for $f(W)$ using the method in Section \ref{sec:LHW}.

\small

\subsection*{Acknowledgements} 
B. B. Bhattacharya was supported by NSF CAREER grant DMS 2046393, NSF grant DMS 2113771, and a Sloan Research Fellowship. Anirban Chatterjee was supported by NSF grant 1953848.

\bibliographystyle{abbrvnat}
\bibliography{graphonbib} 

\normalsize

\appendix 

\section{Proof of Theorem \ref{thm:asymp-joint-dist}} 
\label{sec:distributionpf}

For a graphon $W $ and a non-empty simple graph $H= (V(H), E(H))$, the number of copies $X(H, G_n)$ (recall \eqref{eq:XHW}) can be expressed as a generalized $U$-statistic as follows:  
\begin{align}\label{eq:XHWY}
    & X(H, G_n) - \dfrac{(n)_{|V(H)|}}{|\mathrm{Aut}(H)|}t(H,W) \nonumber \\ 
    & = \sum_{1\leq i_{1}<\cdots<i_{|V(H)|}\leq n}f^{(H)}(U_{i_{1}},\cdots,  U_{i_{|V(H)|}},Y_{i_{1}i_{2}} \cdots, Y_{i_{|V(H)|-1}i_{|V(H)|}}) , 
\end{align} 
where  
\begin{align}\label{eq:def_f}
    f^{(H)}(U_{1},\cdots, U_{|V(H)|}, & Y_{12},\cdots, Y_{|V(H)|-1~ |V(H)|}) \nonumber \\ 
    = &\sum_{H'\in \mathscr{G}_{H}}\prod_{(a, b) \in E(H')}\one\left\{Y_{ab}\leq W(U_{a},U_{b})\right\} - \left|\mathscr{G}_{H}\right|t(H, W) 
\end{align} 
and $\mathscr{G}_{H} := \mathscr{G}_{H}(\{1, 2, \ldots, |V(H)|\})$. In this section, using the representation in \eqref{eq:XHWY}, we will derive the joint distribution of 
$$\bm X(\mathcal H, G_n) :=  (X(H_1, G_n), X(H, G_n), \ldots, X(H_r, G_n)) , $$
for a collection non-empty simple graphs $\mathcal H:= \{ H_1, H_2, \ldots, H_r\}$, where $H_{i}=(V(H_{i}), E(H_{i})$ and $V(H_{i})=\{1, 2, \ldots, |V(H_{i})|\}$, for $1 \leq i \leq r$. 

We begin recalling the framework of generalized $U$-statistics developed in \cite{janson1991asymptotic} in the following section.

\subsection{Orthogonal Decomposition of Generalized $U$-Statistics}
\label{sec:generalized_U}

Suppose $\{U_{i}: 1 \leq i \leq n\}$ and $\{Y_{ij}: 1\leq i<j \leq n\}$ are
i.i.d.\ sequences of $U[0,1]$ random variables. Fix $R \geq 1$ and denote by $K_R$ the complete
graph on the set of vertices $\{1, 2, \ldots, R\}$ and let $G= (V(G), E(G))$
be a subgraph of $K_{R}$. Let $\mathcal{F}_{G}$ be the $\sigma$-algebra
generated by the collections $\{U_{i}\}_{i\in V(G)}$ and $\{Y_{ij}\}_{ij\in E(G)}$,
and let $L^{2}(G)=L^2(\mathcal{F}_G)$ 
be the space of all square integrable random
variables that are functions of $\{U_{i}: i\in V(G)\}$ and $\{Y_{ij}:
(i,j)\in E(G) \}$. Now, consider the following subspace of $L^{2}(G)$: 
\begin{align}\label{eq:def-MG}
M_{G}:=\{Z\in L^{2}(G) : \mathbb{E}[ZV]=0\text{ for every }V\in L^{2}(F)\text{ such that } F \subset G\}. 
\end{align}
(For the empty graph, $M_{\emptyset}$ is the space of all constants.) 
Equivalently,
$Z\in M_{G}$ if and only if $Z\in L^{2}(G)$ and 
\begin{align*}
\mathbb{E}\left[Z\mid X_{i},Y_{ij}: i \in V(F), (i, j) \in E(F)\right] = 0,
\quad \text{for all } F \subset G.
\end{align*}
Then, one has the following orthogonal decomposition (see \cite[Lemma 1]{janson1991asymptotic})  
\begin{align}\label{eq:LGH}
L^{2}(G)=\bigoplus_{F \subseteq G} M_F, 
\end{align}
that is, $L^{2}(G)$ is the orthogonal direct sum of $M_{F}$ for all 
subgraphs $F\subseteq G$. This allows us to decompose 
 any function in $L^{2}(G)$ as the sum of  its
projections onto $M_{F}$ for $F \subseteq G$. For any closed subspace $M$ of $L^2(K_R)$, denote the orthogonal projection onto $M$ by $P_M$.

Now, consider a {\it symmetric} function $f$ defined on $L^2(K_R)$, that is, 
\begin{align}\label{eq:def-f-general}
    f & = f(U_{1},U_{2},\cdots, U_{R}, Y_{12}, \cdots, Y_{R-1~R}) \nonumber \\ 
    &= f(U_{\sigma(1)}, U_{\sigma(2)},\cdots, U_{\sigma(R)}, Y_{\sigma(1) \sigma(2)}, \cdots, Y_{\sigma(R-1)~\sigma(R)}) . 
\end{align}
for any permutation $\sigma$ of $\{1, 2, \ldots, R\}$. 

Then $f$ can be decomposed as  

\begin{align}\label{eq:fG}
    f = \sum_{G \subseteq K_{R} } f_{G} , 
\end{align}
where $f_{G}=P_{M_G}f$ is the orthogonal projection of $f$ onto $M_{G}$. 
Further, for $1 \leq s \leq R$, define 
\begin{align}\label{eq:fs}
    f_{(s)}:=\sum_{G \subseteq K_{R}: |V(G)|=s} f_{G} .
\end{align}
The smallest positive $d$ such that $f_{(d)}\neq 0$ almost surely is called the {\it principal degree} of $f$.
It is easy to observe that for any $G\subseteq K_{R}$,
\begin{align}\label{eq:L2G-projection} 
    P_{L^{2}(G)} = \mathbb{E}\left[\cdot\middle| \mathcal{F}_{G}\right].
\end{align}
Moreover, by \eqref{eq:LGH} we have,
\begin{align}\label{eq:PL2-decomp}
    P_{L^{2}(G)} = \sum_{F \subseteq G}P_{M_{F}} . 
\end{align} 

For $f \in L^{2}(K_R)$ define  
\begin{align}\label{eq:def-Sn}
    S_{n,R}(f) = \sum_{1\leq i_{1}<i_{2}<\cdots<i_{R}\leq n}f(U_{i_{1}},U_{i_{2}},\cdots, U_{i_{R}}, Y_{i_{1} i_{2}},\cdots, Y_{i_{R-1}~i_{R}}) , 
\end{align}
and the symmetrized version  
\begin{align}\label{eq:def-Sn-star}
    \tilde{S}_{n,R}(f) = \sum_{ 1\leq i_{1} \ne i_2 \ne i_{R} \leq n }f(U_{i_{1}},U_{i_{2}},\cdots, U_{i_{R}}, Y_{i_{1}i_{2}},\cdots, Y_{i_{R-1}~i_{R}}) , 
\end{align}
where $Y_{ji}:= Y_{ij}$ for $1 \leq i < j \leq n$.
The symmetry of $f$, the decomposition \eqref{eq:fG}, and the linearity of $S_n^\star(\cdot)$ implies that  
\begin{align}\label{eq:SnfG}
S_{n,R}(f) = \frac{1}{R!} \tilde{S}_{n,R}(f) = \frac{1}{R!} \sum_{G \subseteq K_R} \tilde{S}_{n,R}(f_{G}) .  
\end{align}
The symmetry of $f$ also implies that if $G_1$ and $G_2$ are isomorphic subgraphs of $K_R$, then $\tilde{S}_{n,R}(f_{G_1}) = \tilde{S}_{n,R}(f_{G_2})$. Hence, from \eqref{eq:SnfG},  
\begin{align}
S_{n,R}(f) = \sum_{s=0}^R \sum_{G\in\Gamma_{s}}\frac{\tilde{S}_{n,R}(f_{G})}{(R-s)!|\mathrm{Aut}(G)|} , 
\label{eq:SnfGR}
\end{align}
where $\Gamma_s$ is the collection of non-isomorphic graphs with $s$ vertices.

The following result from  \cite{janson1991asymptotic} gives the leading order in the expansion \eqref{eq:SnfGR} for symmetric functions $f$ with principal degree $d$. We include the proof for the sake of completeness:

\begin{prop}[\cite{janson1991asymptotic}]
\label{ppn:Sn}
Suppose $f \in L^{2}(K_R)$ is symmetric and has principal degree $d$. Then 
\begin{align}\label{eq:limit-equiv}
   \mathbb E\left[\left|S_{n,R}(f) - \sum_{G\in\Gamma_{d}}\frac{\tilde{S}_{n,R}(f_{G})}{(R-d)!|\mathrm{Aut}(G)|} \right|^2\right] = O(n^{2R - d -1}) .  
\end{align} 
\end{prop}

\begin{proof} 
Since $f$ has principal degree $d$, by \eqref{eq:SnfGR} and \eqref{eq:fs}, 
\begin{align*}
S_{n,R}(f) = \sum_{s=d}^R \sum_{G\in\Gamma_{s}}\frac{\tilde{S}_{n,R}(f_{G})}{(R-s)!|\mathrm{Aut}(G)|} . 
\end{align*} 
Hence, 
\begin{align}\label{eq:SnfdG}
\mathbb E \left[\left|S_{n,R}(f) - \sum_{G\in\Gamma_{d}}\frac{\tilde{S}_{n,R}(f_{G})}{(R-d)!|\mathrm{Aut}(G)|}  \right|^2\right]  \lesssim_R  \sum_{s=d+1}^R \sum_{G\in\Gamma_{s}}\frac{\mathbb E[|\tilde{S}_{n,R}(f_{G})|^2]}{(R-s)!^2|\mathrm{Aut}(G)|^2} .  
\end{align}
By \cite[Lemma 4]{janson1991asymptotic}, for $G \in \Gamma_s$,  
$$\mathbb E[|\tilde{S}_{n,R}(f_{G})|^2]  = \frac{n! (n- s)!}{(n-R)!^2} \mathbb E[f_G^2] \lesssim_R n^{2R -s} \mathbb E[f^2], $$
where the second inequality uses $\mathbb E[f_G^2] \leq \mathbb E[f^2]$ (recall the orthogonal decomposition \eqref{eq:fG}). 
Applying the above bound in \eqref{eq:SnfdG}, the result in \eqref{eq:limit-equiv} follows. 

\end{proof}

Using the above framework, we now proceed with the proof of Theorem \ref{thm:asymp-joint-dist}. The proof is organized as follows: 
\begin{itemize} 
\item In Section \ref{sec:expansionasymptoticdistribution} we show that the asymptotic distribution of $\bm{Z}(\cH, G_n)$ can be expressed as infinite linear or quadratic forms in i.i.d. Gaussian variables depending on whether $H$ is $W$-irregular or $W$-regular, respectively (see Proposition \ref{ppn:linearTHW}). 
\item In Section \ref{sec:equivalenceasymptoticdistribution} we identify the limit obtained in Proposition \ref{ppn:linearTHW} with limit in Theorem \ref{thm:asymp-joint-dist} (see Proposition \ref{ppn:R} and Proposition \ref{ppn:QR}). 
\end{itemize}

\subsection{Asymptotic Expansion of $\textbf{\textit{Z}}(\cH, G_n)$}
\label{sec:expansionasymptoticdistribution}

Recall that $\mathcal H=\{H_1, H_2, \ldots, H_r\}$ is a collection of non-empty subgraphs such that $W$ is $H_i$-irregular for $1 \leq i \leq q$ and $W$ is $H_i$-regular for $q+1 \leq i \leq r$. Let $f_i:=f^{(H_i)}$ denote the function defined in \eqref{eq:def_f} with $H$ replaced by $H_i$, for $1 \leq i \leq r$. It follows from \cite[Lemma 5.4]{BCJ}, that 
$$\Var[(f_i)_{(1)}]=0 \text{ if and only if } W \text{ is } H_i\text{-regular}.$$ 
Since $\E[(f_i)_{(1)}]=0$, this implies, $(f_i)_{(1)}=0$ almost surely if and only if  $W$ is $H_i$-regular. Hence, the principal degree of $f_i$ is $1$ for $1 \leq i \leq q$ and the principal degree of $f_i$ is $2$ for $q+1 \leq i \leq r$. (Technically, it is possible that the principal degree of $f_i$, for some $q+1 \leq i \leq r$, is greater than 2 (see \cite[Section 4.3]{BCJ} for an example). Here, we will assume that the principal degree of $f_i$ is equal to 2 for all $q+1 \leq i \leq r$, with the understanding that the limit given by Theorem \ref{thm:asymp-joint-dist} in this case can be degenerate if the principal degree is larger.)

Note that, for $1\leq i\leq r$, $f_{i}\in L^{2}\left(K_{|V(H_{i})|}\right)$, is symmetric and has $0$ mean. In particular, 
\begin{align}\label{eq:Snfcount}
    S_{n,|V(H_i)|}(f_{i}) = X(H_{i}, G_n) - \dfrac{(n)_{|V(H_{i})|}}{|\mathrm{Aut}(H_{i})|}t(H_{i},W),
\end{align}
for $1\leq i\leq r$.  
Now, to apply Proposition \ref{ppn:Sn} note that 
\begin{align*}
    \Gamma_{1} = \{K_{\{1\}}\} \text{ and }\Gamma_{2} = \{E_{\{1,2\}},K_{\{1,2\}}\} , 
\end{align*}
where 
\begin{itemize} 

\item $K_{\{1\}}$ is the graph with a single vertex $1$, 

\item $E_{\{1,2\}} = (\{1, 2\}, \emptyset)$ is the graph with two vertices $\{1,2\}$ and no edge between them, 

\item $K_{\{1,2\}} = (\{1, 2\}, \{(1, 2)\})$ is the complete graph on two vertices $\{1,2\}$. 

\end{itemize} 
Then by Proposition \ref{ppn:Sn} and Markov's inequality the following hold: 
\begin{itemize} 

\item For $1 \leq i \leq q$, 
\begin{align}\label{eq:limit-equiv}
   \left|S_{n,|V(H_i)|}(f_i) - \frac{\tilde{S}_{n,|V(H_i)|}((f_i)_{K_{\{1\}}})}{(|V(H_i)|-1)!| } \right| = O_P(n^{|V(H_i)| - 1}) .  
\end{align} 

\item For $q+1 \leq i \leq r$, 
\begin{align}\label{eq:S_n-equiv-1}
    \left|S_{n,|V(H_i)|}(f_{i}) - \frac{\tilde{S}_{n,|V(H_i)|}((f_i)_{E_{\{1,2\}}}) + \tilde{S}_{n,|V(H_i)|}((f_i)_{{K_{\{1,2\}}}})}{2(|V(H_{i})|-2)!} \right| = O_P(n^{|V(H_i)| - \frac{3}{2}}) . 
\end{align} 

\end{itemize} 
Now, define 
\begin{align}\label{eq:TnHiW}
T(H_i, G_n) := 
\left\{
\begin{array}{cc}
  \dfrac{\tilde{S}_{n,|V(H_i)|}((f_i)_{K_{\{1\}}})}{(|V(H_i)|-1)!| n^{|V(H_i)|-\frac{1}{2}}} & \text{ for } 1 \leq i \leq q , \\ \\ 
 \dfrac{ \tilde{S}_{n,|V(H_i)|}((f_i)_{E_{\{1,2\}}}) + \tilde{S}_{n,|V(H_i)|}((f_i)_{{K_{\{1,2\}}}}) }{2(|V(H_i)|-2)!| n^{|V(H_i)|-1}} & \text{ for } q+1 \leq i \leq r.
\end{array}
\right. 
\end{align} 
Then recalling the definition of $Z(H_i, G_n)$ from \eqref{eq:ZnHW} and using from \eqref{eq:Snfcount}, \eqref{eq:limit-equiv}, and \eqref{eq:S_n-equiv-1}, it follows that 
$$Z(H_i, G_n) = T(H_i, G_n) + o_P(1).$$ 
Hence, recalling \eqref{eq:graphW} 
\begin{align}\label{eq:ZHWleadingterm}
\bm Z(\mathcal H, G_n) = \bm T(\mathcal H, G_n) + o_P(1), 
\end{align}
where 
\begin{align}\label{eq:TnHW}
 \bm T(\mathcal H, G_n) = (T(H_1, G_n), T(H_2, G_n), \ldots, T(H_r, G_n))^\top. 
\end{align}  

The result in \eqref{eq:ZHWleadingterm} shows that to obtain the asymptotic joint distribution of $\bm Z(\mathcal H, G_n)$ it suffices to obtain the joint distribution of $ \bm T(\mathcal H, G_n)$. The first step towards this is to compute the projections $(f_i)_{K_{\{1\}}}$,  $(f_i)_{E_{\{1,2\}}}$, and $(f_i)_{K_{\{1,2\}}}$. We begin with a few definitions: 
A function $f\in L^{2}(G)$ is said to be $G$-{\it symmetric} if \eqref{eq:def-f-general} holds whenever $\sigma$ is an automorphism of $G$. 
Also, for two functions $h_{1}, h_{2}$ we define,
\begin{align*}
    h_{1} \otimes h_{2}(x,y) := h_{1}(x) h_{2}(y).
\end{align*}
Now, let $\{\phi_{s}\}_{s \geq 1}$ be a orthonormal basis of $M_{K_{\{1\}}}$
Then $\{\phi_{s}\otimes\phi_{t}\}_{s, t \geq 1}$ is a orthonormal set whose span contains the $E_{\{1,2\}}$-symmetric functions in $M_{E_{\{1,2\}}}$. Also, let $\{\psi_{s}\}_{s \geq 1}$ be a orthonormal basis of the subspace of $K_{\{1,2\}}$-symmetric functions in $M_{K_{\{1,2\}}}$. Then using results in \cite{BCJ,janson1991asymptotic} we the following lemma: 

\begin{lemma}\label{lm:fHiprojection} For $1 \leq i \leq q$, the projection of $f_{(i)}$ on $K_{\{1\}}$ is given by: 
\begin{align}\label{eq:f-decompK1}
    (f_{i})_{K_{\{1\}}} = \mathbb E[f_i|U_1] = \sum_{s \geq 1}\mathbb{E}\left[f_{i} \phi_{s}\right]\phi_{s} . 
\end{align}
Moreover, for  $q+1 \leq i \leq r$, the following hold: 
\begin{itemize}
\item The projection of $f_{(i)}$ on $E_{\{1,2\}}$ is given by: 
\begin{align}\label{eq:f-decompE12}
    (f_{i})_{E_{\{1,2\}}} = \mathbb E[f_i|U_1, U_2] = \sum_{s, t \geq 1}\mathbb{E}\left[f_{i} (\phi_{s}\otimes\phi_{t}) \right]\phi_{s}\otimes\phi_{t}, 
    \end{align} 

\item The projection of $f_{(i)}$ on $K_{\{1,2\}}$ is given by: 
    \begin{align}\label{eq:f-decompK12}
    (f_{i})_{K_{\{1,2\}}} =  \mathbb{E}\left[f_{i}\middle| U_{1},U_{2},Y_{12}\right] - \mathbb{E}\left[f_{i}\middle| U_{1},U_{2}\right] = \sum_{s \geq 1}\mathbb{E}\left[f_{i}\psi_{s}\right]\psi_{s} . 
    \end{align} 
\end{itemize}
\end{lemma}

\begin{proof}
For $1 \leq i \leq q$, by \eqref{eq:L2G-projection} and \eqref{eq:PL2-decomp},
\begin{align*}
(f_{i})_{K_{\{1\}}}:= P_{ M_{K_{\{1\}}}} (f_{i})_{K_{\{1\}}} = P_{L^2(K_{\{1\}})} f_{i} -P_{M_\emptyset} f_{i}
= \mathbb E[f_i|U_1]- \mathbb E[f_i] = \mathbb E[f_i|U_1] .
\end{align*}
This proves the first equality in \eqref{eq:f-decompK1}. To establish the second equality, expand $(f_{i})_{K_{\{1\}}}$ in the basis $\{\phi_{s}\}_{s \geq 1}$ as follows (see \cite[Chapter 6]{laxfunctional}): 
\begin{align}\label{eq:fHiK1}
(f_{i})_{K_{\{1\}}} = \sum_{s \geq 1} \mathbb E[ (f_{i})_{K_{\{1\}}} \phi_s]  \phi_s . \end{align}
By the first equality in \eqref{eq:f-decompK1}, 
$$\mathbb E[ (f_{i})_{K_{\{1\}}} \phi_s] = \mathbb E[ (f_{i})_{K_{\{1\}}}(U_1) \phi_s(U_1) ] =  \mathbb E [ \mathbb E[f_i|U_1] \phi_s (U_1) ] = \mathbb E[f_i \phi].$$ 
Applying the above identity in \eqref{eq:fHiK1} the second equality in \eqref{eq:f-decompK1} follows.

Next, we will prove \eqref{eq:f-decompE12}. By \eqref{eq:L2G-projection} and \eqref{eq:PL2-decomp}, 
\begin{align}\label{eq:fE12asEU}
(f_{i})_{E_{\{1,2\}}}= P_{ M_{E_{\{1,2\}}}} f_{i} &= 
P_{L^2(E_{\{1,2\}})} f_{i} - P_{M_{K_{\{ 1 \}}}} f_{i} -  P_{M_{K_{\{ 2 \}}}} f_{i}-P_{M_\emptyset} f_{i} 
\nonumber \\
&= P_{L^2(E_{\{1,2\}})} f_{i} - P_{L^2(K_{\{ 1 \}})} f_{i} - P_{L^2(K_{\{ 2 \}})} f_{i} +P_{M_\emptyset} f_{i} \nonumber \\ 
& = \mathbb E[f_{i} | U_1, U_2] - \mathbb E[f_{i} | U_1 ] - \mathbb E[f_{i} | U_2] + \mathbb E[f_{i}] \nonumber \\ 
& = \mathbb E[f_{i} | U_1, U_2] , 
\end{align} 
where the last step follows by noting that $\mathbb E[ f_{i} ] = 0$ and $\mathbb E[f_i |U_j] =0$, for $j \in \{1, 2\}$, since $W$ is $H_i$-regular, for $q+1 \leq i \leq r$. This proves the first equality in \eqref{eq:f-decompE12}.  To establish the second equality in \eqref{eq:f-decompE12}, note that 
$$\mathbb E[ (f_{i})_{E_{\{1,2\}}} ( \phi_s \otimes \phi_t )  ] =  \mathbb E [ \mathbb E[f_i|U_1, U_2] \phi_s (U_1) \phi_t(U_2 ] = \mathbb E[f_i (\phi_s \otimes \phi_t) ].$$ 
Hence, 
$$ (f_{i})_{E_{\{1,2\}}} = \sum_{s, t \geq 1} \mathbb E[ (f_{i})_{E_{\{1,2\}}}  (\phi_s \otimes \phi_t) ]  (\phi_s \otimes \phi_t) = \sum_{s, t \geq 1} \mathbb E[ f_{i}  (\phi_s \otimes \phi_t) ]  (\phi_s \otimes \phi_t) .$$

Finally, we prove \eqref{eq:f-decompK12}. For this note that 
\begin{align}\label{eq:fK12asEU}
    (f_{i})_{K_{\{1,2\}}} = P_{M_{K_{\{1,2\}}}}f_{i} 
    & = P_{L^2(K_{\{1,2\}})}f_{i} -P_{M_{E_{\{1,2\}}}}f_{i}- P_{M_{K_{\{1\}}}}f_{i} - P_{M_{K_{\{2\}}}}f_{i} -P_{M_\emptyset}f_{i}\nonumber\\
    & = P_{L^2(K_{\{1,2\}})}f_{i} - P_{L^{2}(E_{\{1,2\}})}f_{i}\nonumber\\
    & = \mathbb{E}\left[f_{i}\middle| U_{1},U_{2},Y_{12}\right] - \mathbb{E}\left[f_{i}\middle| U_{1},U_{2}\right].
\end{align} 
This proves the first equality in \eqref{eq:f-decompK12}. For the second inequality note that $\mathbb E[ \mathbb E[f_{i}|U_1, U_2] \psi_s ] =0$, 
since $\mathbb E[f_{i}|U_1, U_2] \in M_{E_{\{1,2\}}}$ and $M_{E_{\{1,2\}}}$ is orthogonal to $M_{K_{\{1,2\}}}$. This implies, 
$$\mathbb E[ (f_{i})_{K_{\{1,2\}}} \psi_s]  = \mathbb E[ \mathbb{E}\left[f_{i}\middle| U_{1},U_{2},Y_{12}\right] \psi_s] = \mathbb E[ f_{i} \psi_s]$$ and, hence,  
\begin{align}\label{eq:fprojectionK12}
(f_{i})_{K_{\{1,2\}}} = \sum_{s \geq 1} \mathbb E[ (f_{i})_{K_{\{1,2\}}}  \psi_s ]  \psi_s = \sum_{s \geq 1} \mathbb E[ f_{i}  \psi_s ]  \psi_s. 
\end{align} 
This completes the proof of \eqref{eq:f-decompK12}. 
\end{proof}
Lemma \ref{lm:fHiprojection} and \cite[Chapter 6, Lemma 8]{laxfunctional} we now can compute the $L^{2}$ norms of the projections, denoted by $\| \cdot \|_2$, as follows: For $1 \leq i \leq q$,   
\begin{align}\label{eq:finite-expansionK1} 
	\left\|(f_{i})_{K_{\{1\}}}\right\|_{2}^{2} = \sum_{s \geq 1}\mathbb{E}\left[f_{i}\phi_{s}\right]^2 \leq \|f_{i}\|_{2}^{2}  , 
	\end{align} 
	where the last step follows by Bessel's Inequality. Similarly, for $q+1 \leq i \leq r$, 
	\begin{align}\label{eq:finite-expansionK12}
	\left\| (f_{i})_{E_{\{1,2\}}}\right\|_{2}^{2} = \sum_{s, t \geq 1}\mathbb{E}\left[f_{i}(\phi_{s}\otimes\phi_{t}) \right]^2  \leq \|f_{i}\|_{2}^{2} \text{ and }  \left\|(f_{i})_{K_{\{1,2\}}}\right\|_{2}^{2} = \sum_{s \geq 1}\mathbb{E}\left[f_{i}\psi_{s}\right]^2 \leq \|f_{i}\|_{2}^{2} . 
\end{align}

Next, using Lemma \ref{lm:fHiprojection}, the linearity of $\tilde{S}_{n,(\cdot)}$, and a standard truncation argument we obtain the expansions of $\tilde{S}_{n,|V(H_i)|}((f_{i})_{K_{\{1\}}})$, $\tilde{S}_{n,|V(H_i)|}((f_{i})_{E_{\{1,2\}}})$, and $\tilde{S}_{n,|V(H_i)|}((f_{i})_{K_{\{1,2\}}})$. Specifically, for $1\leq i\leq q$, from \eqref{eq:f-decompK1}  we have 
\begin{align}\label{eq:sum-exchange-3}
    \tilde{S}_{n,|V(H_i)|}((f_{i})_{K_{\{1\}}}) = \sum_{s \geq 1}\mathbb{E}\left[f_{i}\phi_{s}\right]\tilde{S}_{n,|V(H_i)|}\left(\phi_{s}\right),
\end{align}
for $1\leq i\leq q$, where the equality hold almost surely.
Similarly, for $q+1 \leq i \leq r$, by \eqref{eq:f-decompE12} and \eqref{eq:f-decompK12} we have the following: 
\begin{align}
    \tilde{S}_{n,|V(H_i)|}((f_{i})_{E_{\{1,2\}}}) 
    & = \sum_{s \geq 1}\mathbb{E}\left[f_{i}(\phi_{s}\otimes\phi_{s})\right]\tilde{S}_{n,|V(H_i)|}\left(\phi_{s}\otimes\phi_{s}\right)\nonumber\\
    & + 2\sum_{s < t}\mathbb{E}\left[f_{i}(\phi_{s}\otimes\phi_{t})\right]\tilde{S}_{n,|V(H_i)|}\left(\phi_{s}\otimes\phi_{t}\right), \label{eq:sum-exchange-2}\\ 
    \tilde{S}_{n,|V(H_i)|}((f_{i})_{K_{\{1,2\}}}) & = \sum_{s \geq 1}\mathbb{E}\left[f_{i}\psi_{s}\right]\tilde{S}_{n,|V(H_i)|}(\psi_{s}). \label{eq:sum-exchange-1}
\end{align}

Using the above expansions we can now derive the asymptotic distribution of $ \bm T(\mathcal H, G_n)$ and hence, that of $\bm Z(\mathcal H, G_n)$.

\begin{prop}\label{ppn:linearTHW}
For  
$ \bm T(\mathcal H, G_n)$ as defined in \eqref{eq:TnHW} and \eqref{eq:TnHiW},  the following hold as $n \rightarrow \infty$: 
\begin{align*} 
 \bm T(\mathcal H, G_n) \dto  \bm T(\mathcal H, W) , 
\end{align*} 
where $\bm T(\mathcal H, W) = (T_1, T_2, \ldots, T_r)^\top$ with 
\begin{align*}
T_i  = & \left\{
\begin{array}{*2{>{\displaystyle}l}}
 \dfrac{1}{(|V(H_{i})|-1)!} \sum_{s \geq 1}\mathbb{E}\left[f_{i}\phi_{s}\right]\eta_{s}  & \hspace{-0.15in} \text{for } 1 \leq i \leq q , \\ 
 \dfrac{ 1 }{2(|V(H_{i})|-2)!} \Bigg\{ \sum_{s \geq 1}\mathbb{E}\left[f_{i}(\phi_{s}\otimes\phi_{s})\right]\left(\eta_{s}^2 - 1\right) + 2\sum_{s < t }\mathbb{E}\left[f_{i} (\phi_{s}\otimes\phi_{t})\right]\eta_{s}\eta_{t} \\ 
 \hspace{3.05in} + \sum_{s \geq 1}\mathbb{E}\left[f_{i}\psi_{s}\right]\tilde \eta_{s} \Bigg\} & \hspace{-0.15in} \text{for } q+1 \leq i \leq r, 
\end{array} 
\right. 
\end{align*}  
and $\{\eta_{s}\}_{s \geq 1}$ and $\{\tilde \eta_{s}\}_{s \geq 1}$ are independent collections of $N(0, 1)$ and $N(0, 2)$ random variables, defined on some probability space $\left(\Omega,\mathcal{F},\mathbb{P}\right)$, respectively. This implies, from \eqref{eq:ZHWleadingterm}, 
\begin{align*} 
 \bm Z(\mathcal H, G_n) \dto  \bm T(\mathcal H, W) . 
\end{align*} 
\end{prop}

\subsubsection{Proof of Proposition \ref{ppn:linearTHW}}

Fix $L \geq 1$ and define the truncated version of $ \tilde{S}_{n,|V(H_i)|}((f_{i})_{K_{\{1\}}})$ (recall \eqref{eq:sum-exchange-3}) as follows: 
\begin{align}
    \tilde{S}_{n,|V(H_i)|}^{(L)}((f_{i})_{K_{\{1\}}}) = \sum_{s = 1}^L \mathbb{E}\left[f_{i}\phi_{s}\right]\tilde{S}_{n,|V(H_i)|}\left(\phi_{s}\right) , \nonumber 
\end{align} 
for $1\leq i\leq q$. Similarly, for $q+1 \leq i \leq r$, recalling \eqref{eq:sum-exchange-2} and \eqref{eq:sum-exchange-1} define the truncated versions: 
\begin{align}
    \tilde{S}_{n,|V(H_i)|}^{(L)}((f_{i})_{E_{\{1,2\}}}) 
    & = \sum_{s = 1}^L \mathbb{E}\left[f_{i}(\phi_{s}\otimes\phi_{s})\right]\tilde{S}_{n,|V(H_i)|}\left(\phi_{s}\otimes\phi_{s}\right)\nonumber\\
    & + 2\sum_{1 \leq s < t \leq L}\mathbb{E}\left[f_{i}(\phi_{s}\otimes\phi_{t})\right]\tilde{S}_{n,|V(H_i)|}\left(\phi_{s}\otimes\phi_{t}\right), \nonumber \\ 
    \tilde{S}_{n,|V(H_i)|}^{(L)}((f_{i})_{K_{\{1,2\}}}) & = \sum_{s = 1}^L \mathbb{E}\left[f_{i}\psi_{s}\right]\tilde{S}_{n,|V(H_i)|}(\psi_{s}) . \nonumber 
\end{align} 
and for $q+1\leq i\leq r$. Now, recalling \eqref{eq:TnHiW}, define the truncated version of $T(H_i, G_n)$ as follows: 
\begin{align}\label{eq:TnLHiW}
T^{(L)}(H_i, G_n) := 
\left\{
\begin{array}{cc}
  \dfrac{\tilde{S}_{n,|V(H_i)|}^{(L)}((f_i)_{K_{\{1\}}})}{(|V(H_i)|-1)!| n^{|V(H_i)|-\frac{1}{2}}} & \text{ for } 1 \leq i \leq q , \\ \\ 
 \dfrac{ \tilde{S}_{n,|V(H_i)|}^{(L)}((f_i)_{E_{\{1,2\}}}) + \tilde{S}_{n,|V(H_i)|}^{(L)}((f_i)_{{K_{\{1,2\}}}}) }{2(|V(H_i)|-2)!| n^{|V(H_i)|-1}} & \text{ for } q+1 \leq i \leq r , 
\end{array}
\right. 
\end{align} 
and 
\begin{align}\label{eq:TnLHW}
\bm T^{(L)}(\mathcal H, G_n) = (T^{(L)}(H_1, G_n), T^{(L)}(H_2, G_n), \ldots, T^{(L)}(H_r, G_n))^\top. 
\end{align}  
The following lemma shows that $\bm T^{(L)}(\mathcal H, G_n)$ converges to a truncated version of $\bm T(\cH, W)$.

\begin{lemma}\label{lem:linearTLHW}
Fix $L \geq 1$ and let $\bm T^{(L)}(\mathcal H, G_n)$ be as defined in \eqref{eq:TnLHW}. Then the following hold as $n \rightarrow \infty$: 
\begin{align*} 
 \bm T^{(L)}(\mathcal H, G_n) \dto  \bm T^{(L)}(\mathcal H, W) , 
\end{align*} 
where $\bm T^{(L)}(\mathcal H, W) = (T_1^{(L)}, T_2^{(L)}, \ldots, T_r^{(L)})^\top$ with 
\begin{align*}
T_i^{(L)}  = & \left\{
\begin{array}{*2{>{\displaystyle}l}}
 \dfrac{1}{(|V(H_{i})|-1)!} \sum_{s =1}^L \mathbb{E}\left[f_{i}\phi_{s}\right]\eta_{s}  & \hspace{-0.65in} \text{for } 1 \leq i \leq q , \\ 
 \dfrac{ 1 }{2(|V(H_{i})|-2)!} \Bigg\{ \sum_{s = 1}^L \mathbb{E}\left[f_{i}(\phi_{s}\otimes\phi_{s})\right]\left(\eta_{s}^2 - 1\right) + 2\sum_{ 1 \leq s < t \leq L }\mathbb{E}\left[f_{i} (\phi_{s}\otimes\phi_{t})\right]\eta_{s}\eta_{t} \\ 
 \hspace{3.05in} + \sum_{s = 1}^L\mathbb{E}\left[f_{i}\psi_{s}\right]\tilde \eta_{s} \Bigg\} & \hspace{-0.65in} \text{for } q+1 \leq i \leq r, 
\end{array} 
\right. 
\end{align*}  
where $\{\eta_{s}\}_{s \geq 1}$ and $\{\tilde \eta_{s}\}_{s \geq 1}$ are independent collections of $N(0, 1)$ and $N(0, 2)$ random variables, respectively. 
\end{lemma} 

\begin{proof}
Recalling the definition of $\tilde S_{n, \cdot}(\cdot)$ from \eqref{eq:def-Sn-star} we get 
\begin{align}\label{eq:SnL}
\frac{1}{n^{|V(H_i)|-\frac{1}{2}}}\tilde S_{n,|V(H_i)|}(\phi_s) = \frac{1+o(1)}{\sqrt n }\sum_{i=1}^n \phi_s(U_i), 
\end{align} 
for $s \geq 1$ and $1 \leq i \leq q$. Similarly, for $s , t \geq 1$ and $q+1 \leq i \leq r$, 
\begin{align}\label{eq:decompSnphiphi}
\frac{1}{n^{|V(H_i)|-1}}\tilde S_{n,|V(H_i)|}(\phi_s \otimes \phi_t) = \frac{1+o(1)}{n}\sum_{1 \leq i \ne j \leq n} \phi_s(U_i) \phi_t(U_j) , 
\end{align} 
and 
\begin{align}\label{eq:decompSnpsi}
\frac{1}{n^{|V(H_i)|-1}}\tilde S_{n,|V(H_i)|}(\psi_s) = \frac{1+o(1)}{n}\sum_{1 \leq i \ne j \leq n} \psi_s(U_i, U_j, Y_{ij}) . 
\end{align} 
Now, by \cite[Lemma 8]{janson1991asymptotic} the collection,
\begin{align*}
    \left\{\left\{\frac{1}{\sqrt{n}}\sum_{i=1}^{n}\phi_s(U_i)\right\}_{s=1}^{L}, \left\{\frac{1}{n}\sum_{1 \leq i \ne j \leq n} \phi_s(U_i) \phi_t(U_j)\right\}_{s=1}^L, \left\{\frac{1}{n}\sum_{1 \leq i \ne j \leq n} \psi_s(U_i, U_j, Y_{ij})\right\}_{s=1}^{L}\right\}
\end{align*}
converges jointly to 
$$\left\{ \left\{ \eta_s \right\}_{1 \leq s \leq L}, \left\{ \eta_s \eta_t - \mathbb E[\eta_s \eta_t] \right\}_{1 \leq s, t \leq L}, \left\{ \tilde\eta_s \right\}_{1 \leq s \leq L} \right\} . $$
The result in Lemma \ref{lem:linearTLHW} then follows by recalling the definition of  $\bm T^{(L)}(\mathcal H, G_n)$ from \eqref{eq:TnLHW} and \eqref{eq:TnLHiW} and the decompositions from \eqref{eq:SnL}, \eqref{eq:decompSnphiphi} and \eqref{eq:decompSnpsi}.

\end{proof}

Now, to complete the proof of Proposition \ref{ppn:linearTHW} it suffices to show the following: 
\begin{itemize} 
\item[$(1)$] $\bm T^{(L)}(\mathcal H, G_n)$ and $ \bm T(\mathcal H, G_n)$ are asymptotically close and

\item[$(2)$] $\bm T^{(L)}(\mathcal H, W)$ converges to $\bm T(\mathcal H, W)$, as $L \rightarrow \infty$. 
\end{itemize}
These are established in the following 2 lemmas, respectively. 
\begin{lemma}\label{lemma:unif-n-convg}
Let $\bm T(\cH, G_n)$ and $\bm T^{(L)}(\cH, G_n)$ be defined in \eqref{eq:TnHW} and \eqref{eq:TnLHW}, respectively. Then 
\begin{align*}
 \lim_{L \rightarrow \infty}  \sup_{n \geq 1} \mathbb E \left[ \left\| \bm T(\cH, G_n) - \bm T^{(L)}(\cH, G_n) \right\|_2^2 \right] = 0. 
\end{align*}
\end{lemma}
\begin{proof}
Note that for $1 \leq i \leq q$, 
\begin{align}\label{eq:TnL1pf}
\mathbb E \left[ | \bm T(H_i, G_n) - \bm T^{(L)}(H_i, G_n) |^2 \right]  
 & =   \frac{1}{n^{2 |V(H_i)| -2}} \mathbb E \left[  \left |\sum_{s = L+1}^N \mathbb{E}\left[f_{i}\phi_{s}\right]\tilde{S}_{n,|V(H_i)|}\left(\phi_{s}\right) \right |^2 \right] . 
\end{align} 
By the orthogonality of the basis $\{ \phi_s\}_{s \geq 1}$, it is easy to verify that 
$\mathbb E[\tilde{S}_{n,\cdot}(\phi_s) \tilde{S}_{n,\cdot}(\phi_t)] = 0$, for $s \ne t$. Moreover, from \eqref{eq:SnL}, it follows that $\frac{1}{n^{2|V(H_i)| - 2 }}\mathbb E[\tilde S_{n,|V(H_i)|}(\phi_s)^2] = 1 + o(1)$, for $s \geq 1$. Hence, from \eqref{eq:TnL1pf} and \eqref{eq:finite-expansionK1}, 
\begin{align}\label{eq:TnL1}
\lim_{L \rightarrow \infty} \sup_{n \geq 1}\mathbb E \left[ | \bm T(H_i, G_n) - \bm T^{(L)}(H_i, G_n) |^2 \right]   & \lesssim_{H_i} \lim_{L \rightarrow \infty} \sum_{m=L+1}^{\infty}\mathbb{E}\left[f_{i}\phi_{s}\right]^2 = 0. 
\end{align}
Similarly, for $q+1 \leq i \leq r$ it can be shown that 
\begin{align}\label{eq:TnL2}
& \lim_{L \rightarrow \infty} \sup_{n \geq 1}\mathbb E \left[ | \bm T(H_i, G_n) - \bm T^{(L)}(H_i, G_n) |^2 \right]   \nonumber \\ 
& \lesssim_{H_i} \lim_{L \rightarrow \infty}  
\left\{ \sum_{s=L+1}^{\infty}\mathbb{E}\left[f_{i}(\phi_{s}\otimes\phi_{s})\right]^2
+  \sum_{t=L+1}^\infty \sum_{1 \leq s<t} \mathbb{E}\left[f_{i}(\phi_{s}\otimes\phi_{t})\right]^2 
+ \sum_{s=L+1}^{\infty}\mathbb{E}\left[f_{i}\psi_{s}\right]^2 \right\} \nonumber \\ 
& = 0 ,  
\end{align} 
by \eqref{eq:finite-expansionK12}. 

Combining \eqref{eq:TnL1} and \eqref{eq:TnL2} the proof of Lemma \ref{lemma:unif-n-convg} follows.  
\end{proof} 

\begin{lemma}\label{lemma:Z_N-convg} 
Let $\bm T(\cH, W) = (T_1, T_2, \ldots, T_r)^\top$ and $\bm T^{(L)}(\cH, W) =  (T_1^{(L)}, T_2^{(L)}, \ldots, T_r^{(L)})^\top$ be as defined in Proposition \ref{ppn:linearTHW} and Lemma \ref{lem:linearTLHW}, respectively. Then 
\begin{align*}
    \lim_{L \rightarrow \infty} \mathbb E \left[ \| \bm T^{(L)}(\cH, W) - \bm T(\cH, W) \|_2^2 \right] = 0. 
\end{align*}
\end{lemma} 

\begin{proof} 

For $1\leq i\leq q$, by \cite[Lemma 8]{laxfunctional} we get,
\begin{align}\label{eq:LK1}
\mathbb E |T_i^{(L)} - T_i| \lesssim_{H_i} \mathbb E \left|\sum_{s = L+1}^{\infty}\mathbb{E}\left[f_{i}\psi_{s}\right] \eta_{s}\right|^{2} \leq \sum_{s=L+1}^\infty \mathbb{E}\left[f_{i}\psi_{s}\right]^2 \rightarrow 0 , 
\end{align}
as $L \rightarrow \infty$, by \eqref{eq:finite-expansionK1}. Similarly, for $q+1\leq i\leq r$,
\begin{align}\label{eq:LK12}
    \mathbb E \left|\sum_{s = L+1}^\infty \mathbb{E}\left[f_{i}\phi_{s}\right] \tilde{\eta}_{s} \right|^{2} \rightarrow 0 , 
\end{align} 
as $L \rightarrow \infty$, by \eqref{eq:finite-expansionK12}. Also,
since $\{\eta_{s}^2-1: s\geq 1\}$ are orthogonal and $\mathbb{E}(\eta_{s}^2-1)^2=2$,
\begin{align}\label{eq:LE12s}
  \mathbb E \left|\sum_{ s = L+1}^\infty \mathbb{E}\left[f_{i}(\phi_{s}\otimes\phi_{s})\right]\left(\eta_{s}^2 - 1\right)\right|^2 \leq  2\sum_{s= L+1}^\infty \mathbb{E}\left[f_{i}(\phi_{s}\otimes\phi_{s})\right]^2 \rightarrow 0 . 
\end{align} 
Once again by definition $\{\eta_{s}\eta_{t}: s < t\}$ are orthogonal and $\mathbb{E}\eta_{s}^2\eta_{t}^2=1$. Hence, as above, 
\begin{align}\label{eq:LE12st}
    \mathbb E \left| \sum_{t=1}^\infty \sum_{1 \leq s < t }\mathbb{E}\left[f_{i}(\phi_{s}\otimes\phi_{t})\right]\eta_{s}\eta_{t}  -  \sum_{t = L+1}^\infty \sum_{1 \leq s < t} \mathbb{E}\left[f_{i}(\phi_{s}\otimes\phi_{t})\right]\eta_{s}\eta_{t} \right| \rightarrow 0 , 
\end{align} 
as $L \rightarrow \infty$. Combining \eqref{eq:LK12}, \eqref{eq:LE12s}, and \eqref{eq:LE12st}, we get $\mathbb E |T_i^{(L)} - T_i|^2 \rightarrow 0$, as $L \rightarrow \infty$, for $q+1 \leq i \leq r$. This together with \eqref{eq:LK1} completes the proof of Lemma \ref{lemma:Z_N-convg}.
\end{proof}

Combining Lemma \ref{lem:linearTLHW}, Lemma \ref{lemma:unif-n-convg}, and Lemma \ref{lemma:Z_N-convg} along with 
\cite[Lemma 6]{janson1991asymptotic} the result in Proposition \ref{ppn:linearTHW} follows. \hfill $\Box$

\subsection{Equivalence of \texorpdfstring{$\bm{T}$}{T}$(\mathcal{H}, W)$ and \texorpdfstring{$\bm{Z}$}{Z}$(\mathcal{H}, W)$} 
\label{sec:equivalenceasymptoticdistribution}

For $1 \leq i \leq q$, define  
\begin{align}\label{eq:TQ}
Q_i  = \dfrac{1}{(|V(H_{i})|-1)!} \sum_{s \geq 1}\mathbb{E}\left[f_{i}\phi_{s}\right]\eta_{s} . 
 \end{align}
Also, for $q+1 \leq i \leq r$, define 
\begin{align}\label{eq:TRE12}
R_i & = \dfrac{ 1 }{2(|V(H_{i})|-2)!} \left\{ \sum_{s \geq 1}\mathbb{E}\left[f_{i}(\phi_{s}\otimes\phi_{s})\right]\left(\eta_{s}^2 - 1\right) + 2\sum_{s < t }\mathbb{E}\left[f_{i} (\phi_{s}\otimes\phi_{t})\right]\eta_{s}\eta_{t} \right\} ,  \\ 
\tilde R_i & = \dfrac{ 1 }{2(|V(H_{i})|-2)!} \sum_{s \geq 1}\mathbb{E}\left[f_{i}\psi_{s}\right]\tilde \eta_{s} , \label{eq:TRK12} 
\end{align}  
where $\{\eta_{s}\}_{s \geq 1}$ and $\{\tilde \eta_{s}\}_{s \geq 1}$ are independent collections of $N(0, 1)$ and $N(0, 2)$ random variables, respectively. Then recalling the definition of $T_i$ from Proposition \ref{ppn:linearTHW}, note that  
\begin{align*}
T_i  =  \left\{
\begin{array}{cc} 
Q_i  & \text{for } 1 \leq i \leq q , \\ 
R_i + \tilde R_i & \text{for } q+1 \leq i \leq r . 
\end{array} 
\right. 
\end{align*}  
Denote $\bm Q = (Q_1, Q_2, \ldots, Q_q)^\top$, $\bm R = (R_{q+1}, R_{q+2}, \ldots, R_r)^\top$ and $\tilde{\bm R} = (\tilde R_{q+1}, \tilde R_{q+2}, \ldots, \tilde R_r)^\top$.
In the following 2 propositions we identify $\bm Q$, $\bm R$, $\tilde{\bm R}$ with their corresponding components in $\bm Z(\cH, W)$ (as defined in Theorem \ref{thm:asymp-joint-dist}). 

\begin{prop}\label{ppn:R}
Let $\tilde {\bm R} = (\tilde R_{q+1}, \tilde R_{q+2}, \ldots, \tilde R_r)^\top$ be as defined in \eqref{eq:TRK12}. Then the following hold: 
\begin{align*}
  \tilde{\bm R} \sim N_{r-q}\left(\bm{0},\Sigma\right) , 
\end{align*}
where $\Sigma$ is as defined in Theorem \ref{thm:asymp-joint-dist}.
\end{prop}

\begin{prop}\label{ppn:QR} Let $\bm Q = (Q_1, Q_2, \ldots, Q_q)^\top$ and ${\bm R} = (R_{q+1}, R_{q+2}, \ldots, R_r)^\top$ be as defined in \eqref{eq:TQ} and \eqref{eq:TRK12}, respectively. Suppose $\{B_{t}\}_{t\in [0,1]}$ be the standard Brownian motion in $[0, 1]$. Then the following hold: 
\begin{align}\label{eq:QBx}
\bm Q \stackrel{D} = \left( \dfrac{1}{|\mathrm{Aut}(H_{i})|}  \int_{0}^{1} \left\{ \sum_{a=1}^{|V(H_{i})|} \left( t_{a}(x,H_{i},W) - t(H_{i},W) \right) \right\} \mathrm d B_{x} \right)_{1 \leq i \leq q}, 
\end{align} 
and 
\begin{align}\label{eq:RBx}
\bm R \stackrel{D} = \left(\int_{0}^{1}\int_{0}^{1} \left\{ W_{H_{i}}(x,y) - \dfrac{|V(H_{i})|\left(|V(H_{i})-1|\right)}{2|\mathrm{Aut}(H_{i})|}t(H_{i},W) \right \} \mathrm d B_{x}\mathrm d B_{y} \right)_{ q+1 \leq i \leq r} . 
\end{align}
\end{prop}

The proofs of Proposition \ref{ppn:R} and Proposition \ref{ppn:QR} are given in Section \ref{sec:Rpf} and Section \ref{sec:QRpf}, respectively.  These 2 results combined establishes the equivalence of $\bm{T}(\mathcal{H}, W)$ and $\bm{Z}(\mathcal{H}, W)$.

\subsubsection{Proof of Proposition \ref{ppn:R}}
\label{sec:Rpf}

Fix $L \geq 1$ and define the truncated version of $\tilde R_i$ as follows: 
\begin{align*}
\tilde R_i^{(L)}  = \dfrac{1}{2(|V(H_{i})|-2)!} \sum_{s = 1}^L\mathbb{E}\left[f_{i}\psi_{s}\right]\tilde \eta_{s} , 
 \end{align*} 
 for $q+1 \leq i \leq r$. Denote $\tilde{\bm R}^{(L)} = (\tilde R_{q+1}^{(L)}, \tilde R_{q+2}^{(L)}, \ldots, \tilde R_r^{(L)})^\top$. Then from \eqref{eq:finite-expansionK12} it is follows that 
\begin{align}\label{eq:L2norm-vector-convg}
  \lim_{L \rightarrow \infty}  \mathbb E \left\| \tilde{\bm R} - \tilde{\bm R}^{(L)} \right\|_{2}^{2} = 0 . 
\end{align} 
Since the collection $\{\eta_s\}_{s \geq 1}$ are independent $N(0, 2)$, 
$$\tilde{\bm R}^{(L)} \stackrel{D} = N_{r-q}(0, \Gamma^{(L)}) , $$ 
where $\Gamma^{(L)} = ((\gamma_{ij}^{(L)}))_{q+1 \leq i, j \leq r}$ is given by 
$$
\gamma_{ij}^{(L)} = 
\left\{
\begin{array}{cc} 
\dfrac{1}{2(|V(H_{i})|-2)!^2} \sum_{s = 1}^L\mathbb{E}\left[f_{i}\psi_{s}\right]^2 & \text{for } q+1 \leq i = j \leq r , \\ 
\dfrac{1}{2(|V(H_{i})|-2)! (|V(H_{j})|-2)!} \sum_{s = 1}^L\mathbb{E}\left[f_{i}\psi_{s}\right] \mathbb{E}\left[f_{j}\psi_{s}\right] & \text{for } q+1 \leq i \ne j \leq r. 
\end{array} 
\right. 
$$
By computing characteristic functions and recalling \eqref{eq:finite-expansionK12} one has $\tilde{\bm R}^{(L)} \dto N_{r-q}(0, \Gamma)$, as $L \rightarrow \infty$, where 
$\Gamma = ((\gamma_{ij}))_{q+1 \leq i, j \leq r}$ is given by 
$$
\gamma_{ij} = 
\left\{
\begin{array}{cc} 
\dfrac{1}{2(|V(H_{i})|-2)!^2}  \mathbb E[ (f_i)_{K_{\{1, 2\} }}^2 ]  & \text{for } q+1 \leq i = j \leq r , \\ 
\dfrac{1}{2(|V(H_{i})|-2)! (|V(H_{j})|-2)!} \mathbb E[ (f_i)_{K_{\{1, 2\} }} (f_j)_{K_{\{1, 2\} }} ]  & \text{for } q+1 \leq i \ne j \leq r. 
\end{array} 
\right. 
$$
Here, we use the identity 
$\sum_{s = 1}^\infty \mathbb{E}\left[f_{i}\psi_{s}\right]^2 = \mathbb E[ (f_i)_{K_{\{1, 2\}}}^2]$ (recall \eqref{eq:finite-expansionK12}) and 
$$\sum_{s = 1}^\infty \mathbb{E}\left[f_{i}\psi_{s}\right] \mathbb{E}\left[f_{j}\psi_{s}\right] = \mathbb E[ (f_i)_{K_{\{1, 2\} }} (f_j)_{K_{\{1, 2\} }} ] , $$ 
which follows from the expansion \eqref{eq:fprojectionK12} and the orthogonality of the functions $\{ \psi_s \}_{s \geq 1}$.

The proof of Proposition now follows from Lemma \ref{lemma:covariance-structure}, which shows that the matrix $\Gamma$ is same as the matrix $\Sigma$ in Theorem \ref{thm:asymp-joint-dist}. \hfill $\Box$ 

\begin{lemma}\label{lemma:covariance-structure}
For all $q+1 \leq i,j\leq r$,
\begin{align*}
    \mathbb{E}\left[(f_{i})_{K_{\{1,2\}}}f_{j,K_{\{1,2\}}}\right] 
    =& \dfrac{(|V(H_{i})|-2)!(|V(H_{j})|-2)!}{\left|\mathrm{Aut}(H_{i})\right|\left|\mathrm{Aut}(H_{j})\right|}\\
    &\sum_{\substack{(a,b)\in E^{+}(H_{i})\\(c,d)\in E^{+}(H_{j})}}\left(t\left(H_{i}\bigominus_{(a,b),(c,d)}H_{j},W\right) - t\left(H_{i}\bigoplus_{(a,b),(c,d)}H_{j},W\right)\right) . 
\end{align*}
\end{lemma}

\begin{proof} 
For $q+1 \leq i \leq r$, from Lemma \ref{lm:fHiprojection} we have 
\begin{align*}
    (f_{i})_{K_{\{1,2\}}} & = \mathbb{E}\left[f_{i}\middle| U_{1},U_{2},Y_{12}\right] - \mathbb{E}\left[f_{i}\middle| U_{1},U_{2}\right] \nonumber \\ 
    & = \sum_{H'\in\mathscr{G}_{H_{i},\{1,2\}}} t^-_{1,2}(U_{1},U_2,H',W) , 
\left(\one\{Y_{12}\leq W(U_{1},U_{2})\} -W(U_1, U_2)\right).
\end{align*}
where $\mathscr{G}_{H_{i},\{1,2\}} := \left\{H'\in \mathscr{G}_{H_{i}}:(1,2)\in E(H')\right\}$ for all $q+1\leq i\leq r$ and,
\begin{align*}
    t_{1,2}^{-}\left(U_1, U_2, H', W\right) = \E\left[\prod_{(i,j)\in E(H')\setminus\{(1,2)\}}W(U_i, U_j)\middle|U_1, U_2\right]\text{ for all }H'\in\mathscr{G}_{H_{i},\{1,2\}}.
\end{align*}
Then for $q+1 \leq i,j\leq r$, 
\begin{align}\label{eq:Cov-fifj-1}
    \mathbb{E}
    &\left[(f_{i})_{K_{\{1,2\}}}f_{j,K_{\{1,2\}}}\right]\nonumber\\
    & = \sum_{\substack{H_{1}\in\mathscr{G}_{H_{i},\{1,2\}}\\ H_{2}\in\mathscr{G}_{H_{j,\{1,2\}}}}}\mathbb{E}\bigg[t^-_{1,2}(U_{1},U_2,H_{1},W)t^-_{1,2}(U_{1},U_2,H_{2},W)\left(\one\{Y_{12}\leq W(U_{1},U_{2})\} -W(U_1, U_2)\right)^2\bigg]\nonumber\\
    & = \sum_{\substack{H_{1}\in\mathscr{G}_{H_{i},\{1,2\}}\\ H_{2}\in\mathscr{G}_{H_{j,\{1,2\}}}}}\mathbb{E}\bigg[t^-_{1,2}(U_{1},U_2,H_{1},W)t^-_{1,2}(U_{1},U_2,H_{2},W)W(U_{1},U_{2})(1-W(U_{1},U_{2}))\bigg]\nonumber\\
    & = \sum_{\substack{H_{1}\in\mathscr{G}_{H_{i},\{1,2\}}\\ H_{2}\in\mathscr{G}_{H_{j,\{1,2\}}}}}\left(t\left(H_{1}\bigominus_{(1,2),(1,2)}H_{2},W\right) - t\left(H_{1}\bigoplus_{(1,2),(1,2)}H_{2},W\right)\right) , 
\end{align}
recalling the join operations from Definition \ref{defn:H1H2ab}. Now, considering $V_{H_{\ell}}^{2} = \left\{(a,b)\in V(H_{\ell}):a\neq b\right\}$ for $\ell \in \{i,j\}$ define,
\begin{align*}
    \underline{t}\left(H_{1}\bigominus_{(a,b),(c,d)} H_{2},W\right) =  
    t\left(H_{1}\bigominus_{(a,b),(c,d)} H_{2},W\right) \bm 1\{(a,b) \in E^{+}(H_{1}) \text{ and }(c,d) \in E^{+}(H_{2})\}
 \end{align*} 
and similarly, 
\begin{align*}
\underline{t}\left(H_{1}\bigoplus_{(a,b),(c,d)} H_{2},W\right) = 
t\left(H_{1}\bigoplus_{(a,b),(c,d)} H_{2},W \right) \bm 1\{(a,b) \in E^{+}(H_{1}) \text{ and }(c,d) \in E^{+}(H_{2})\} .
\end{align*}
where $(a,b)\in V_{H_{i}}^{2}, (c,d)\in V_{H_{j}}^{2}$ and $H_{1}\in \mathscr{G}_{H_{i}}$ and $H_{2}\in \mathscr{G}_{H_{j}}$. Then we can rewrite \eqref{eq:Cov-fifj-1} as,
\begin{align}\label{eq:Cov-lowbar}
    \mathbb{E}
    &\left[(f_{i})_{K_{\{1,2\}}}f_{j,K_{\{1,2\}}}\right] = \sum_{\substack{H_{1}\in\mathscr{G}_{H_{i}}\\ H_{2}\in\mathscr{G}_{H_{j}}}}\left(\underline{t}\left(H_{1}\bigominus_{(1,2),(1,2)}H_{2},W\right) - \underline{t}\left(H_{1}\bigoplus_{(1,2),(1,2)}H_{2},W\right)\right)
\end{align}
Now consider the permutations $\pi_{(a, b)}:V(H_{i})\rightarrow V(H_{i})$ and  $\pi'_{(c, d)}:V(H_{j})\rightarrow V(H_{j})$ such that $\pi_{(a, b)}(a) = \pi'_{(c, d)}(c) = 1$ and $\pi_{(a, b)}(b) = \pi'_{(c, d)}(d) = 2$. Then

\begin{align}\label{eq:t-lowbar-1}
    \sum_{\substack{(a,b)\in V_{H_{i}}^{2}\\(c,d)\in V_{H_{j}}^{2}}}
    &\sum_{\substack{H_{1}\in\mathscr{G}_{H_{i}}\\ H_{2}\in\mathscr{G}_{H_{j,}}}}\underline{t}\left(H_{1}\bigominus_{(a,b),(c,d)}H_{2},W\right) \nonumber\\
    &= \sum_{\substack{(a,b)\in V_{H_{i}}^{2}\\(c,d)\in V_{H_{j}}^{2}}}\sum_{\substack{H_{1}\in\mathscr{G}_{H_{i}}\\ H_{2}\in\mathscr{G}_{H_{j,}}}}\underline{t}\left(\pi_{(a, b)}(H_{1})\bigominus_{(1,2),(1,2)}\pi'_{(c, d)}(H_{2}),W\right)\nonumber\\
    &= |V_{H_{i}}^{2}||V_{H_{j}}^{2}|\sum_{\substack{H_{1}\in\mathscr{G}_{H_{i}}\\ H_{2}\in\mathscr{G}_{H_{j,}}}}\underline{t}\left(H_{1}\bigominus_{(1,2),(1,2)}H_{2},W\right)
\end{align}
where the last equality follows by observing that 
\begin{align*}
    \left(H_{1},H_{2}\right)\rightarrow\left(\pi_{(a, b)}(H_{1}),\pi'_{(c, d)}(H_{2})\right)
\end{align*}
is a bijection from $\mathscr{G}_{H_{i}}\times \mathscr{G}_{H_{j}}$ to itself for all $(a,b)\in V_{H_{i}}^{2}$ and $(c,d)\in V_{H_{j}}^{2}$. By considering isomorphisms $\tau_{H_{1}}$ and $\tau'_{H_{1}}$ for $H_{1}\in\mathscr{G}_{H_{i}}$ and $H_{2}\in\mathscr{G}_{H_{j}}$ such that $\tau_{H_{1}}(H_{1}) = H_{i}$ and $\tau'_{H_{1}}(H_{2}) = H_{j}$ a similar argument as above shows that,
\begin{align}\label{eq:t-lowbar-2}
    \sum_{\substack{(a,b)\in V_{H_{i}}^{2}\\(c,d)\in V_{H_{j}}^{2}}}
    \sum_{\substack{H_{1}\in\mathscr{G}_{H_{i}}\\ H_{2}\in\mathscr{G}_{H_{j}}}}\underline{t}\left(H_{1}\bigominus_{(a,b),(c,d)}H_{2},W\right)
    & = |\mathscr{G}_{H_{i}}||\mathscr{G}_{H_{j}}|\sum_{\substack{(a,b)\in V_{H_{i}}^{2}\\(c,d)\in V_{H_{j}}^{2}}}\underline{t}\left(H_{i}\bigominus_{(a,b),(c,d)}H_{j},W\right) . 
\end{align}
Thus, combining \eqref{eq:t-lowbar-1} and \eqref{eq:t-lowbar-2} we find,
\begin{align}\label{eq:t-lowbar-minus}
    \sum_{\substack{H_{1}\in\mathscr{G}_{H_{i}}\\ H_{2}\in\mathscr{G}_{H_{j,}}}}\underline{t}\left(H_{1}\bigominus_{(1,2),(1,2)}H_{2},W\right) = \dfrac{|\mathscr{G}_{H_{i}}||\mathscr{G}_{H_{j}}|}{|V_{H_{i}}^{2}||V_{H_{j}}^{2}|}\sum_{\substack{(a,b)\in V_{H_{i}}^{2}\\(c,d)\in V_{H_{j}}^{2}}}\underline{t}\left(H_{i}\bigominus_{(a,b),(c,d)}H_{j},W\right) . 
\end{align}
Similarly, 
\begin{align}\label{eq:t-lowbar-plus}
    \sum_{\substack{H_{1}\in\mathscr{G}_{H_{i}}\\ H_{2}\in\mathscr{G}_{H_{j,}}}}\underline{t}\left(H_{1}\bigoplus_{(1,2),(1,2)}H_{2},W\right) = \dfrac{|\mathscr{G}_{H_{i}}||\mathscr{G}_{H_{j}}|}{|V_{H_{i}}^{2}||V_{H_{j}}^{2}|}\sum_{\substack{(a,b)\in V_{H_{i}}^{2}\\(c,d)\in V_{H_{j}}^{2}}}\underline{t}\left(H_{i}\bigoplus_{(a,b),(c,d)}H_{j},W\right) . 
\end{align}
Notice that by definition,
\begin{align}\label{eq:sum-V2=E+}
    \sum_{\substack{(a,b)\in V_{H_{i}}^{2}\\(c,d)\in V_{H_{j}}^{2}}}
    &\Bigg(\underline{t}\left(H_{i}\bigominus_{(a,b),(c,d)}H_{j},W\right)
    -\underline{t}\left(H_{i}\bigoplus_{(a,b),(c,d)}H_{j},W\right)\Bigg)\nonumber\\
    =& \sum_{\substack{(a,b)\in E^{+}(H_{i})\\(c,d)\in E^{+}(H_{j})}}\left(t\left(H_{i}\bigominus_{(a,b),(c,d)}H_{j},W\right)
    -t\left(H_{i}\bigoplus_{(a,b),(c,d)}H_{j},W\right)\right) . 
\end{align}
Recall that $|\mathscr{G}_{H_{\ell}}| = \frac{|V(H_{\ell})|!}{|\mathrm{Aut}(H_{\ell})|}$, for $\ell \in \{ i,j\}$.
Then observing that $|V_{H_{\ell}}^{2}| = |V(H_{\ell})|(|V(H_{\ell})|-1)$, for $\ell \in \{i,j\}$, and using \eqref{eq:t-lowbar-minus}, \eqref{eq:t-lowbar-plus} in combination with \eqref{eq:Cov-lowbar} and \eqref{eq:sum-V2=E+} completes the proof of Lemma \ref{lemma:covariance-structure}. 
\end{proof}

\subsubsection{Proof of Proposition \ref{ppn:QR}} 
\label{sec:QRpf} 

For $1 \leq i \leq q$, using the expansion of $(f_i)_{K_{\{ 1 \}}}$ in \eqref{eq:f-decompK1} and Proposition \ref{prop:I2-f_E} it follows that 
\begin{align}\label{eq:I1K1}
    I_1((f_{i})_{K_{\{1\}}}) \stackrel{a.s.} = \sum_{s \geq 1}\mathbb{E}\left[f_{i} \phi_{s}\right]I_1(\phi_{s}) ,  
\end{align}
where $I_1(\cdot)$ is the 1-dimensional stochastic integral as defined in Section \ref{sec:stochasticintegral}. Note that $\{I_1(\phi_s)\}_{s \geq 1}$ is a collection of independent $N(0, 1)$ random variables.  Hence,  
$$ \frac{1}{(|V(H_i)| - 1)!} I_1((f_{i})_{K_{\{1\}}}) \stackrel{D} = \frac{1}{(|V(H_i)| - 1)!}  \sum_{s \geq 1}\mathbb{E}\left[f_{i} \phi_{s}\right] \eta_s = Q_i ,$$
for $Q_i$ as defined in \eqref{eq:TQ}. 
Now, recalling \eqref{eq:f-decompK1} and Definition \ref{defn:tabxyHW} note that 
\begin{align}\label{eq:f_K1-non-reg}
    (f_{i})_{K_{\{1\}}}(x)
    &= \mathbb{E}\left[f_{i}|U_{1} = x\right] \nonumber\\
    &= \sum_{H'\in \mathscr{G}_{H_{i}}}t_{1}(x,H',W) - |\mathscr{G}_{H_{i}}|t(H_{i},W) \nonumber\\
    &= \dfrac{|\mathscr{G}_{H_{i}}|}{|V(H_{i})|}\sum_{a\in V(H_{i})}\left(t_{a}\left(x,H_{i},W\right) - t(H_{i},W)\right) , 
\end{align}
where the last equality follows by arguments similar to proof of \eqref{eq:t-lowbar-minus}. Hence, using \eqref{eq:f_K1-non-reg} in \eqref{eq:I1K1} and recalling that $|\mathscr{G}_{H_{i}}| = \frac{|V(H_{i})|!}{|\mathrm{Aut}(H_{i})|}$  gives, 
$$Q_i \stackrel{D} = \dfrac{1}{|\mathrm{Aut}(H_{i})|}  \int_{0}^{1} \left\{ \sum_{a=1}^{|V(H_{i})|} \left( t_{a}(x,H_{i},W) - t(H_{i},W) \right) \right\} \mathrm d B_{x} ,$$
for $1 \leq i \leq q$. This shows \eqref{eq:QBx}.

Now, suppose $q+1 \leq i \leq r$. Then from the expansion of $(f_i)_{E_{\{1, 2 \}}}$ in \eqref{eq:f-decompE12} and Proposition \ref{prop:I2-f_E},  
\begin{align*}
    I_{2}\left((f_{i})_{E_{\{1,2\}}}\right) & \stackrel{a.s.} = \sum_{s \geq 1}\mathbb{E}\left[f_{i}(\phi_{s}\otimes\phi_{s})\right] I_2(\phi_s \otimes \phi_s) + 2\sum_{s < t }\mathbb{E}\left[f_{i}(\phi_{s}\otimes\phi_{t})\right] I_2(\phi_s \otimes \phi_t) \nonumber \\ 
 & = \sum_{s \geq 1}\mathbb{E}\left[f_{i}(\phi_{s}\otimes\phi_{s})\right] \left( I_1(\phi_s)^2 -1 \right)  + 2\sum_{s < t }\mathbb{E}\left[f_{i}(\phi_{s}\otimes\phi_{t})\right] I_1(\phi_s) I_1(\phi_t) , 
    \end{align*} 
by \eqref{eq:fgstochasticintegral}, since $\mathbb E[I_1(\phi_s)] = 0$ and $\mathbb E[I_1(\phi_s)^2] = 1$. As before, noting that $\{I_1(\phi_s)\}_{s \geq 1}$ is a collection of independent $N(0, 1)$ random variables and recalling \eqref{eq:TRE12} gives, 
\begin{align}\label{eq:I2E12}
     & \dfrac{ 1 }{2(|V(H_{i})|-2)!}  I_{2}\left((f_{i})_{E_{\{1,2\}}}\right) \nonumber \\ 
     & \stackrel{D}= \dfrac{ 1 }{2(|V(H_{i})|-2)!} \left\{ \sum_{s \geq 1}\mathbb{E}\left[f_{i}(\phi_{s}\otimes\phi_{s})\right]\left(\eta_{s}^2 - 1\right) + 2\sum_{s < t }\mathbb{E}\left[f_{i} (\phi_{s}\otimes\phi_{t})\right]\eta_{s}\eta_{t} \right\} \nonumber \\ 
     & = R_i . 
\end{align}
 Recalling Definition \ref{defn:WH} and Lemma \ref{lemma:f_E12-expression} we have for all $q+1 \leq i\leq r$,
\begin{align*}
     (f_{i})_{E_{\{1,2\}}}(x,y)  
    & = 2(|V(H_{i})|-2)!W_{H_{i}}(x,y) - \dfrac{|V(H_{i})|!}{|\mathrm{Aut}(H_{i})|}t(H_{i},W)
\end{align*}
for almost every $(x,y)\in [0,1]^2$. Thus, for $q+1 \leq i\leq r$,
\begin{align}\label{eq:I2fi}
    & I_{2}\left((f_{i})_{E_{\{1,2\}}}\right) \nonumber \\ 
    & = 2(|V(H_{i})|-2)!\int_{[0, 1]^2} \left\{ W_{H_{i}}(x,y) - \dfrac{|V(H_{i})|\left(|V(H_{i})-1|\right)}{2|\mathrm{Aut}(H_{i})|}t(H_{i},W) \right\} \mathrm d B_{x}\mathrm d B_{y} . 
\end{align}
Combining \eqref{eq:I2E12} and \eqref{eq:I2fi}, the result in \eqref{eq:RBx} follows. This completes  the proof of Proposition \ref{ppn:QR}. \hfill $\Box$

\begin{lemma}\label{lemma:f_E12-expression}
For $q+1\leq i\leq r$,
\begin{align*}
    (f_{i})_{E_{\{1,2\}}}(x,y) = \frac{(|V(H_{i})|-2)!}{|\Aut(H_{i})|}\sum_{1 \leq a\neq b \leq |V(H)|}\left(t_{ab}(x,y,H_{i},W) - t(H_{i},W)\right) , 
\end{align*}
for almost every $(x,y)\in [0,1]^{2}$.
\end{lemma}
\begin{proof}
From \eqref{eq:f-decompE12} and Definition \ref{defn:WH} we have,  
\begin{align}\label{eq:f_E12-1}
    (f_{i})_{E_{\{1,2\}}}(x,y)
    & = \mathbb{E}\left[f_{i}\middle| U_{1}=x,U_{2}=y\right]\nonumber\\
    & = \sum_{H'\in \mathscr{G}_{H_{i}}}t_{1,2}\left(x,y,H',W\right) - |\mathscr{G}_{H_{i}}|t(H_{i},W) , 
\end{align}
for almost every $(x,y)\in [0,1]^2$. Denote by $S_{|V(H_{i})|}$ the set of all $|V(H_{i})|!$ permutations of $V(H_{i})$. 
Then it is easy to observe that 
\begin{align}\label{eq:all_iso1}
    \sum_{\xi\in S_{|V(H_{i})|}}t_{12}(x,y,\xi(H_{i}),W)=|\Aut(H)|\sum_{H'\in\mathscr{G}_{H_{i}}}t_{12}(x,y,H',W).
\end{align}
where $\xi(H)$ is the graph obtained by permuting the vertex labels of $H$ according to the permutation $\xi$. Also,
\begin{align}
    \sum_{\xi\in S_{|V(H_{i})|}}t_{12}(x,y,\xi(H_{i}),W)
    &=\sum_{1\leq a\neq b\leq |V(H_{i})|}\sum_{\substack{\xi\in S_{|V(H_{i})|}\\\xi(a)=1,\xi(b)=2}}t_{12}(x,y,\xi(H_{i}),W)\nonumber\\
    &=\sum_{1\leq a\neq b\leq |V(H_{i})|}\sum_{\substack{\xi\in S_{|V(H_{i})|}\\\xi(a)=1,\xi(b)=2}}t_{\xi^{-1}(1)\xi^{-1}(2)}(x,y,H_{i},W)\nonumber\\
    &=\sum_{1\leq a\neq b\leq |V(H_{i})|}\sum_{\substack{\xi\in S_{|V(H_{i})|}\\\xi(a)=1,\xi(b)=2}}t_{a, b}(x,y,H_{i},W)\nonumber\\
    &=(|V(H_{i})|-2)!\sum_{1\leq a\neq b\leq |V(H_{i})|}t_{a,b}(x,y,H_{i},W)\label{eq:all_iso2}
\end{align}
Combining \eqref{eq:all_iso1} and \eqref{eq:all_iso2}, we have,
\begin{align}\label{eq:WH_alt}
    \sum_{H'\in\mathscr{G}_{H_{i}}}t_{12}(x,y,H',W)=\frac{(|V(H_{i})|-2)!}{|\Aut(H_{i})|}\sum_{1\leq a\neq b\leq |V(H_{i})|}t_{a, b}(x,y,H_{i},W).
\end{align}
Thus combining \eqref{eq:f_E12-1} and \eqref{eq:WH_alt} gives,
\begin{align*}
    (f_{i})_{E_{\{1,2\}}}(x,y)
    & = \frac{(|V(H_{i})|-2)!}{|\Aut(H_{i})|}\sum_{1\leq a\neq b\leq |V(H_{i})|}t_{a, b}(x,y,H_{i},W) - |\mathscr{G}_{H_{i}}|t(H_{i},W)\\
    & = \frac{(|V(H_{i})|-2)!}{|\Aut(H_{i})|}\sum_{a\neq b}\left(t_{ab}(x,y,H_{i},W) - t(H_{i},W)\right) 
\end{align*}
where the last equality follows by recalling that $|\mathscr{G}_{H_{i}}| = \frac{|V(H_{i})|!}{|\mathrm{Aut}(H_{i})|}$, for all $q+1\leq i\leq r$.
\end{proof}

\subsection{Completing the Proof of Theorem \ref{thm:asymp-joint-dist}} 

The result in Theorem \ref{thm:asymp-joint-dist} follows by combining Proposition \ref{ppn:linearTHW},  Proposition \ref{ppn:R}, Proposition \ref{ppn:QR}, and by noting that $\bm T(\cH, W) = (\bm Q^\top, (\bm R +  \tilde{\bm R})^\top)^\top$. 

\section{Moment Generating Function of the Limiting Distribution}

In this section we derive the moment generating function (MGF) of the limiting distribution $\bm Z(\cH, W)$ obtained in Theorem \ref{thm:asymp-joint-dist}. We begin by introducing some notation: For any symmetric function $U: [0, 1]^2 \rightarrow \R$, for $L \geq 2$ define its {\it $L$-th path composition} as follows: For $x, y \in [0, 1]$,  
\begin{align}\label{def:UL}
U^{(L)}(x, y) = \int_{[0, 1]^{L -1 }} U(x, w_1) U(w_1, w_2) \cdots U(w_{L-1}, y) \mathrm d w_1 \mathrm d w_2 \cdots \mathrm d w_{L-1} . 
\end{align} 
$\bm{\alpha} = (\alpha_1, \alpha_2, \ldots, \alpha_r)^\top \in\mathbb{R}^{r}$, define the functions $V_{\bm \alpha} : [0, 1] \rightarrow \mathbb R$ and $U_{\bm \alpha} : [0, 1]^2 \rightarrow \mathbb R$ as: 
\begin{align}\label{eq:def-V}
    V_{\bm \alpha}(x) := \sum_{i=1}^{q}\alpha_{i}\left[\dfrac{1}{\left|\text{Aut}(H_{i})\right|}\sum_{a=1}^{|V(H_{i})|}t_{a}(x,H_{i},W) - \dfrac{|V(H_{i})|}{\left|\text{Aut}(H_{i})\right|}t(H_{i},W)\right] , 
\end{align}
and 
\begin{align}\label{eq:def-U}
    U_{\bm \alpha} (x, y) := \sum_{i=q+1}^{r}\alpha_{i}\left(W_{H_{i}} (x, y) - c_{H_i}(W) \right)  , 
\end{align} 
where $c_{H_i}(W) = \frac{ |V(H_{i})| ( |V(H_i)| - 1) }{2 \left|\text{Aut}(H_{i})\right|}t(H_{i},W)$ and $W_{H_{i}}$ is as in Definition \ref{defn:WH}. We can now express the MGF of $Z(\cH, W)$, for $\mathcal{H} = \{H_1, H_2 , \ldots, H_r\}$ as in Theorem \ref{thm:asymp-joint-dist}, as follows:

\begin{prop}\label{prop:MGF-alpha-limit}
Fix $\bm{\alpha} = (\alpha_1, \alpha_2, \ldots, \alpha_r)^\top \in \mathbb{R}^{r}$ and let $\mathcal{C}:= \sum_{i=q+1}^{r}|\alpha_{i}|\frac{|V(H_{i}) |V(H_{i} - 1 )| }{\left|\text{Aut}(H_{i})\right|}$. Then for $|\theta|<\frac{1}{32\mathcal{C}}$, 
\begin{align*}
    \log&\ \mathbb{E}\left[ e^{\theta \bm{\alpha}^{\top}\bm{Z}(\mathcal{H},W)} \right]\\
    =&\left(  \eta_{\bm \alpha} + \tilde \eta_{\bm \alpha} \right)\dfrac{\theta^{2}}{2} + \sum_{L = 1}^{\infty}2^{L-1} \theta^{L + 2}\int_{[0,1]^{2}}V_{\bm \alpha}(x)V_{\bm \alpha}(y)U_{\bm \alpha}^{(L)}(x,y)\mathrm{d}x\mathrm{d}y + \dfrac{1}{2}\sum_{L=3}^{\infty}\dfrac{(2 \theta)^{L}}{L}\int_{0}^{1} U_{\bm \alpha}^{(L)}(x,x)\mathrm{d}x , 
\end{align*}
where $U_{\bm \alpha}^{(L)}$ is the $L$-th path composition of $U^{(L)}$ as defined in \eqref{def:UL} and 
\begin{align*} 
\eta_{\bm \alpha}  & = \sum_{1 \leq i,j\leq q}\dfrac{\alpha_{i}\alpha_{j}}{\left|\text{Aut}(H_{i})\right|\left|\text{Aut}(H_{j})\right|}\sum_{\substack{a\in V(H_{i})\\ b\in V(H_{j})}}\left(t\left(H
    _{i}\bigoplus_{a,b}H_{j},W\right) - t(H_{i},W)t(H_{j},W)\right) , \nonumber \\
\tilde \eta_{\bm \alpha} & = \sum_{q+1\leq i,j\leq r}\dfrac{\alpha_{i}\alpha_{j}}{2\left|\text{Aut}(H_{i})\right|\left|\text{Aut}(H_{j})\right|}\sum_{\substack{a\neq b\in V(H_{i})\\c\neq d\in V(H_{j})}}t\left(H_{i}\bigominus_{(a,b),(c,d)}H_{j},W\right) - 2\left(\sum_{i=q+1}^{r}\alpha_{i}c_{H_i}(W)\right)^2 . 
\end{align*}  
\end{prop}

\subsection{Proof of Proposition \ref{prop:MGF-alpha-limit}} 
\label{sec:ZHWpf}

Recalling the definition of $\bm{Z}(\mathcal{H},W)$ from Theorem \ref{thm:asymp-joint-dist} note that 
\begin{align}\label{eq:alpha-t-Z}
    \bm{\alpha}^{\top}\bm{Z}(\mathcal{H},W) = \sum_{i=q+1}^{r}\alpha_{i}G_{i} + \int\int U_{\bm \alpha}\mathrm{d}B_{x}\mathrm{d}B_{y} + \int V_{\bm \alpha}\mathrm{d}B_{x} , 
\end{align}
where $\{G_{i}: q+1\leq i\leq r\}\sim N_{r-q}(\bm{0},\Sigma)$ with $\Sigma$ as Definition \ref{def:Sigma} is independent of the standard Brownian motion $\{B_{t}:t\in [0,1]\}$. Now, observe that $U_{\bm \alpha} \in L^{2}([0,1]^{2})$ is symmetric and hence,  the operator,
\begin{align}\label{eq:def-TU}
    T_{U_{\bm \alpha}}f(x) = \int U_{\bm \alpha}(x,y)f(y)\mathrm{d}y, 
\end{align} 
where $f\in L^{2}[0,1]$, is a self-adjoint Hilbert-Schmidt integral operator. Then by the spectral theorem (see \cite[Theorem 8.94 and Theorem 8.83]{renardy2006introduction}) we can find a set of orthonormal eigenfunctions $\{\phi_{s}\}_{s \geq 1}$ corresponding to eigenvalues (with repetition) $\{\lambda_{s}\}_{s \geq 1}$ of $T_{U_{\bm \alpha}}$ which forms a basis of $L^{2}[0,1]$ and 
\begin{align}\label{eq:operator-expansion}
    U_{\bm \alpha}(x, y) = \sum_{s=1}^{\infty}\lambda_{s}\phi_{s}(x)\phi_{s}(y) , 
\end{align}
where the above sum converges in $L^{2}$. Further, we assume that $\{\lambda_{s}\}_{s \geq 1}$ are arranged according to non-increasing order of magnitude and $\lim_{s \rightarrow \infty} \lambda_{s} = 0$. 
Moreover, by the orthonormality of the eigenvectors (see, for example, \cite[Lemma 8, Chapter 6]{laxfunctional}),  
\begin{align}\label{eq:lambda-sum-finite}
   \sum_{s=1}^{\infty}\lambda_{s}^{2} = \|U_{\bm \alpha}\|_{2}^{2}<\infty.
\end{align}
Also, since $V_{\bm \alpha} \in L^{2}[0,1]$, expanding $V_{\bm \alpha}$ using the basis $\{\phi_s\}_{s \geq 1}$ 
we have the following, 
\begin{align}\label{eq:V-expansion}
    V_{\bm \alpha}(x) = \sum_{s=1}^{\infty} \gamma_{s}\phi_{s}(x), 
\end{align} 
where once again the above sum converges in $L^{2}$ and 
\begin{align}\label{eq:def-delta-i}
    \sum_{s=1}^{\infty} \gamma_{s}^{2} = \|V_{\bm \alpha}\|_{2}^{2}<\infty .
\end{align}

Hence, recalling the expression of $\bm{\alpha}^{\top}\bm{Z}(\mathcal{H},W)$ from \eqref{eq:alpha-t-Z} along with Proposition \ref{prop:I2-f_E} and the expansions of $U_{\bm{\alpha}}$ and $V_{\bm{\alpha}}$ from \eqref{eq:operator-expansion} and from \eqref{eq:V-expansion} respectively gives, 
\begin{align*}
    \bm{\alpha}^{\top}\bm{Z}(\mathcal{H},W) & \stackrel{a.s.}= \sum_{i=q+1}^r \alpha_i G_i + \sum_{s=1}^\infty \lambda_s I_2(\phi_s \times \phi_s) + \sum_{s=1}^\infty \gamma_s I_1(\phi_s) \nonumber \\ 
   & \stackrel{a.s.} = \sum_{i=q+1}^r \alpha_i G_i + \sum_{s=1}^\infty \lambda_s ( I_1(\phi_s)^2 -1 ) + \sum_{s=1}^\infty \gamma_s I_1(\phi_s) \tag*{ (by \eqref{eq:fgstochasticintegral}) } \nonumber \\ 
& \stackrel{D} = \sum_{i=q+1}^{r}\alpha_{i}G_{i} + \sum_{s=1}^{\infty}\lambda_{s}(\eta_{s}^{2}-1) + \sum_{s=1}^{\infty} \gamma_s \eta_{s} , 
\end{align*} 
where $\{\eta_{s}\}_{s \geq 1}$ is an independent collection of standard Gaussian random variables which is also independent of $\{G_{i}\}_{q+1 \leq i \leq r}$. Now, for $K \geq 1$, define the truncated version of $ \bm{\alpha}^{\top}\bm{Z}(\mathcal{H},W)$ as follows: 
\begin{align}\label{eq:YMGF}
    Y_{\bm \alpha, K}:= \sum_{i=q+1}^{r}\alpha_{i}G_{i} + \sum_{s=1}^{K}\lambda_{s}(\eta_{s}^{2}-1) + \sum_{s=1}^{K} \gamma_s \eta_{s} . 
\end{align} 
We begin by computing the MGF of $ Y_{\bm \alpha, K}$ in the following lemma:

\begin{lemma}\label{lemma:log-MGF-YN} 
Let $Y_{\bm \alpha, K}$ be as defined above. Then for $|\theta|< \frac{1}{16\mathcal{C}}$, where $\mathcal C$ as in Proposition \ref{prop:MGF-alpha-limit}, the MGF of $Y_{\bm \alpha, K}$ is given by 
\begin{align*}
    \log\mathbb{E}\left[ e^{\theta Y_{\bm \alpha, K}} \right] = \bm{\alpha}_{+}^{\top}\Sigma\bm{\alpha}_{+}\dfrac{\theta^{2}}{2} + \sum_{s=1}^{K}\sum_{L = 1}^{\infty}2^{L-2}\theta^{L+1}\gamma_{s}^{2}\lambda_{s}^{L-1} + \frac{1}{2}\sum_{s=1}^{K}\sum_{L=2}^{\infty}\dfrac{\left(2\lambda_{s}\theta\right)^{L}}{L} , 
\end{align*}
where $\bm{\alpha}_+ = (\alpha_{q+1},\cdots,\alpha_{r})^{\top}$. 
\end{lemma}

\begin{proof} 
For all $K \geq 1$ define,
\begin{align}\label{eq:YK}
    Y_{\bm \alpha}^{(1)} := \sum_{i=q+1}^{r}\alpha_{i}G_{i}, ~Y_{\bm \alpha, K}^{(2)} := \sum_{s=1}^{K}\lambda_{s}(\eta_{s}^{2}-1), \text{ and }  Y_{\bm \alpha, K}^{(3)} := \sum_{s=1}^{K} \gamma_s \eta_{s}.
\end{align}
 by the independence of $Y_{\bm \alpha}^{(1)}$ and $(Y_{\bm \alpha, K}^{(2)}, Y_{\bm \alpha, K}^{(3)})$ and the Cauchy-Schwartz inequality we have,
\begin{align*}
    \mathbb{E}\left[ e^{\theta Y_{\bm \alpha, K}} \right] \leq  \mathbb{E}\left[ e^{\theta Y_{\bm \alpha}^{(1)}} \right] \left( \mathbb{E}\left[ e^{ 2 \theta Y_{\bm \alpha, K}^{(2)}} \right]  \mathbb{E}\left[  e^{2 \theta Y_{\bm \alpha, K}^{(3)}} \right]  \right)^{\frac{1}{2}} . 
\end{align*}
Note that the MGFs of $Y_{\bm \alpha}^{(1)}$ and $Y_{\bm \alpha, K}^{(3)}$ exist for all $\theta$. Also, by \cite[Theorem 2.2]{hladky2019limit}, 
$$\mathbb{E}\left[ e^{ 2 \theta Y_{\bm \alpha, K}^{(2)}} \right] < \infty, \quad \text{ for $| \theta |<\min\left\{\frac{1}{16\sqrt{\sum_{s=1}^{K}\lambda_{s}^{2}}},\frac{1}{16\mathcal{C}}\right\}$}. 
$$
Recalling \eqref{eq:def-U} and \eqref{eq:lambda-sum-finite} observe that, 
\begin{align*}
\sum_{s=1}^{K}\lambda_{s}^{2} \leq \| U_{\bm \alpha}\|_2 \leq  \sum_{i=q+1}^{r}|\alpha_{i}| \frac{|V(H_{i})| ( |V(H_i)| -1 ) }{ |\text{Aut}(H_{i})|}  = \mathcal{C} , 
\end{align*}
since $|W_{H_{i}}|\leq \frac{|V(H_{i})| ( |V(H_i)| -1 ) }{2 |\text{Aut}(H_{i})|}$, for $q+1 \leq i \leq r$. 
This implies, 
\begin{align*}
    \mathbb{E}\left[ e^{\theta Y_{\bm \alpha, K}} \right] < \infty , \quad \text{ for }| \theta |<\frac{1}{16\mathcal{C}}.
\end{align*}

Now, we proceed to compute the MGF of $Y_{\bm \alpha, K}$. Observe that by independence, 
\begin{align}\label{eq:MGF-YN-1-2}
    \mathbb{E}\left[ e^{\theta Y_{\bm \alpha, K} } \right] = \mathbb{E}\left[ e^{\theta Y_{\bm \alpha}^{(1)} } \right] \mathbb{E}\left[ e^{ \theta (Y_{\bm \alpha, K}^{(2)} + Y_{\bm \alpha, K}^{(3)} ) } \right] . 
\end{align} 
We will consider the following 2 cases: 

\begin{itemize}  
\item[{\it Case} 1:] First, assume that $\lambda_{s}\neq 0$, for all $1\leq s\leq K$. Then recalling \eqref{eq:YK} and completing the square,
\begin{align}\label{eq:def-YN2}
    Y_{\bm \alpha, K}^{(2)} + Y_{\bm \alpha, K}^{(3)}  & = \sum_{s=1}^{K}\lambda_{s}\left(\eta_{s}^{2} + \dfrac{\gamma_{s}\eta_{s}}{\lambda_{s}}\right) - \sum_{s=1}^{K}\lambda_{s} \nonumber \\ 
   & = \sum_{s=1}^{K}\lambda_{s}\left(\eta_{s} + \dfrac{\gamma_{s}}{2\lambda_{s}}\right)^{2} - \sum_{s=1}^{K}\dfrac{\gamma_{s}^{2}}{4\lambda_{s}} - \sum_{s=1}^{K}\lambda_{s}
\end{align}
Recall that $\{\eta_{s}\}_{s \geq 1}$ are i.i.d. $N(0,1)$. Hence, from the MGF of non-central chi-squared distribution and \eqref{eq:lambda-sum-finite} we have, 
\begin{align*}
    \mathbb{E}\left[ e^{\theta \sum_{s=1}^{K}\lambda_{s} (\eta_{s} + \frac{\gamma_{s}}{2\lambda_{s}} )^{2} } \right] 
    & = \prod_{s=1}^{K}\dfrac{ e^{ \frac{ \theta \gamma_{s}^{2}}{4\lambda_{s}\left(1-2\lambda_{s} \theta\right)} } }{\sqrt{1-2\lambda_{s} \theta}} \quad 
    \text{ for all }| \theta |<\frac{1}{16\mathcal{C}} \leq \frac{1}{2 \max_{1 \leq s \leq K} | \lambda_s | } . 
\end{align*}
This implies, for all $|\theta|<\frac{1}{16\mathcal{C}}$, 
\begin{align*}
    \log \mathbb{E}\left[ e^{\theta \sum_{s=1}^{K}\lambda_{s} (\eta_{s} + \frac{\gamma_{s}}{2\lambda_{s}} )^{2} } \right]
    & = \sum_{s=1}^{K}\frac{s\gamma_{s}^{2}}{4\lambda_{s}}(1-2\lambda_{s} \theta)^{-1} - \frac{1}{2}\sum_{s=1}^{K}\log\left(1-2\lambda_{s} \theta\right)\nonumber\\
    & = \sum_{s=1}^{K}\frac{s\gamma_{s}^{2}}{4\lambda_{s}}\sum_{L=0}^{\infty}(2\lambda_{s} \theta)^{L} + \frac{1}{2}\sum_{s=1}^{K}\sum_{L=1}^{\infty}\dfrac{(2\lambda_{s} \theta)^{L}}{L} \nonumber \\ 
    & = \sum_{s=1}^{K}\sum_{L=0}^{\infty}2^{L-2}\theta^{L+1}\gamma_{s}^{2}\lambda_{s}^{L-1} + \sum_{s=1}^{K}\lambda_{s} \theta + \frac{1}{2}\sum_{s=1}^{K}\sum_{L=2}^{\infty}\frac{(2\lambda_{s} \theta)^{L}}{L} . \nonumber 
\end{align*}
Recalling \eqref{eq:def-YN2},  this implies that 
\begin{align}\label{eq:MGF-Y2-nz}
    \log\mathbb{E}\left[ e^{ \theta \left( Y_{\bm \alpha, K}^{(2)} + Y_{\bm \alpha, K}^{(3)}\right) } \right] = \sum_{s=1}^{K}\sum_{L=1}^{\infty}2^{L-2}\theta^{L+1}\gamma_{s}^{2}\lambda_{s}^{L-1}
    & + \frac{1}{2}\sum_{s=1}^{K}\sum_{L=2}^{\infty}\frac{(2\lambda_{s} \theta)^{L}}{L}.
\end{align}

\item[{\it Case} 2:] There exists $1\leq t \leq K$ such that $\lambda_{t} = 0$. Recall that $\{\lambda_{s}\}_{s \geq 1}$ are arranged according to non-increasing order of magnitude. Hence, $\lambda_{s} = 0$ for all $t \leq s \leq K$. 
In this case,
\begin{align*}
    Y_{\bm \alpha, K}^{(2)} + Y_{\bm \alpha, K}^{(3)} =  Y_{\bm \alpha, t-1}^{(2)} + Y_{\bm \alpha, t-1}^{(3)}  + \sum_{s=t}^{K} \gamma_{s} \eta_{s} . 
\end{align*} 
Clearly, $Y_{\bm \alpha, t-1}^{(2)} + Y_{\bm \alpha, t-1}^{(3)}$ is independent of 
$\sum_{s=t}^{K} \gamma_{s} \eta_{s}$. 
Moreover, note that $\sum_{s=t}^{K} \gamma_{s} \eta_{s} \sim N(0, \sum_{s=t}^{K} \gamma_s^2)$. Hence, using \eqref{eq:MGF-Y2-nz} we have,
\begin{align}\label{eq:MGF-Y2-z}
    \log\mathbb{E}\left[ e^{\theta (Y_{\bm \alpha, K}^{(2)} + Y_{\bm \alpha, K}^{(3)})} \right] 
    &= \sum_{s=1}^{t-1}\sum_{L=1}^{\infty}2^{L-2}\theta^{L+1}\gamma_{s}^{2}\lambda_{s}^{L-1}
    + \frac{1}{2}\sum_{s=1}^{t-1}\sum_{L=2}^{\infty}\frac{(2\lambda_{s} \theta)^{L}}{L} + \frac{\theta^{2}}{2} \sum_{s=t}^{K} \gamma_{s}^{2} \nonumber \\
    & = \sum_{s=1}^{K}\frac{1}{2} \theta^{2}\gamma_{s}^{2} + \sum_{s=1}^{t-1}\sum_{L=2}^{\infty}2^{L-2}\theta^{L+1}\gamma_{s}^{2}\lambda_{s}^{L-1}
    + \frac{1}{2}\sum_{s=1}^{K}\sum_{L=2}^{\infty}\frac{(2\lambda_{s} \theta)^{L}}{L}\nonumber \\
    & = \sum_{s=1}^{K}\sum_{L=1}^{\infty}2^{L-2}\theta^{L+1}\gamma_{s}^{2}\lambda_{s}^{L-1}
    + \frac{1}{2}\sum_{s=1}^{K}\sum_{L=2}^{\infty}\frac{(2\lambda_{s} \theta)^{L}}{L} . 
\end{align}

\end{itemize}

Combining \eqref{eq:MGF-Y2-nz} and \eqref{eq:MGF-Y2-z} we can conclude that for all $K \geq 1$,
\begin{align}\label{eq:MGF-Y2}
    \log\mathbb{E}\left[ e^{\theta (Y_{\bm \alpha, K}^{(2)} + Y_{\bm \alpha, K}^{(3)} ) } \right]  = \sum_{s=1}^{K}\sum_{L=1}^{\infty}2^{L-2}\theta^{L+1}\gamma_{s}^{2}\lambda_{s}^{L-1} 
    + \frac{1}{2}\sum_{s=1}^{K}\sum_{L=2}^{\infty}\frac{(2\lambda_{s} \theta)^{L}}{L} . 
\end{align}
Finally, recall that $\bm{G} = (G_{q+1},\cdots, G_{r})^{\top}\sim N_{r-q}(\bm{0},\Sigma)$. This implies, 
\begin{align}\label{eq:MGF-Y1}
    \log\mathbb{E}\left[ e^{\theta Y_{\bm \alpha}^{(1)}} \right] = \log\mathbb{E}\left[ e^{\theta \bm{\alpha}_{+}^{\top}\bm{G} } \right] = \frac{\theta^{2}}{2}\bm{\alpha}_{+}^{\top}\Sigma\bm{\alpha}_{+}.
\end{align} 
Now by \eqref{eq:MGF-YN-1-2}, \eqref{eq:MGF-Y2}, and \eqref{eq:MGF-Y1} we conclude that,
\begin{align*}
    \log\mathbb{E}\left[ e^{\theta Y_{\bm \alpha, K}} \right] = \frac{\theta^{2}}{2}\bm{\alpha}_{+}^{\top}\Sigma\bm{\alpha}_{+} + \sum_{s=1}^{K}\sum_{L=1}^{\infty}2^{L-2}\theta^{L+1}\gamma_{s}^{2}\lambda_{s}^{L-1}
    & + \frac{1}{2}\sum_{s=1}^{K}\sum_{L=2}^{\infty}\frac{(2\lambda_{s} \theta)^{L}}{L} , 
\end{align*}
for all $|\theta|<\frac{1}{16\mathcal{C}}$. 
\end{proof}

Now, we compute the MGF of $\bm{\alpha}^{\top}\bm{Z}(\bm{\mathcal{H}},W))$. 

\begin{lemma}\label{lemma:MGF-limit-1st}
The moment generating function of $\bm{\alpha}^{\top}\bm{Z}(\bm{\mathcal{H}},W))$ exists for all $|\theta|<\frac{1}{32\mathcal{C}}$ and is given by,
\begin{align}\label{eq:YMGFlimit}
    \log\mathbb{E}\left[ e^{\theta \bm{\alpha}^{\top}\bm{Z}(\bm{\mathcal{H}},W)} \right] 
    &= \dfrac{\theta^{2} c_2}{2}  + \sum_{L = 1}^{\infty} 2^{L-1}\theta^{L+2} \sum_{s=1}^{\infty} \gamma_{s}^{2}\lambda_{s}^{L} + \frac{1}{2}\sum_{L=3}^{\infty}\dfrac{\left(2 \theta\right)^{L}}{L} \sum_{s=1}^{\infty} \lambda_{s}^L , 
\end{align}  
where $c_2: = \bm{\alpha}_{+}^{\top}\Sigma\bm{\alpha}_{+} + \|V_{\bm \alpha}\|_2^2 + 2 \| U_{\bm \alpha} \|_2^2$. 
\end{lemma}

\begin{proof} Define,
\begin{align}\label{eq:YK}
    Y_{\bm \alpha}^{(1)} := \sum_{i=q+1}^{r}\alpha_{i}G_{i}, ~Y_{\bm \alpha}^{(2)} := \sum_{s=1}^{\infty}\lambda_{s}(\eta_{s}^{2}-1), \text{ and }  Y_{\bm \alpha}^{(3)} := \sum_{s=1}^{\infty} \gamma_s \eta_{s}.
\end{align}
Observe that $\{\frac{\eta_{s}^{2}-1}{\sqrt{2}}\}_{ s \geq 1 }$ and $\{\eta_{s}\}_{s \geq 1}$ are orthonormal. Hence, for $Y_{\bm \alpha, K}$ as defined in \eqref{eq:YMGF} we have, 
\begin{align*}
   \mathbb{E} [(\bm{\alpha}^{\top}\bm{Z}(\bm{\mathcal{H}},W) - Y_{\bm \alpha , K})^2] & \leq \mathbb{E} \left[\sum_{s=K+1}^{\infty}\lambda_{s}(\eta_{s}^{2}-1)\right] + \mathbb{E} \left[ \sum_{s=K+1}^{\infty}\gamma_{s}\eta_{s} \right] \nonumber \\
   & \leq 2\sum_{s=K+1}^{\infty}\lambda_{i}^{2} + \sum_{s=K+1}^{\infty}\gamma_{s}^{2} \rightarrow 0 , 
\end{align*}
as $K \rightarrow \infty$, by \eqref{eq:lambda-sum-finite} and \eqref{eq:def-delta-i}.Thus, 
\begin{align*}
    e^{\theta Y_{\bm \alpha , K} } \overset{P}{\rightarrow} e^{\theta \bm{\alpha}^{\top}\bm{Z}(\bm{\mathcal{H}},W)} , \quad \text{ for all }|\theta|<\frac{1}{32\mathcal{C}}. 
\end{align*} 
From the proof of Lemma \ref{lemma:log-MGF-YN} it follows that $\mathbb{E}[e^{2 \theta Y_{\bm \alpha , K} }] < \infty$ for $|\theta|<\frac{1}{32\mathcal{C}}$. Hence, $\{ e^{\theta Y_{\bm \alpha, K}} : K \geq 1\}$ is uniformly integral for $|\theta|<\frac{1}{32\mathcal{C}}$ and 
\begin{align}\label{eq:MGF-convg}
   \lim_{K \rightarrow \infty} \log\mathbb{E}\left[ e^{\theta Y_{\bm \alpha, K}} \right] = \log\mathbb{E}\left[ e^{\theta \bm{\alpha}^{\top}\bm{Z}(\bm{\mathcal{H}},W)} \right] , \quad \text{ for all } |\theta|<\frac{1}{32\mathcal{C}} . 
\end{align}
Now, note that 
\begin{align}\label{eq:logMGF-YN-convg-exp}
    \bigg|\bm{\alpha}_{+}^{\top}\Sigma\bm{\alpha}_{+}\dfrac{\theta^{2}}{2}
    & + \sum_{L = 1}^{\infty}\sum_{s=1}^{\infty}2^{L-2}\theta^{L+1}\gamma_{s}^{2}\lambda_{s}^{L-1}
     + \frac{1}{2}\sum_{L=2}^{\infty}\sum_{s=1}^{\infty}\dfrac{\left(2\lambda_{s}\theta\right)^{L}}{L} - \log\mathbb{E}\left[ e^{\theta Y_{\bm \alpha, K}} \right]\bigg|\nonumber\\
     &\leq \left|\sum_{s=K+1}^{\infty}\sum_{L = 1}^{\infty}2^{L-2}\theta^{L+1}\gamma_{s}^{2}\lambda_{s}^{L-1} + \frac{1}{2}\sum_{s=K+1}^{\infty}\sum_{L=2}^{\infty}\dfrac{\left(2\lambda_{s}\theta\right)^{L}}{L}\right|\nonumber\\ 
       & \leq \sum_{s=K+1}^{\infty}\sum_{L = 1}^{\infty}2^{L-2}|\theta|^{L+1} \gamma_{s}^{2}|\lambda_{s}|^{L-1}
     + \frac{1}{2}\sum_{s=K+1}^{\infty}\sum_{L=2}^{\infty}\dfrac{\left|2\lambda_{s} \theta\right|^{L}}{L}\nonumber\\
    &\leq \frac{1}{8}\sum_{s=K+1}^{\infty}\sum_{L=1}^{\infty}\frac{2^{L+1}}{32^{L+1}\mathcal{C}^{L+1}}\gamma_{s}^{2}|\lambda_{s}|^{L-1} + \frac{1}{32^{2}}\sum_{s=K+1}^{\infty}\frac{\lambda_{s}^{2}}{\mathcal{C}^{2}} + \frac{1}{2}\sum_{s=K+1}^{\infty}\sum_{L=3}^{\infty}\frac{2^{L}\mathcal{C}^{L}}{L s^{L/2}32^{L}\mathcal{C}^{L}}\nonumber\\
     &\leq \frac{1}{8}\sum_{s=K+1}^{\infty}\gamma_{s}^{2}\sum_{L=1}^{\infty}\frac{1}{16^{L+1}\mathcal{C}^{2}}  + \sum_{s=K+1}^{\infty}\frac{\lambda_{s}^{2}}{\mathcal{C}^{2}}  + \frac{1}{2}\sum_{s=K+1}^{\infty}\frac{1}{s^\frac{3}{2}}\sum_{L=3}^{\infty}\frac{1}{L16^{L}}   \rightarrow 0,
\end{align}
as $K \rightarrow \infty$.
Thus,  combining \eqref{eq:MGF-convg} and \eqref{eq:logMGF-YN-convg-exp},  for $|\theta|<\frac{1}{32\mathcal{C}}$,   
\begin{align*}
    & \log\mathbb{E}\left[ e^{\theta \bm{\alpha}^{\top}\bm{Z}(\bm{\mathcal{H}},W)} \right] \\ 
    &= \bm{\alpha}_{+}^{\top}\Sigma\bm{\alpha}_{+}\dfrac{\theta^{2}}{2} + \sum_{L = 1}^{\infty}\sum_{s=1}^{\infty}2^{L-2}\theta^{L+1}\gamma_{s}^{2}\lambda_{s}^{L-1}
     + \frac{1}{2}\sum_{L=2}^{\infty}\sum_{s=1}^{\infty}\dfrac{\left(2\lambda_{s}\theta\right)^{L}}{L}\\
    & = \dfrac{\theta^{2}}{2}\left(\bm{\alpha}_{+}^{\top}\Sigma\bm{\alpha}_{+} + \sum_{s=1}^{\infty}\gamma_{s}^{2} + 2\sum_{s=1}^{\infty}\lambda_{s}^{2}\right) + \sum_{L = 2}^{\infty}\sum_{s=1}^{\infty}2^{L-2}\theta^{L+1}\gamma_{s}^{2}\lambda_{s}^{L-1} + \frac{1}{2}\sum_{L=3}^{\infty}\sum_{s=1}^{\infty}\dfrac{\left(2\lambda_{s}\theta\right)^{L}}{L} \\ 
       & = \dfrac{\theta^{2}}{2}\left(\bm{\alpha}_{+}^{\top}\Sigma\bm{\alpha}_{+} + \sum_{s=1}^{\infty}\gamma_{s}^{2} + 2\sum_{s=1}^{\infty}\lambda_{s}^{2}\right) + \sum_{L = 1}^{\infty} 2^{L-1}\theta^{L+2} \sum_{s=1}^{\infty} \gamma_{s}^{2}\lambda_{s}^{L} + \frac{1}{2}\sum_{L=3}^{\infty}\dfrac{\left(2 \theta\right)^{L}}{L} \sum_{s=1}^{\infty} \lambda_{s}^L . 
\end{align*} 
Now, observing that $\sum_{s=1}^{\infty}\gamma_{s}^{2} = \|V_{\bm \alpha}\|_2^2$ and $\sum_{s=1}^{\infty}\lambda_{s}^{2} = \| U_{\bm \alpha} \|_2^2$, the result in Lemma \ref{lemma:MGF-limit-1st} follows.
\end{proof} 

Now we relate the terms in \eqref{eq:YMGFlimit} with those in Proposition \ref{prop:MGF-alpha-limit}. To begin with note that $\int_{0}^{1} U_{\bm \alpha}^{(L)}(x, x) \mathrm d x $ can be interpreted as the density of the $L$-cycle for the function $U_{\bm \alpha}$. Since $U_{\bm \alpha}: [0, 1]^2 \rightarrow \mathbb R$, is a symmetric and bounded function, it follows from arguments in \cite[Section 7.5]{lovasz2012large} that, for $L \geq 3$, 
\begin{align}\label{eq:ULcycle}
\sum_{s=1}^\infty \lambda_s^L = \int_{0}^{1} U_{\bm \alpha}^{(L)}(x, x) \mathrm d x , 
\end{align}
where $\{\lambda_{s}\}_{s \geq 1}$ are the eigenvalues, with eigenfunctions $\{\phi_{s}\}_{s \geq 1}$, for the operator $T_{U_{\bm \alpha}}$ as defined in \eqref{eq:def-TU}. 
This implies, by the spectral theorem, 
\begin{align}\label{eq:ULxy} 
 U_{\bm \alpha}^{(L)}(x,y) = \sum_{s=1}^\infty \lambda_s^{L} \phi_s(x) \phi_s(y) , 
\end{align} 
Using this we relate $\sum_{s=1}^{\infty}\gamma_{s}^{2}\lambda_{s}^{L}$ in terms of the functions $U_{\bm \alpha}$ and $V_{\bm \alpha}$.

\begin{lemma}\label{lemma:sum-delta-lambda}
For all $L \geq 1$,
\begin{align*}
    \sum_{s=1}^{\infty}\gamma_{s}^{2}\lambda_{s}^{L} = \int_{[0, 1]^2} V_{\bm \alpha}(x)V_{\bm \alpha}(y)U_{\bm \alpha}^{(L)}(x,y)\mathrm{d}x\mathrm{d}y . 
\end{align*}
\end{lemma}

\begin{proof} 
Since $V_{\bm \alpha}(x) = \sum_{s=1}^{\infty} \gamma_{s}\phi_{s}(x)$ (recall \eqref{eq:V-expansion}), we have $\gamma_{s} = \int_{0}^{1} V_{\bm \alpha}(x) \phi_{s}(x) \mathrm{d}x$, for $s \geq 1$, by the orthonormality of the eigenvectors. 
Hence, 
\begin{align*} 
\sum_{s=1}^{\infty}\gamma_{s}^{2}\lambda_{s}^{L} = \int_{[0, 1]^2} \left(\sum_{s=1}^\infty \lambda_{s}^{L} \phi_s(x) \phi_s(y) \right) V_{\bm \alpha}(x) V_{\bm \alpha}(y)  \mathrm d x \mathrm dy = \int_{[0, 1]^2} V_{\bm \alpha}(x)V_{\bm \alpha}(y)U_{\bm \alpha}^{(L)}(x,y)\mathrm{d}x\mathrm{d}y , 
\end{align*}
by \eqref{eq:ULxy}.
\end{proof}

Combining Lemma \ref{lemma:MGF-limit-1st}, \eqref{eq:ULcycle}, and Lemma \ref{lemma:sum-delta-lambda} gives, 
\begin{align*}
    & \log\mathbb{E}\left[ e^{\theta \bm{\alpha}^{\top}\bm{Z}(\bm{\mathcal{H}},W)} \right] \nonumber \\ 
    &= \dfrac{\theta^{2} c_2}{2}  + \sum_{L = 1}^{\infty} 2^{L-1}\theta^{L+2} \int_{[0, 1]^2} V_{\bm \alpha}(x)V_{\bm \alpha}(y)U_{\bm \alpha}^{(L)}(x,y)\mathrm{d}x\mathrm{d}y + \frac{1}{2}\sum_{L=3}^{\infty}\dfrac{\left(2 \theta\right)^{L}}{L} \int_{0}^{1} U_{\bm \alpha}^{(L)}(x, x)\mathrm{d}x , 
\end{align*}  
where $c_2 = \bm{\alpha}_{+}^{\top}\Sigma\bm{\alpha}_{+} + \|V_{\bm \alpha}\|_{2}^{2} + 2 \|U_{\bm \alpha}\|_{2}^{2}$.
The result in Proposition \ref{prop:MGF-alpha-limit} now follows from the next lemma.

\begin{lemma}\label{lemma:norm-U-2}
 $\eta_{\bm \alpha}$, $\tilde \eta_{\bm \alpha}$ as defined in Proposition \ref{prop:MGF-alpha-limit} the following hold: 
\begin{align*}
 \|V_{\bm \alpha}\|_{2}^{2} = \eta_{\bm \alpha} \quad \text{ and } \quad   2 \|U_{\bm \alpha} \|_{2}^{2} + \bm{\alpha}_{+}^{\top}\Sigma\bm{\alpha}_{+} = \tilde \eta_{\bm \alpha}  . 
\end{align*} 
\end{lemma}

\begin{proof} 
Recalling the definition of $V_{\bm \alpha}$ from \eqref{eq:def-V} we have, 
\begin{align}\label{eq:V2-form}
    & \int_{0}^{1} V_{\bm \alpha}^{2}(x)\mathrm{d}x\nonumber\\
    & = \sum_{1 \leq i , j \leq q } \kappa_i \kappa_j \int_{0}^{1} \sum_{\substack{a\in V(H_{i})\\b\in V(H_{j})}}\left(t_{a}(x,H_{i},W) - t(H_{i},W)\right)\left(t_{a}(x,H_{j},W) - t(H_{j},W)\right)\mathrm{d}x , 
\end{align} 
where $\kappa_i := \frac{\alpha_{i}}{|\text{Aut}(H_{i})|}$. Observe that,
\begin{align}\label{eq:V2-norm-term-expression}
    & \int_0^1 \left(t_{a}(x,H_{i},W) - t(H_{i},W)\right) \left(t_{a}(x,H_{j},W) - t(H_{j},W)\right)\mathrm{d}x\nonumber\\
    & = \int_0^1 t_{a}(x,H_{i},W)t_{b}(x,H_{j},W)  - t(H_{i},W)t_{b}(x,H_{j},W) - t(H_{j},W)t_{a}(x,H_{j},W) + t(H_{i},W)t(H_{j},W)\mathrm{d}x\nonumber\\
    & = t\left(H_{i}\bigoplus_{a,b}H_{j},W\right) - t(H_{i},W)t(H_{j},W) . 
\end{align}
Combining \eqref{eq:V2-form} and \eqref{eq:V2-norm-term-expression} gives $\|V_{\bm \alpha}\|_{2}^{2}=\eta_{\bm \alpha}$.

Next, recalling the definition of $U_{\bm \alpha}$ from \eqref{eq:def-U} we have,
\begin{align}\label{eq:U2-first}
    \int_{[0, 1]^2} U_{\bm \alpha}^{2}\mathrm{d}x\mathrm{d}y 
    & = \int_{[0, 1]^2}\left(\sum_{i=q+1}^{r}\alpha_{i}W_{H_{i}}(x,y)\right)^{2}\mathrm{d}x\mathrm{d}y - \left(\sum_{i=q+1}^{r}\alpha_{i}c(H_{i},W)\right)^{2} , 
\end{align}
since $\int_{[0, 1]^2}W_{H_{i}}(x,y)\mathrm{d}x\mathrm{d}y = c(H_{i},W)$. By definition,
\begin{align}\label{eq:alphaWH-sq-norm}
    \int_{[0, 1]^2}\left(\sum_{i=q+1}^{r}\alpha_{i}W_{H_{i}}(x,y)\right)^{2}\mathrm{d}x\mathrm{d}y = \sum_{q+1 \leq i,j \leq r} \alpha_{i}\alpha_{j}\int_{[0, 1]^2}W_{H_{i}}(x,y)W_{H_{j}}(x,y)\mathrm{d}x\mathrm{d}y , 
\end{align}
and 
\begin{align}\label{eq:WhiWhj-integral}
    & \int_{[0, 1]^2} W_{H_{i}}(x,y)
    W_{H_{j}}(x,y)\mathrm{d}x\mathrm{d}y \nonumber\\
    & = \dfrac{1}{4\left|\text{Aut}(H_{i})\right|\left|\text{Aut}(H_{j})\right|}\sum_{\substack{a\neq b\in V(H_{i})\\ c\neq d\in V(H_{j})}}\int_{[0, 1]^2}t_{a,b}(x,y,H_{i},W)t_{c,d}(x,y,H_{j},W)\mathrm{d}x\mathrm{d}y\nonumber\\
    & = \dfrac{1}{4\left|\text{Aut}(H_{i})\right|\left|\text{Aut}(H_{j})\right|}\sum_{\substack{a\neq b\in V(H_{i})\\ c\neq d\in V(H_{j})}} t\left(H_{i}\bigoplus_{(a,b),(c,d)}H_{j},W\right).
\end{align}
Finally, recalling the definition of $\Sigma$ from Definition \ref{def:Sigma} note that 
\begin{align}\label{eq:gaussian-var-expression}
    & \bm{\alpha}_{+}^{\top}\Sigma\bm{\alpha}_{+} \nonumber \\ 
    & = \sum_{q+1 \leq i, j \leq r} \frac{\kappa_{i}\kappa_{j}}{2} \sum_{\substack{(a,b)\in E^{+}(H_{i})\\(c,d)\in E^{+}(H_{j})}}\left(t(H_{i}\bigominus_{(a,b),(c,d)}H_{j},W) - t(H_{i}\bigoplus_{(a,b),(c,d)}H_{j},W)\right)\nonumber\\
    & = \sum_{q+1 \leq i, j \leq r} \frac{\kappa_{i}\kappa_{j}}{2}\sum_{\substack{a\neq b\in V(H_{i})\\c\neq d\in V(H_{j})}}\left(t(H_{i}\bigominus_{(a,b),(c,d)}H_{j},W) - t(H_{i}\bigoplus_{(a,b),(c,d)}H_{j},W)\right) . 
\end{align}
where the last equality follows from Remark \ref{rem:def_ext_join}. 
Combining \eqref{eq:U2-first}, \eqref{eq:alphaWH-sq-norm}, \eqref{eq:WhiWhj-integral}, and \eqref{eq:gaussian-var-expression} we have $2 \|U_{\bm \alpha}\|_{2}^{2} +  \bm{\alpha}_{+}^{\top}\Sigma\bm{\alpha}_{+} =\tilde \eta_{\bm \alpha}$. 
\end{proof}

\begin{remark}\label{rem:def_ext_join} 
Note that 
both the {weak} and {strong} edge join
operations can be extended to arbitrary $(a,b)\in V(H_{1})^{2}$ and $(c,d)\in
V(H_{2})^2$,  with $ a\neq b$ and $ c\neq d$ as follows: For the strong join we keep all edges, while for the weak join we keep the joined graph simple by merging any resulting double edge. In particular, if either $(a,b)\not\in E^{+}(H_{1})$ or $(c,d)\not\in
E^{+}(H_{2})$, then the {weak} and {strong} edge joins are the same graph. This implies, 
$$t(H_{1}\bigominus_{(a,b),(c,d)}H_{2},W) = t(H_{1}\bigoplus_{(a,b),(c,d)}H_{2},W) , $$
which explains the step in \eqref{eq:gaussian-var-expression}. 
\end{remark}

\section{Proof of Corollary \ref{cor:irregjoint} and Corollary \ref{cor:marginaldist}}\label{sec:proofofmarginaldist}

\subsection{Proof of Corollary \ref{cor:irregjoint}} 

Since $W$ is $H$-irregular for all $H \in \mathcal H$, by Theorem \ref{thm:asymp-joint-dist}, 
\begin{align*}
    \bm Z(\cH, G_n)\dto \bm Z(\cH, W) = (Z(H_{1},W),\cdots, Z(H_{r},W))
\end{align*}
where 
\begin{align*}
    Z(H_{i},W) = \int_{0}^{1} \left\{ \dfrac{1}{|\mathrm{Aut}(H_{i})|}\sum_{a=1}^{|V(H_{i})|}t_{a}(x,H_{i},W) - \dfrac{|V(H_{i})|}{|\mathrm{Aut}(H_{i})|}t(H_{i},W) \right\} \mathrm d B_{x} ,  
\end{align*} 
for $1\leq i\leq r$. Since a linear stochastic integral has a centered Gaussian distribution, $$\bm Z(\cH, W) \stackrel{D}= N_r( \bm 0, \Gamma),$$ where $\Gamma = (( \tau_{ij} ))_{1 \leq i, j \leq r}$, with  
\begin{align*}
\tau_{ij}  & =    \Cov(Z(H_{i},W),Z(H_{j},W)) \\ 
    & = \dfrac{1}{|\mathrm{Aut}(H_{i})| |\mathrm{Aut}(H_{j})|} \Bigg\{ \sum_{a=1}^{|V(H_{i})|}\sum_{b=1}^{|V(H_{j})|}\int_0^1 t_{a}(x,H_{i},W)t_{b}(x,H_{j},W)\mathrm{d}x   \\ 
      & \hspace{2.75in} - |V(H_{i})| t(H_{i},W) |V(H_{j})| t(H_{j},W) \Bigg \}.
\end{align*} 
A direct computation shows that, for all $1\leq i,j\leq r$,
\begin{align*}
    \int_0^1 t_{a}(x,H_{i},W)t_{b}(x,H_{j},W)\mathrm{d}x = t\left(H_i\bigoplus_{a,b}H_j,W\right) . 
\end{align*}
This shows, $\tau_{ij} = \Cov(Z(H_{i},W),Z(H_{j},W))$ equals to the expression in the statement of Corollary \ref{cor:irregjoint}.

\subsection{Proof of Corollary \ref{cor:marginaldist}}

Note that when $W$ is $H$-irregular, the result is immediate from Corollary \ref{cor:irregjoint}. Hence, suppose $W$ is $H$-regular. 
In this case, from Theorem \ref{thm:asymp-joint-dist} we know that,
\begin{align}\label{eq:limitdist1}
    Z(H, G_n)\dto G
    + \underbrace{\int_{0}^{1}\int_{0}^{1} \left\{ W_{H}(x,y) - \dfrac{|V(H)|\left(|V(H)-1|\right)}{2|\mathrm{Aut}(H)|} t(H,W) \right\} \mathrm d B_{x}\mathrm d B_{y}}_{Z(H,W)} , 
\end{align}
where $G\sim N(0,\eta_{H,W}^2)$, with 
\begin{align}\label{eq:exprvarreg}
    \eta_{H,W}^2 = \frac{1}{2|\Aut(H)|^{2}}\sum_{(a,b),(c,d)\in
    E^{+}(H)}\left[t\left(H\bigominus_{(a,b),(c,d)}
    H,W\right)-t\left(H\bigoplus_{(a,b),(c,d)} H,W\right)\right] ,
\end{align}
and $G$ is independent of the Brownian motion $\{B_t:0\leq t\leq 1\}$. Note that the expression of $\eta_{H,W}$ follows from Theorem \ref{thm:asymp-joint-dist} and \eqref{eq:Gvariance}. From \eqref{eq:degreeH} recall that,
\begin{align*}
    d_{W_H} = \dfrac{|V(H)|\left(|V(H)-1|\right)}{2|\mathrm{Aut}(H)|} t(H,W),
\end{align*}
is an eigenvalue of the kernel $W_H$ with corresponding eigenfunction $1$. Now, considering the spectral decomposition of $W_H$ notice,
\begin{align*}
    W_{H}(x,y) - \dfrac{|V(H)|\left(|V(H)-1|\right)}{2|\mathrm{Aut}(H)|} t(H,W) = \sum_{\lambda\in \Spec^{-}(W_H)}\lambda\phi_{\lambda}(x)\phi_{\lambda}(y) , 
\end{align*} 
almost everywhere. 
Then,
\begin{align*}
    Z(H,W) = \int_{0}^{1}\int_{0}^{1} \sum_{\lambda\in \Spec^{-}(W_H)}\lambda\phi_{\lambda}(x)\phi_{\lambda}(y) \mathrm d B_{x}\mathrm d B_{y} = \sum_{\lambda\in \Spec^{-}(W_H)}\int_{0}^{1}\int_{0}^{1}\lambda\phi_{\lambda}(x)\phi_{\lambda}(y) \mathrm d B_{x}\mathrm d B_{y} , 
\end{align*}
where the last equality follows by Proposition \ref{prop:I2-f_E}. Now, using \eqref{eq:fgstochasticintegral} and $\int_{0}^{1}\phi_\lambda(x)^2\mathrm{d}x =1$, we get 
\begin{align*}
    \int_{0}^{1}\int_{0}^{1}\lambda\phi_{\lambda}(x)\phi_{\lambda}(y) \mathrm d B_{x}\mathrm d B_{y} = \lambda\left[\left(\int_{0}^{1}\phi_\lambda(x)\mathrm{d}B_{x}\right)^2-1\right], \text{ for all }\lambda\in \Spec^{-}(W_H).
\end{align*}
The orthonormality of the eigenvectors $\{\phi_{\lambda}\}_{\lambda\in \Spec^{-}(W_H)}$ implies,
\begin{align}\label{eq:eqdistZ2}
    Z(H,W) = \sum_{\lambda\in\Spec^{-}(W_H)}\lambda \left[\left(\int_{0}^{1}\phi(x)\mathrm{d}B_{x}\right)^2-1\right]\overset{D}{=} \sum_{\lambda\in \Spec^{-}(W_H)}\lambda(Z_\lambda^2-1) , 
\end{align}
where $\{Z_\lambda:\lambda\in \Spec^{-}(W_H)\}$ are i.i.d. $N(0,1)$ which are independent from $G$. The proof of Corollary \ref{cor:marginaldist} is now complete by collecting \eqref{eq:limitdist1}, \eqref{eq:exprvarreg}, and \eqref{eq:eqdistZ2}.

\section{Proof of Theorem \ref{thm:ZnHGn}} 
\label{sec:limitWGnpf}

We begin by expressing the estimated distributions $\hat {\bm Z}(\cH, G_{n})$ (recall \eqref{eq:ZHGnestimate}) in terms of stochastic integrals. For this, suppose $I_1, I_2, \ldots, I_n$ be a partition of $[0, 1]$ into intervals of length $1/n$, that is, $I_s = [ \frac{s-1}{n}, \frac{s}{n})$, for $1 \leq s \leq n$. Let $\eta_s = \int_{I_s} \mathrm d B_s$, where $\left(B_{t}\right)_{t\in[0,1]}$ is a standard Brownian motion on $[0,1]$ independent of $\{G_{n}\}_{n\geq 1}$. Then $\{\eta_1, \eta_2, \ldots, \eta_n\}$ is a collection of i.i.d. $N(0, 1/n)$ random variables.  With notations as in \eqref{eq:ZHestimate}, define 
\begin{align}\label{eq:ZHGnB}
    \hat Z'(H_{i}, G_{n}) : = 
    \begin{cases}
        \sum_{v=1}^{n} ( \hat t(v, H_i, G_n) -  \bar t(H_i, G_n) ) \eta_v  & \text{ if }  1 \leq i \leq q, \\ \\ 
   \sum_{1 \leq u, v \leq n} ( \hat W^{G_n}_{H_i}(u, v) - \bar W^{G_n}_{H_i} ) \left(\eta_{u}\eta_{v} - \frac{\delta_{u, v}}{n} \right) & \text{ if } q+1 \leq i \leq  r , 
    \end{cases}
\end{align} 
where, recall that, $\delta_{u, v} = \bm 1\{u=v\}$, $\bar t(H_i, G_n) = \frac{1}{n} \sum_{v=1}^n \hat t(v, H_i, G_n)$, and $\bar W^{G_n}_{H_i} = \frac{1}{n^2} \sum_{1 \leq u, v \leq n} \hat W^{G_n}_{H_i}(u, v)$. Denote, 
\begin{align}\label{eq:ZHB}
\hat{\bm Z}'( \mathcal{H}, G_{n}) = (\hat Z'( H_1, G_{n}), \hat Z'( H_2, G_{n}), \ldots, \hat Z'( H_r, G_{n})) ^\top . 
\end{align}  
Note that $\hat{\bm Z}'( \mathcal H, G_{n})$ has the same distribution as $\hat{\bm Z}(\mathcal H, G_n)$ and $\hat{\bm Z}'( \mathcal H, G_{n})$ is defined on the same probability space as $\{B_t\}_{t \in [0, 1]}$. Now, recalling \eqref{eq:tHGnestimate}, for $x \in [0, 1]$, define 
\begin{align}\label{eq:tHGnestimatefn} 
\hat t( x , H_i, G_n) = \hat t( \lceil n x \rceil, H, G_n) . 
\end{align} 
Note that $\int_0^1 \hat t( x, H_i, G_n) \mathrm d x = \bar t(H_i, G_n)$. Also, recalling \eqref{eq:WGnuv}, for $x, y \in [0, 1]$, define 
\begin{align}\label{eq:WGnHestimate} 
\hat W^{G_n}_{H_i}(x, y) = \hat W^{G_n}_{H_i}( \lceil n x \rceil, \lceil n y \rceil) . 
\end{align} 
Observe that $\int_{[0, 1]^2} \hat W^{G_n}_{H_i}(x, y) \mathrm d x \mathrm d y = \bar W^{G_n}_{H_i}$. Hence, \eqref{eq:ZHGnB} can be expressed as: 
\begin{align*}
    \hat Z'(H_{i}, G_{n}) : = 
    \begin{cases}
      \int_0^1 \left( \hat t( x , H_i, G_n) -  \int_0^ 1 \hat t( x , H_i, G_n) \mathrm d  x\right) \mathrm dB_x  & \text{ if }  1 \leq i \leq q, \\ \\ 
  \int_{[0,1]^2 }  \left( \hat W^{G_n}_{H_i}(x, y) - \int_{[0, 1]^2} \hat W^{G_n}_{H_i}(x, y) \mathrm d x \mathrm d y \right) \mathrm d B_x \mathrm d B_y & \text{ if } q+1 \leq i \leq  r . 
    \end{cases}
\end{align*}

In the next lemma we show that distribution of $ \hat Z'(H_{i}, G_{n})$ remains unchanged in the limit when $\hat t( x , H_i, G_n)$ is replaced by $\overline{t}(x,H,W^{G_n})$ (recall \eqref{eq:H_regular}) and $\hat W^{G_n}_{H_i}(x, y)$ (recall \eqref{eq:WGnHestimate}) is replaced by $W^{G_n}_{H_{i}}$, the 2-point conditional homomorphism kernel of the empirical graphon $W^{G_n}$ (recall \eqref{eq:WH}).
This is because the difference between all homomorphisms and injective homomorphisms is negligible in the limit.

\begin{lemma}\label{lm:ZHGndistribution} For $1 \leq i \leq r$, as $n \rightarrow \infty$, 
  \begin{align}\label{eq:ZhatprimeYdiffo1}
        \E\left[\left|\hat Z'(H_i,G_n) - Y(H_i,W^{G_n})\right|^2\middle| G_n\right] \stackrel{a.s.}{\ra} 0 ,  
    \end{align}
where $Y(H_i, W^{G_{n}})$ is defined as follows: 
\begin{itemize}
\item for  $1 \leq i \leq q$, 
\begin{align*}
  Y(H_i, W^{G_{n}}):= \int_{0}^{1} \left\{ \dfrac{1}{|\mathrm{Aut}(H_{i})|}\sum_{a=1}^{|V(H_{i})|}t_{a}(x,H_{i},W^{G_{n}}) - \dfrac{|V(H_{i})|}{|\mathrm{Aut}(H_{i})|}t(H_{i},W^{G_{n}}) \right\} \mathrm d B_{x} , 
\end{align*} 
\item for $q+1\leq i\leq  r$, 
\begin{align*}
    Y(H_i, W^{G_{n}})    & := \int_{0}^{1}\int_{0}^{1} \left\{ W^{G_n}_{H_{i}}(x,y) - \dfrac{|V(H_{i})|\left(|V(H_{i})-1|\right)}{2|\mathrm{Aut}(H_{i})|}t(H_{i},W^{G_{n}}) \right\} \mathrm d B_{x}\mathrm d B_{y} . 
\end{align*}
\end{itemize}
\end{lemma}

\begin{proof}
    We start the proof by showing \eqref{eq:ZhatprimeYdiffo1} for $1\leq i\leq q$, that is, when $W$ is $H_i$-irregular. 

    Towards this,  notice that $t(H_i, W^{G_n}) = \int t_a(x, H_i, W^{G_n})\mathrm d x$, for all $1\leq a\leq |V(H_i)|$, and hence, 
    \begin{align}\label{eq:ZY}
        & \E\left[\left|\hat Z'(H_i,G_n) - Y(H_i,W^{G_n})\right|^2\middle|G_n\right] \nonumber \\ 
        &\lesssim \int_0^1 \left(\hat t(x, H_i, G_n) - \dfrac{1}{|\mathrm{Aut}(H_{i})|}\sum_{a=1}^{|V(H_{i})|}t_{a}\left(x,H_{i},W^{G_{n}}\right)\right)^2\mathrm dx \nonumber \\
        =&\sum_{v=1}^{n}\int_{I_v}\left(\hat t(v, H_i, G_n) - \dfrac{1}{|\mathrm{Aut}(H_{i})|}\sum_{a=1}^{|V(H_{i})|}t_{a}\left(x,H_{i},W^{G_{n}}\right)\right)^2\mathrm dx , 
    \end{align} 
    almost surely, where $I_v = [ \frac{v-1}{n}, \frac{v}{n})$, for $1 \leq v \leq n$. Note that the first inequality follows by the boundedness property of stochastic integrals (see Section \ref{sec:stochasticintegral}) and the second equality follows from the definition of $\hat t(\cdot, H_i, G_n)$ in \eqref{eq:tHGnestimatefn}. From \eqref{eq:tHGnestimate} recall that,
    \begin{align*}
        \hat t(v, H_i, G_n) = \frac{1}{|\Aut(H_i)|}\sum_{a=1}^{|V(H_i)|}\dfrac{X_a(v, H_i, G_n)}{n^{|V(H_i)|-1}}, 
    \end{align*}
    for $1\leq v\leq n$. Then from \eqref{eq:ZY} we get, 
    \begin{align}\label{eq:ZhatprimeYdiff}
        & \E\left[\left|\hat Z'(H_i,G_n) - Y(H_i,W^{G_n})\right|^2\middle|G_n\right] \nonumber \\ 
        &\lesssim_{H_i}\sum_{v=1}^{n}\sum_{a=1}^{|V(H_i)|}\int_{I_v}\left(\dfrac{X_a(v, H_i, G_n)}{n^{|V(H_i)|-1}} - t_{a}\left(x,H_{i},W^{G_{n}}\right)\right)^2\mathrm dx , 
    \end{align} 
    almost surely. Now, recall Definition \ref{defn:tabxyHW} to write,  
    \begin{align*}
        t_a(x, H_i, W^{G_n}) = \int_{[0,1]^{|V(H_i)|-1}}\prod_{u \in N_{H_i}(a)}W^{G_n}(x,x_u)\prod_{(u, v)\in E(H_i\setminus \{a\})}W^{G_n}(x_u, x_v)\prod_{v \in V(H_i\setminus\{a\})}\mathrm dx_v .
    \end{align*}
    Recalling the construction of empirical graphon $W^{G_n}$ from \eqref{eq:emp_graph}, $t_a(x, H_i, W^{G_n})$ can be equivalently written as follows:   
    \begin{align*}
        t_a(x, H_i, W^{G_n}) = \frac{1}{n^{|V(H_i)|-1}}\sum_{\tilde{\bm s}_{\{a\}^c}}\prod_{y \in N_{H_i}(a)}w_{vs_y}(G_n)\prod_{(x, y)\in E(H_i\setminus \{a\})}w_{ s_x s_y }(G_n) ,
    \end{align*} 
    for all $x\in I_v$ and $1\leq v\leq n$,  where the sum is over tuples $\tilde{\bm s}_{\{a\}^c} = (s_x)_{x\in V(H_i)\setminus\{a\}}\in \{[n]\setminus\{v\}\}^{|V(H_i)|-1}$. From Definition \ref{def:Xav} recall that $X_{a}(v,H_i, G_n)$ is sum over tuples $\bm s_{\{a\}^c} = (s_x)_{x\in V(H_i)\setminus\{a\}}$ where the elements are all distinct. Thus, for all $x\in I_v$ and $1\leq v\leq n$, 
    \begin{align}\label{eq:Xav-tavO1/n}
        \left|\dfrac{X_a(v, H_i, G_n)}{n^{|V(H_i)|-1}} - t_{a}\left(x,H_{i},W^{G_{n}}\right)\right| \lesssim_{H_i}\frac{1}{n} ,
    \end{align}
    since the difference in the LHS above counts the number of non-injective homomorphisms $\phi:V(H_i)\ra V(G_n)$ such that $\phi(a) = v$, up to a constant depending on $H_i$. Substituting the bound from \eqref{eq:Xav-tavO1/n} in \eqref{eq:ZhatprimeYdiff} gives,
    \begin{align*}
        \E\left[\left|\hat Z'(H_i,G_n) - Y(H_i,W^{G_n})\right|^2\middle|G_n\right] 
        &\lesssim_{H_i}\frac{1}{n^2} , 
    \end{align*}
    almost surely. This proves \eqref{eq:ZhatprimeYdiffo1}, for $1\leq i\leq q$.

    Next, consider $q+1\leq i\leq r$, that is, $W$ is $H_i$-regular. Note that 
    \begin{align*}
        t(H_i, W^{G_n}) = \int_0^1\int_0^1 W_{H_i}^{G_n}(x,y)\mathrm dx\mathrm dy.
    \end{align*}
    Then once again by the boundedness property of stochastic integrals we get,
    \begin{align}
        & \E\bigg[\bigg|\hat Z'(H_i,G_n) - Y(H_i,W^{G_n})\bigg|^2 \bigg| G_n \bigg] \nonumber \\ 
        &\lesssim \int_0^1\int_0^1\left(\hat W_{H_i}^{G_n}(x,y) - W_{H_i}^{G_n}(x,y)\right)^2\mathrm dx\mathrm dy \nonumber \\
        &\lesssim_{H_i}\sum_{1\leq u\neq v\leq n}\int_{I_u\times I_v}\left(\hat W_{H_i}^{G_n}(x,y) - W_{H_i}^{G_n}(x,y)\right)^2 + \frac{1}{n} \label{eq:ZYr} \\ 
        &\lesssim_{H_i} \sum_{\substack{1\leq u\neq v\leq n\\1\leq a\neq b\leq |V(H_i)|}}\int_{I_u\times I_v}\left(\dfrac{X_{a,b}(u,v,H_i,G_n)}{n^{|V(H_i)|-2}} - W_{H_i}^{G_n}(x,y)\right)^2 + \frac{1}{n} , \label{eq:ZYWr} 
    \end{align}
    almost surely. Note that the inequality in \eqref{eq:ZYr} follows since $\hat W_{H_i}^{G_n} = 0$ on $I_u\times I_u$, for all $1\leq u\leq n$, and $W_{H_i}^{G_n}$ is bounded. Furthermore, \eqref{eq:ZYWr} follows from the definition of $\hat W_{H_i}^{G_n}$ in \eqref{eq:WGnuv}. 

Now, recalling $X_{a,b}$ from Definition \ref{def:Xabuv} and by counting arguments similar to \eqref{eq:Xav-tavO1/n} gives, 
    \begin{align*}
        \E\bigg[\bigg|\hat Z'(H_i,G_n) - Y(H_i,W^{G_n})\bigg|^2 \bigg| G_n \bigg]
        &\lesssim_{H_i}\frac{1}{n^2} ,
    \end{align*}
    almost surely. This completes the proof of Lemma \ref{lm:ZHGndistribution}. 
    \end{proof}

Now, define 
\begin{align}\label{eq:def-Yn}
    \bm{Y}(\mathcal{H},W^{G_{n}}):=\left(Y(H_{1},W^{G_{n}}),\cdots, Y(H_{r},W^{G_{n}})\right)^{\top} . 
\end{align} 
Note that although $\bm{Y}(\mathcal{H},W^{G_{n}})$ resembles the empirical analogue of $\bm{Z}(\mathcal{H},W)$ (recall \eqref{eq:ZnHWlimit}), with $W$ replaced with $W^{G_{n}}$, one important difference is that in the regular regime, that is, when $q+1\leq i\leq  r$, $Y(H_i, W^{G_{n}})$ does not have the Gaussian component $G_i$, unlike in its population counterpart $Z(H_i, W^{G_{n}})$.  

Since $\hat{\bm Z}'( \mathcal{H}, G_{n})$ has the same distribution as $\hat{\bm Z}( \mathcal{H}, G_{n})$ (recall \eqref{eq:ZHB}), Lemma \ref{lm:ZHGndistribution} implies that to prove Theorem \ref{thm:ZnHGn} it suffices to show that $\bm{Y}(\mathcal{H},W^{G_{n}})|G_n$ converges in distribution (almost surely) to $\bm{Z}(\mathcal{H},W)$. We will establish this by showing that the MGF of $\bm \alpha^\top \bm{Y}(\mathcal{H},W^{G_n})$ conditioned on the graph $G_n$ will converge the MGF of $\bm{\alpha}^\top \bm{Z}(\mathcal{H},W)$ almost surely in a neighborhood of zero, for all $\bm{\alpha} \in \mathbb R^r$. This is formalized in the following Proposition \ref{prop:MGF-convg}, which is proved in Section \ref{sec:mgfWGnpf}.

\begin{prop}\label{prop:MGF-convg} 
For any $\bm \alpha \in\mathbb{R}^{r}$ and $|\theta|<\frac{1}{32\mathcal{C}}$, where $\mathcal{C}$ is defined in Proposition \ref{prop:MGF-alpha-limit}, the following hold: 
\begin{align*}
    \lim_{n \rightarrow \infty }\log\mathbb{E}\left[ e^{\theta \bm{\alpha}^{\top}\bm{Y}(\mathcal{H},W^{G_{n}})}\middle|G_{n}\right] = \log\mathbb{E}\left[ e^{ \theta \bm{\alpha}^{\top}\bm{Z}(\mathcal{H},W) } \right] , 
\end{align*}
on a set $\mathcal{A}$ (not depending on $\bm \alpha$) such that $\mathbb P(\mathcal A) = 1$. 
\end{prop}

Proposition \ref{prop:MGF-convg} implies that 
$$ \bm{Y}(\mathcal{H},W^{G_{n}})|G_n \stackrel {D} \rightarrow \bm{Z}(\mathcal{H},W), $$
for all $\bm{\alpha} \in \R^r$, on the set $\mathcal A$. Since the above convergence  holds for all $\bm{\alpha} \in \R^r$, the result in Theorem \ref{thm:ZnHGn} follows from the Cram\'er-Wold device, Lemma \ref{lm:ZHGndistribution}, and recalling that $\hat{\bm Z}'( \mathcal{H}, G_{n})$ has the same limiting distribution as $\hat{\bm Z}( \mathcal{H}, G_{n})$.

\subsection{Proof of Proposition \ref{prop:MGF-convg} }
\label{sec:mgfWGnpf}

Let $\bm \alpha = (\alpha_{1},\alpha_{2},\cdots,\alpha_{r})^\top \in\mathbb{R}^r$. Similar to \eqref{eq:def-V} we define,
\begin{align}\label{eq:def-Vn}
    V_{\bm \alpha, G_n}(x) := \sum_{i=1}^{q}\alpha_{i}\left[\dfrac{1}{\left|\text{Aut}(H_{i})\right|}\sum_{a=1}^{|V(H_{i})|}t_{a}(x,H_{i},W^{G_{n}}) - \dfrac{|V(H_{i})|}{\left|\text{Aut}(H_{i})\right|}t(H_{i},W^{G_{n}})\right] , 
\end{align}
and, similar to \eqref{eq:def-U}, let us define, 
\begin{align}\label{eq:def-Un}
    U_{\bm \alpha, G_n}(x, y) := \sum_{i=q+1}^{r}\alpha_{i}\left(W^{G_n}_{H_{i}} (x, y)  - c(H_{i},W^{G_{n}})\right) , 
\end{align}
where $c(H_{i},W^{G_{n}}):= \frac{|V(H_{i})|(|V(H_{i})| -1) }{2|\text{Aut}(H_{i})|}t(H_{i},W^{G_{n}})$. 

\begin{lemma}\label{lemma:Unl-convg}
There exists a set $\mathcal{A}$ such that $\mathbb P(\mathcal{A}) = 1$ on which  
the following hold:
\begin{itemize}
 \item[$(1)$] For $L \geq 3$, 
 $$\lim_{n \rightarrow \infty} \int U_{\bm \alpha, G_n}^{(L)}(x, x) \mathrm{d}x = \int U_{\bm \alpha}^{(L)}(x,x)\mathrm{d}x,$$
where $U_{\bm \alpha, G_n}^{(L)}$ and $U_{\bm \alpha}^{(L)}$ are the $L$-th path composition of $U_{\bm \alpha, G_n}$ and $U_{\bm \alpha}$, respectively. 

\item[$(2)$] For $L \geq 1$, 
 $$\lim_{n \rightarrow \infty} \int V_{\bm \alpha, G_n}(x) U_{\bm \alpha, G_n}^{(L)}(x, y) V_{\bm \alpha, G_n}(y) \mathrm{d}x \mathrm {d} y = \int V_{\bm \alpha}(x) U_{\bm \alpha}^{(L)}(x, y) V_{\bm \alpha}(y) \mathrm{d}x \mathrm {d} y . $$

\end{itemize} 
\end{lemma}

The proof of Lemma \ref{lemma:Unl-convg} is given in Section \ref{sec:Unlpf}. Here, we apply this lemma to complete the proof of Proposition \ref{prop:MGF-convg}. 

To begin with note that the Proposition \ref{prop:MGF-alpha-limit}  holds for any graphon, in particular, it holds for the emprical graphon $W^{G_{n}}$. Hence, using this expression, with $W$ replaced by $W^{G_{n}}$ and $\Sigma$ replaced by zero (recalling the definition of $\bm{Y}(\mathcal{H},W^{G_{n}})$ from \eqref{eq:def-Yn}) we get,
\begin{align*}
    & \log\mathbb{E}\left[ e^{\theta \bm{\alpha}^{\top}\bm{Y}(\bm{\mathcal{H}},W^{G_n})} | G_n \right] \nonumber \\ 
    &= \dfrac{\theta^{2} c_{2, n}}{2}  + \sum_{L = 1}^{\infty} 2^{L-1}\theta^{L+2} \int_{[0, 1]^2} V_{\bm \alpha, G_n}(x)V_{\bm \alpha, G_n}(y)U_{\bm \alpha, G_n}^{(L)}(x,y)\mathrm{d}x\mathrm{d}y + \frac{1}{2}\sum_{L=3}^{\infty}\dfrac{\left(2 \theta\right)^{L}}{L} \int_{0}^{1} U_{\bm \alpha, G_n}^{(L)}(x, x)\mathrm{d}x , 
\end{align*}  
where 
\begin{align}\label{eq:cn}
c_{2, n} = \| V_{\bm \alpha, G_n} \|_2^2 + 2 \| U_{\bm \alpha, G_n} \|_2^2.
\end{align} 
Note that $|U_{\bm \alpha, G_n}|\leq \mathcal{C}$, where $\mathcal{C}$ is defined in Proposition \ref{prop:MGF-alpha-limit}. Then, for all $|\theta|<\frac{1}{32\mathcal{C}}$ and $L \geq 3$ we get,
\begin{align*}
    \frac{2^{L}|\theta|^{L}}{L}\left|\int U_{\bm \alpha, G_n}^{(L)}(x,x)\mathrm{d}x\right|
    &\leq \frac{2^{L}}{\mathcal{C}^{L}32^{L}L}\mathcal{C}^{L}\leq \frac{1}{16^{L}L}.
\end{align*}
Observe that $\sum_{L=3}^{\infty}\frac{1}{16^{L}L}<\infty$. Hence, using the Dominated Convergence Theorem and Lemma \ref{lemma:Unl-convg} we conclude,
\begin{align}\label{eq:rest-term-convg-1}
    \lim_{n \rightarrow \infty} \frac{1}{2}\sum_{L=3}^{\infty}\frac{(2 \theta)^{L}}{L}\int U_{\bm{\alpha}, G_n}^{(L)}(x,x)\mathrm{d}x = \frac{1}{2}\sum_{L=3}^{\infty}\frac{(2 \theta)^{L}}{L}\int_{0}^{1} U_{\bm \alpha}^{(L)}(x,x)\mathrm{d}x 
\end{align} 
on the set $\mathcal{A}$. 
Next, note that 
\begin{align*}
    2^{L-1}|\theta|^{L+2}\left|\int V_{\bm \alpha, G_n}(x)V_{\bm \alpha, G_n}(y)U_{\bm \alpha, G_n}^{(L)}(x,y)\mathrm{d}x\mathrm{d}y\right| \leq 2^{L-1}\frac{1}{32^{L+2}\mathcal{C}^{L+2}}\tilde{\mathcal{C}}^{2}\mathcal{C}^{L}\leq \frac{1}{16^{L-1}\mathcal{C}^{2}} \tilde{\mathcal{C}}^{2} , 
\end{align*} 
where $\tilde{\mathcal{C}} = 2\sum_{i=1}^{q}|\alpha_{i}|\dfrac{|V(H_{i})|}{|\text{Aut}(H_{i})|}$. Now, observe that $\frac{\tilde{\mathcal{C}}^{2}}{\mathcal{C}^{2}} \sum_{L=1}^{\infty}\frac{1}{16^{L-1}} <\infty$. Hence, using DCT and Lemma \ref{lemma:Unl-convg} we conclude,
\begin{align}\label{eq:rest-term-convg-2}
   & \lim_{n \rightarrow \infty} \sum_{L=1}^{\infty}2^{L-1}s^{L+2}\int V_{\bm \alpha, G_n}(x)V_{\bm \alpha, G_n}(y)U_{\bm \alpha, G_n}^{(L)}(x,y)\mathrm{d}x\mathrm{d}y \nonumber \\ 
   & = \sum_{L=1}^{\infty}2^{L-1}s^{L+2}\int V_{\bm \alpha}(x)V_{\bm \alpha}(y)U_{\bm \alpha}^{(L)}(x,y)\mathrm{d}x\mathrm{d}y , 
\end{align} 
on $\mathcal{A}$.
Now, denote $\kappa_i = \frac{\alpha_i}{|\mathrm{Aut}(H_i)|}$. Then from the proof of Lemma \ref{lemma:norm-U-2} we get, as $n \rightarrow \infty$
\begin{align}\label{eq:VGn}
    \|V_{\bm \alpha, G_n}\|_{2}^{2} & = \sum_{1 \leq i,j \leq q} \kappa_i \kappa_j \sum_{\substack{a\in V(H_{i})\\b\in V(H_{j})}}\left[t\left(H_{i}\bigoplus_{a,b}H_{j},W^{G_{n}}\right) - t(H_{i},W^{G_{n}})t(H_{j},W^{G_{n}})\right] \nonumber \\ 
    & \rightarrow \sum_{1 \leq i,j \leq q} \kappa_i \kappa_j \sum_{\substack{a\in V(H_{i})\\b\in V(H_{j})}}\left[t\left(H_{i}\bigoplus_{a,b}H_{j},W\right) - t(H_{i},W)t(H_{j},W)\right] = \eta_{\bm \alpha} ,  
\end{align}
on the set $\mathcal A$, since the vertex-join of 2 simple graphs produces another simple graph. 
Also, from the proof of Lemma \ref{lemma:norm-U-2}, with $W^{G_{n}}$ in place of $W$ (see in \eqref{eq:U2-first} and \eqref{eq:WhiWhj-integral}) 
\begin{align*}
    \|U_{\bm \alpha, G_n}\|_{2}^{2} = \sum_{q+1 \leq i,j \leq r} \dfrac{ \kappa_i \kappa_j }{4} \sum_{\substack{a\neq b\in V(H_{i})\\c\neq d\in V(H_{j})}}t\left(H_{i}\bigoplus_{(a,b),(c,d)}H_{j},W^{G_{n}}\right) - \left(\sum_{i=1}^{q}\alpha_{i}c(H_{i},W^{G_{n}})\right)^{2}
\end{align*}
Now, since $W^{G_{n}}$ is the empirical graphon corresponding to the graph $G_{n}$, $W^{G_{n}}$ is $\{0,1\}$-valued. Thus, $W^{G_{n}}(x, y)^2= W^{G_{n}} (x, y)$, for all $x, y \in [0, 1]$,  hence, 
\begin{align*}
    t\left(H_{i}\bigoplus_{(a,b),(c,d)}H_{j},W^{G_{n}}\right) = t\left(H_{i}\bigominus_{(a,b),(c,d)}H_{j},W^{G_{n}}\right) , 
\end{align*}
for $a\neq b\in V(H_{i})$, $c\neq d\in V(H_{j})$, and $q+1\leq i,j\leq r$. Thus, on the set $\mathcal A$, 
\begin{align}\label{eq:UGn}
    2 \|U_{\bm \alpha, G_n}\|_{2}^{2} & = \sum_{ q+1 \leq i,j \leq r} \dfrac{\kappa_i \kappa_j}{2}\sum_{\substack{a\neq b\in V(H_{i})\\c\neq d\in V(H_{j})}}t\left(H_{i}\bigominus_{(a,b),(c,d)}H_{j},W^{G_{n}}\right) - \left(\sum_{i=1}^{q}\alpha_{i}c(H_{i},W^{G_{n}})\right)^{2} \nonumber \\ 
       & \rightarrow \sum_{ q+1 \leq i,j \leq r} \dfrac{\kappa_i \kappa_j}{2}\sum_{\substack{a\neq b\in V(H_{i})\\c\neq d\in V(H_{j})}}t\left(H_{i}\bigominus_{(a,b),(c,d)}H_{j},W \right) - \left(\sum_{i=1}^{q}\alpha_{i}c(H_{i},W)\right)^{2} 
       = \tilde{\eta}_{\bm \alpha} , 
\end{align} 
as the weak edge-join of 2 simple graphs produces a simple graph. 
Hence, recalling \eqref{eq:cn}, \eqref{eq:VGn}, and \eqref{eq:UGn} gives, $c_{2, n} \rightarrow \eta_{\bm \alpha} + \tilde{\eta}_{\bm \alpha}$. Combining this with \eqref{eq:rest-term-convg-1}, \eqref{eq:rest-term-convg-2}, and Proposition \ref{prop:MGF-alpha-limit}, the result in Proposition \ref{sec:mgfWGnpf} follows.

\subsection{Proof of Lemma \ref{lemma:Unl-convg}} 
\label{sec:Unlpf}

We begin by recalling the definitions of {\it cut-distance} and {\it cut-metric}. 

\begin{definition}\label{defn:Wconvergence} \cite[Chapter 8]{lovasz2012large}
The {\it cut-distance} between two bounded functions $W_1, W_2 : [0, 1]^2 \rightarrow \mathbb R$ is 
\begin{align*}
||W_1-W_2||_{\square}:=\sup_{f, g: [0, 1] \rightarrow [0, 1]} \left|\int_{[0, 1]^2} \left(W_1(x, y)-W_2(x, y)\right) f(x) g(y) \mathrm dx \mathrm dy \right|. 
\end{align*} 
The {\it cut-metric} between $W_1, W_2$ is defined as,  
\begin{align*}
\delta_{\square}(W_1, W_2):= \inf_{\psi}||W_1^{\psi}-W_2||_{\square}, 
\end{align*} 
with the infimum taken over all measure-preserving bijections $\psi: [0, 1] \rightarrow [0, 1]$, and  $W_1^\psi(x, y):= W_1(\psi(x), \psi(y))$, for $x, y \in [0, 1]$. 
\end{definition}

By \cite[Lemma 10.16]{lovasz2012large} we known that $\delta_{\square}(W^{G_n}, W) \rightarrow 0$ almost surely. 

\subsubsection{Proof of Lemma \ref{lemma:Unl-convg} (1)} 

Recall from \eqref{eq:def-Un}, 
\begin{align*}
    U_{\bm \alpha, G_n}(x, y) := \sum_{i=q+1}^{r}\alpha_{i}\left(W^{G_n}_{H_{i}} (x, y)  - c(H_{i},W^{G_{n}})\right) , 
\end{align*}
where $c(H_i, W) = \frac{|V(H_{i})|(|V(H_{i})| -1) }{2|\text{Aut}(H_{i})|}t(H_{i},W^{G_{n}})$. Denote $\nu:=\max_{1 \leq i \leq r} |V(H_i) (|V(H_i)| -1) $. Then, by the counting lemma (see \cite[Theorem 3.7]{borgs2008convergent}), 
for $1 \leq i \leq r$, 
$$|c(H_{i},W^{G_{n}})- c(H_{i},W)|  = \frac{|V(H_i) (|V(H_i)| -1)}{2 |\mathrm{Aut}(H_i)|} |t(H_i, W^{G_{n}}) - t(H_i, W)| \lesssim_{\nu} \| W^{G_{n}} - W \|_{\square} . $$
Hence, given $f, g: [0, 1] \rightarrow [0, 1]$, by the triangle inequality, 
\begin{align}\label{eq:UGnxy}
& \left|\int f(x) (U_{\bm \alpha, G_n}(x, y) - U_{\bm \alpha}(x, y)) g(y) \mathrm d x \mathrm d y \right| \nonumber \\ 
& \lesssim_{\nu} \sum_{i=q+1}^r |\alpha_i| \left(\left|\int f(x) (W^{G_n}_{H_{i}} (x, y) - W_{H_{i}} (x, y)) g(y) \mathrm d x \mathrm d y \right| + \| W^{G_{n}} - W \|_{\square} \right) . 
\end{align}
Now, by a telescoping argument, replacing $W^{G_n}$ with $W$ one at a time, as in the proof of the counting lemma \cite[Theorem 3.7]{borgs2008convergent}, it can be shown that 
$$\left|\int f(x) (W^{G_n}_{H_{i}} (x, y) - W_{H_{i}} (x, y)) g(y) \mathrm d x \mathrm d y \right| \lesssim_{\nu} \| W^{G_n} - W \|_{\square} . $$
Hence, from \eqref{eq:UGnxy}, 
\begin{align}\label{eq:Uxy}
\| U_{\bm \alpha, G_n} - U_{\bm \alpha}\|_{\square} \lesssim_{\nu} \| W^{G_n} - W \|_{\square}
\end{align}
Again, by the counting lemma (adapted to general bounded functions) we have, for $L \geq 3$, 
\begin{align}\label{eq:UalphaGn}
\left| \int U_{\bm \alpha, G_n}^{(L)}(x, x) \mathrm d x -  \int U_{\bm \alpha}^{(L)}(x, x) \mathrm d x \right| \lesssim_{L} \| U_{\bm \alpha, G_n} - U_{\bm \alpha}\|_{\square} \lesssim \| W^{G_n} - W\|_\square, 
\end{align}
where the last inequality uses \eqref{eq:Uxy}. Since $\int U_{\bm \alpha}^{(L)}(x, x) \mathrm d x$ is invariant to measure preserving transformations of $W$ and $\delta_{\square}(W^{G_n}, W) \rightarrow 0$ almost surely, as $n \rightarrow \infty$, from \eqref{eq:UalphaGn} the result in Lemma \ref{lemma:Unl-convg} (1) follows.

\subsubsection{Proof of Lemma \ref{lemma:Unl-convg} (2)} 

Recall from \eqref{eq:def-Vn}, 
\begin{align*}
    V_{\bm \alpha, G_n}(x) := \sum_{i=1}^{q}\alpha_{i}\left[\dfrac{1}{\left|\text{Aut}(H_{i})\right|}\sum_{a=1}^{|V(H_{i})|} t_{a}(x,H_{i},W^{G_{n}}) - \dfrac{|V(H_{i})|}{\left|\text{Aut}(H_{i})\right|}t(H_{i},W^{G_{n}})\right] .  
\end{align*} 
Note that $ t(H_{i},W^{G_{n}}) \rightarrow t(H_{i},W)$, for all $1 \leq i \leq r$, and 
$$\lim_{n \rightarrow \infty} \int U_{\bm \alpha, G_n}^{(L)}(x, y) \mathrm{d}x \mathrm {d} y = \int U_{\bm \alpha}^{(L)}(x, y) \mathrm{d}x \mathrm {d} y , $$ 
by \eqref{eq:Uxy}, almost surely. Hence, to establish Lemma \ref{lemma:Unl-convg} (2) it suffices to show the following for $1 \leq i, j\leq q$, 
\begin{align}\label{eq:UVtab}
\lim_{n \rightarrow \infty} & \int t_{a}(x,H_{i},W^{G_{n}}) U_{\bm \alpha, G_n}^{(L)}(x, y) t_{b}(y,H_{j},W^{G_{n}}) \mathrm{d}x \mathrm {d} y  \nonumber \\ 
&= \int t_{a}(x,H_{i},W) U_{\bm \alpha}^{(L)}(x, y) t_{b}(y,H_{j},W) \mathrm{d}x \mathrm {d} y , 
\end{align} 
and 
\begin{align}\label{eq:UVta}
& \lim_{n \rightarrow \infty} \int t_{a}(x,H_{i},W^{G_{n}}) U_{\bm \alpha, G_n}^{(L)}(x, y)  \mathrm{d}x \mathrm {d} y = \int t_{a}(x,H_{i},W) U_{\bm \alpha}^{(L)}(x, y)  \mathrm{d}x \mathrm {d} y .  
\end{align} 
where $a \in V(H_i)$ and $b \in V(H_j)$.

We will prove \eqref{eq:UVtab}. The proof of \eqref{eq:UVta} follows similarly. For this, note that by a telescoping argument, 
\begin{align} 
 \int t_{a}(x,H_{i},W^{G_{n}}) (U_{\bm \alpha, G_n}^{(L)}(x, y) - U_{\bm \alpha}^{(L)}(x, y)) t_{b}(y,H_{j},W^{G_{n}}) \mathrm{d}x \mathrm {d} y 
 & \lesssim \| U_{\bm \alpha, G_n} - U_{\bm \alpha} \|_{\square} \nonumber \\
 & \lesssim \| W^{G_{n}} - W \|_{\square}\nonumber, 
\end{align} 
where the last step uses \eqref{eq:Uxy}. Hence, to prove \eqref{eq:UVtab} it suffices to show that 
\begin{align} \label{eq:tabUL}
\lim_{n \rightarrow \infty }\int t_{a}(x,H_{i},W^{G_{n}}) U_{\bm \alpha}^{(L)}(x, y) t_{b}(y,H_{j},W^{G_{n}}) \mathrm{d}x \mathrm {d} y .
\end{align} 
Consider the functions (not necessarily symmetric) $B_{n}(x, y) : = t_{a}(x,H_{i},W^{G_{n}}) t_{b}(y,H_{j},W^{G_{n}})$ and $B(x, y) : = t_{a}(x,H_{i},W) t_{b}(y,H_{j},W)$. By a telescoping argument it can shown that 
$$\| B_{n} - B \|_{\square} \lesssim \| W^{G_{n}} - W \|_{\square}.$$
The result in \eqref{eq:tabUL} then follows from \cite[Lemma 8.22]{lovasz2012large}.

\section{Proofs from Section \ref{sec:LHW}}

\subsection{Proof of Proposition \ref{prop:H01}} \label{sec:H01pf} 

To prove Proposition \ref{prop:H01} we first replace with $R(H,G_n)$ by 
\begin{align*}
    R\left(H, W^{G_n}\right) = \sum_{1\leq a,b\leq |V(H)|}t\left(H\bigoplus_{a,b}H,W^{G_n}\right)-|V(H)|^2t(H,W^{G_n})^2 , 
\end{align*}
where $t(\cdot, W^{G_n})$ is defined in \eqref{eq:tHWGn}. For a finite subgraph $F = (V(F), E(F))$, recalling \eqref{eq:tHWGn} and \eqref{eq:XHWGn} notice,
\begin{align}\label{eq:tGntWGn}
    \left|\hat{t}(F,G_n) - t(F,W^{G_n})\right| 
    &= \left|\frac{1}{n^{|V(F)|}}\sum_{\bm{s} \in [n]^{|V(F)|}}\prod_{(i,j)\in E(F)}w_{s_is_j} - \frac{1}{(n)_{|V(F)|}}\sum_{ \bm{s} \in ([n])_{|V(F)|}}\prod_{(i,j)\in E(F)}w_{s_is_j}\right|\nonumber\\
    &\leq \frac{1}{n^{|V(F)|}}\left|\sum_{\bm{s} \in [n]^{|V(F)|}\setminus([n])_{|V(F)|}}\prod_{(i,j)\in E(F)}w_{s_is_j}\right| + O\left(\frac{1}{n}\right) = O\left(\frac{1}{n}\right).
\end{align}
This implies,
\begin{align}\label{eq:RHGnWGn}
    \left|R(H,G_n) - R(H,W^{G_n})\right| = O\left(\frac{1}{n}\right).
\end{align} 

Notice that, by definition, $R(H,W)\geq 0$ (see \eqref{eq:defRHW}) for any graphon $W$ and in particular for the empirical graphon $W^{G_{n}}$. Now, we consider the following 2 cases: 

\begin{itemize} 

\item \textit{$W$ is $H$-regular:} Recalling that $R(H,W^{G_{n}})\geq 0$ it is now enough to show that 
\begin{align}\label{eq:RHGnexpectation}
\mathbb{E}R(H,W^{G_{n}})=O(1/n).
\end{align}
Towards that, recalling \eqref{eq:RGn} note that,
\begin{align}\label{eq:exprERHWn}
    & \E\left[R(H,W^{G_{n}})\right] \nonumber \\ 
    & = \sum_{ 1 \leq a,b \leq |V(H)| }\E\left[t\left(H\bigoplus_{a,b}H,W^{G_{n}}\right)\right]-|V(H)|^2\E\left[t(H \bigsqcup H,W^{G_{n}})\right] , 
\end{align}
where $H \bigsqcup H$ is the disjoint union of two copies of $H$. Lemma 2.1 and Lemma 2.4 in \cite{lovasz2006limits} implies that 
\begin{align*}
	\mathbb{E}\left[t(H \bigsqcup H,W^{G_{n}})\right]
	&\geq t(H \bigsqcup H,W)-\frac{1}{n}{2|V(H)|\choose 2} \nonumber \\ 
	& =t(H,W)^{2}-\frac{1}{n}{2|V(H)|\choose 2}
\end{align*}
and 
\begin{align*}
	& \sum_{ 1 \leq a,b \leq |V(H)| }\mathbb{E}\left[t\left(H\bigoplus_{a,b}H,W^{G_{n}}\right)\right] \nonumber \\ 
	& \leq \sum_{ 1 \leq a,b \leq |V(H)| }\left[t\left(H\bigoplus_{a,b}H,W\right)+\frac{1}{n}{2|V(H)|-1\choose 2}\right] .
\end{align*}
Substituting these bounds in \eqref{eq:exprERHWn} give, 
\begin{align*}
	\mathbb{E}R(H,W^{G_{n}})
	&\leq \sum_{a,b=1}^{|V(H)|}t\left(H\bigoplus_{a,b}H,W\right)-|V(H)|^2t(H,W)^2 + O(1/n)\\
    & = R(H,W) + O(1/n) , 
\end{align*}
where the last equality follows by recalling the definition of $R(H,W)$ from \eqref{eq:defRHW}. The proof of Proposition \ref{prop:H01} (1) is now complete by noticing that $R(H,W) = 0$ when $W$ is $H$-regular. 

\item \textit{$W$ is $H$-irregular:} From Corollary 10.4 in \cite{lovasz2012large} we know that $R(H,W^{G_{n}}) \overset{P}{\rightarrow} R(H,W)$. Since $R(H,W)>0$ whenever $W$ is $H$-irregular, this implies $\sqrt{n}R(H,W^{G_n})\overset{P}{\rightarrow}\infty$. This completes the proof Proposition \ref{prop:H01} (2). 
\end{itemize}

\begin{remark} 
Note that \eqref{eq:RHGnexpectation} and \eqref{eq:RHGnWGn} implies $R(H,G_{n}) = O_P(n)$, when $W$ is $H$-regular. Hence, $a_n R(H,G_{n}) \overset{P}{\rightarrow} 0$, for any sequence $\{a_n\}_{n \geq 1}$ such that $a_n \rightarrow \infty$ and $a_n/n \rightarrow 0$. While choosing $a_n = \sqrt n$ shows Proposition \ref{prop:H01} (1), the results in Section \ref{sec:LHW} would continue to hold whenever $a_n/n \rightarrow 0$. 
\end{remark}

\subsection{Proof of Theorem \ref{thm:LHW}}\label{sec:confidencesetpf} 

Without loss of generality assume that $W$ is irregular with respect to $H_{1}, H_{2}, \ldots, H_{q}$ and regular with respect to $H_{q+1}, H_{q+2}, \ldots, H_{r}$. To proceed with the proof we first show that the distribution of $\tilde{\bm Z}(\cH, G_n)$ (recall \eqref{eq:ZHconfidenceset}) converges to $\bm Z(\cH, W)$. 

\begin{lemma}\label{lemma:Ztildeconvg}
    Let $\tilde{\bm Z}(\cH, G_{n})$ and $\bm Z(\cH,W)$ be as defined in \eqref{eq:ZHconfidenceset} and  \eqref{eq:ZnHWlimit}, respectively. Then, as $n\ra\infty$,
    \begin{align}\label{eq:ZSHGn}
        \tilde{\bm Z}(\cH, G_{n})\dto \bm Z(\cH,W).
    \end{align}
\end{lemma} 

\begin{proof}
Recall that $S(\cH, G_n) = \{ 1 \leq i \leq r: \sqrt n R(H,G_{n})  > 1 \} $ is the set of indices where the hypothesis of $H_i$-regularity is rejected. 
Define the event 
\begin{align}\label{eq:defcSn}
    \mathcal{D}_{n} := \{S(\cH, G_n) = \{1, 2, \ldots, q\} \} . 
\end{align}

Now, by Proposition \ref{prop:H01} notice that,
\begin{align}\label{eq:Snco1}
    \mathbb P(\mathcal D_n^c) 
    &\leq \sum_{i=1}^{q}\mathbb P\left(i\not\in S(\cH,G_n)\right) + \sum_{i=q+1}^{r}\mathbb P\left(i\in S(\cH,G_n)\right)\nonumber\\
    & = \sum_{i=1}^{q}\mathbb P\left(\sqrt{n}R(H,G_{n})\leq 1\right) + \sum_{i=q+1}^{r}\mathbb P\left(\sqrt{n}R(H,G_{n})> 1\right) = o(1) , 
\end{align}
since $W$ is irregular with respect to $H_{1}, H_{2}, \ldots, H_{q}$ and regular with respect to $H_{q+1}, H_{q+2}, \ldots, H_{r}$. Hence, to prove \eqref{eq:ZSHGn} it is enough to show that the characteristic functions converge. For this, note that for any $\bm t\in \R^r$,
\begin{align*}
    \E\left[ e^{\iota \bm t^{\top}\tilde{\bm Z}(\cH, G_n) } \right] = \E\left[ e^{\iota \bm t^{\top}\tilde{\bm Z}(\cH, G_n)} \one\left\{\mathcal{D}_n\right\}\right] + o(1) . 
\end{align*}
Note that on the event $\mathcal{D}_n$, $\tilde{\bm Z}_n(\cH,G_n) = \bm Z(\cH, G_n)$ (recall \eqref{eq:graphW}). Therefore, \eqref{eq:Snco1} and Theorem \ref{thm:asymp-joint-dist} gives,
\begin{align*}
    \E\left[ e^{\iota \bm t^{\top}\tilde{\bm Z}(\cH, G_n)} \right] 
     = \E\left[ e^{\iota \bm t^{\top}\bm Z(\cH, G_n)}\one\left\{\mathcal{D}_n\right\}\right] + o(1) 
    & = \E\left[ e^{\iota \bm t^{\top}\bm Z(\cH, G_n) } \right] + o(1) \\ 
    & \rightarrow \E\left[ e^{\iota \bm t^{\top}\bm Z(\cH, W) } \right] . 
\end{align*} 
This completes the proof of Lemma \ref{lemma:Ztildeconvg}.

\end{proof}

Now, we show that conditional on $G_n$ the distribution of $\bm Q(\cH, G_n)$ converges to $\bm Z(\cH, W)$ as well.

\begin{lemma}\label{lemma:convgQ}
    Let $\bm Q(\cH, G_n)$ and $\bm Z(\cH,W)$ be as defined in  \eqref{eq:def-QWGn} and \eqref{eq:ZnHWlimit}, respectively. Then, as $n\ra\infty$, 
    \begin{align*}
        \bm Q(\cH, G_n)\mid G_n\dto \bm Z(\cH,W) , 
    \end{align*}
    in probability.
\end{lemma}

\begin{proof}
Recall the event $\mathcal{D}_n$ from \eqref{eq:defcSn}. It follows from \eqref{eq:Snco1} that,
\begin{align}\label{eq:Sncop1}
    \one\left\{\mathcal{D}_n^c\right\} = o_P(1).
\end{align}
Then for any $\bm t\in \R^r$,
\begin{align*}
    \E\left[e^{\iota \bm t^{\top}\bm Q(\cH, G_n)} \middle| G_n \right] = \E\left[e^{\iota \bm t^{\top}\bm Q(\cH, G_n)}\middle|G_n\right]\one\left\{\mathcal{D}_n\right\} + o_P(1) . 
\end{align*} 
As before, on the event $\mathcal{D}_n$, we have $\bm Q(\cH,G_n) = \hat{\bm Z}(\cH, G_n)$, where $\hat{\bm Z}(\cH, G_n)$ is defined in \eqref{eq:ZHGnestimate}. Hence, by \eqref{eq:Sncop1} and Theorem \ref{thm:ZnHGn}, 
\begin{align*}
    \E\left[e^{\iota \bm t^{\top}\bm Q(\cH, G_n)} \middle| G_n \right] 
     = \E\left[ e^{ \iota \bm t^{\top}\hat{\bm Z}(\cH, G_n) } \middle|G_n\right]\one\left\{\mathcal{D}_n\right\} + o_P(1) & = \E\left[ e^{ \iota \bm t^{\top}\hat{\bm Z}(\cH, G_n) } \middle|G_n\right] + o_P(1) \\ 
     & \overset{P} \rightarrow \E\left[ e^{ \iota \bm t^{\top} \bm Z(\cH, W) } \right] . 
\end{align*} 
This completes the proof of Lemma \ref{lemma:convgQ}.

\end{proof}
We now proceed to complete the proof of Theorem \ref{thm:LHW}. To this end, using the continuous mapping theorem along with Lemma  \ref{lemma:Ztildeconvg} and Lemma \ref{lemma:convgQ} gives,
\begin{align*}
    \left\|\tilde{\bm Z}(\cH, G_{n})\right\|_{2}\dto \left\|\bm Z(\cH,W)\right\|_{2}\text{ and }\left\|\bm Q(\cH, G_n)\right\|_{2}\mid G_{n}\dto \left\|\bm Z(\cH,W)\right\|_{2}\text{ in probability.}
\end{align*}
Hence, recalling \eqref{eq:Hconfidenceset} and by Polya's Theorem, 
\begin{align*}
   \mathbb P( C(\cH, G_n) ) =  \mathbb P\left(\left\|\tilde{\bm Z}(\cH, G_{n})\right\|_{2}\leq\hat{q}_{1-\alpha,\cH,G_{n}}\right)  \rightarrow 1-\alpha , 
   \end{align*}
where $\hat{q}_{1-\alpha,\cH,G_{n}}$ is the $(1-\alpha)$-quantile of the distribution of $\left\|\bm Q(\cH, G_n)\right\|_{2}\mid G_{n}$. \hfill $\Box$

\section{Proofs from Section \ref{sec:structure}} 
\label{sec:structurepf}

This section is organized as follows: In Section \ref{sec:structureH0pf}
we prove Proposition \ref{ppn:C4}. Proposition \ref{prop:consistency}  is proved in Section \ref{sec:structureH1pf}. In Section \ref{sec:structureGn} we derive the distribution of $\hat f(G_n)$ under the alternative. 
\subsection{Proof of Proposition \ref{ppn:C4}}
\label{sec:structureH0pf}

For notational convenience define $$T_{n}^{K_2} := \tfrac{1}{2} \hat{t}(K_{2}, G_n) \quad \text{ and } \quad T_{n}^{C_4} := \tfrac{1}{8} \hat{t}(C_{4}, G_n) , $$ 
where $\hat{t}(H, G_n)$ is defined in \eqref{eq:tHGncount}. Recalling \eqref{eq:definitionfGn}, note that 
\begin{align}
\label{eq:fGn}
\hat f(G_n) = h(T_{n}^{K_2}, T_{n}^{C_4}),
\end{align} where $h(x,y) = 16x^4 - 8y$. 

We begin by deriving the joint distribution of $(T_{n}^{K_2},T_{n}^{C_4})$ under $H_0$ as in \eqref{eq:structureH0Wp}. Note that under $H_0$, the $G_n$ is distributed as an Erd\H{o}s-R\'enyi random graph $G(n, p)$, for $p = t(K_{2}, W) \in (0, 1)$. Hence, recalling Example \ref{example:randomgraph}, in particular from \eqref{eq:ErdosRenyidistconvg} and \eqref{eq:ErdosRenyicovmat} we get, 
\begin{align}\label{eq:bmTnconvg}
    n \begin{pmatrix}
    T_{n}^{K_2} - \tfrac{1}{2} p \\ 
    T_{n}^{C_4} - \tfrac{1}{8} p^4 
    \end{pmatrix}
    \dto N_2(0,\Sigma), \quad \text{ where }\Sigma  = \tfrac{1}{2}p(1-p)
    \begin{pmatrix}
    1 & p^{3}\\
    p^{3} & p^{6}
    \end{pmatrix} .
\end{align}
Now, a Taylor expansion of the function $h$ around the point $(\tfrac{1}{2} p,\tfrac{1}{8} p^4)$ gives,
\begin{align*}
    h(T_{n}^{K_2},T_{n}^{C_4}) - h(\tfrac{1}{2} p,\tfrac{1}{8} p^4) = \nabla h(\tfrac{1}{2} p,\tfrac{1}{8} p^4)^{\top}\bm{T}_{n} + \bm{T}_{n}^{\top}\nabla^{2}h(\tfrac{1}{2} p,\tfrac{1}{8} p^4)\bm{T}_{n}
    + o_P\left(\|\bm{T}_{n}\|^{2}\right) , 
\end{align*}
where
\begin{align}\label{eq:defTn}
    \bm T_n := \left(T_{n}^{K_2} - \tfrac{1}{2} p,T_{n}^{C_4} - \tfrac{1}{8} p^4\right)^{\top} ,
\end{align}
and $\nabla h, \nabla^2 h$ denote the gradient and hessian of $h$ evaluated at the corresponding points. By definition, $h(\tfrac{1}{2} p, \tfrac{1}{8} p^4) = 0$. Hence, a direct computation of $\nabla^{2}h(\tfrac{1}{2} p,\tfrac{1}{8} p^4)$ along with the convergence in \eqref{eq:bmTnconvg} gives,
\begin{align}\label{eq:gexpandfinal}
    h(T_{n}^{K_2},T_{n}^{C_4}) = \nabla h(\tfrac{1}{2} p,\tfrac{1}{8} p^4)^{\top}\bm{T}_{n} + O_P\left( \frac{1}{n^{2}} \right) .
\end{align}
Since by definition $\hat f(G_n) = h(T_{n}^{K_2},T_{n}^{C_4})$ (recall \eqref{eq:fGn}), the result in Proposition \ref{ppn:C4} follows from \eqref{eq:gexpandfinal} and the following lemma: 

\begin{lemma}\label{lemma:convgderivativeg}
    Under $H_0$ as in \eqref{eq:structureH0Wp},
    \begin{align*}
        n^\frac{3}{2}\nabla h(\tfrac{1}{2} p,\tfrac{1}{8} p^4)^{\top}\bm{T}_{n}\dto N\left(0, 32p^6(1-p)^2\right) , 
    \end{align*}
    where $\bm T_n$ is defined in \eqref{eq:defTn}.
\end{lemma}

\subsubsection{Proof of Lemma \ref{lemma:convgderivativeg}} 

Notice that $\nabla h(\tfrac{1}{2} p,\tfrac{1}{8} p^4)^{\top} = (8p^{3},-8)$. Hence, recalling the definitions of $\hat t(K_2, G_n)$ and $\hat t(C_4, G_n)$ from \eqref{eq:tHGncount} gives,
\begin{align}\label{eq:Tn}
    \nabla h(\tfrac{1}{2} p,\tfrac{1}{8} p^4)^{\top}\bm{T}_{n} = \frac{8p^{3}}{(n)_{2}}\left(X(K_{2}, G_n) - \frac{(n)_{2}}{2}p\right) - \frac{8}{(n)_{4}}\left(X(C_{4}, G_n) - \frac{(n)_{4}}{8}p^{4}\right) , 
\end{align} 
where $(n)_{2} = n(n-1)$ and $(n)_{4} = n(n-1)(n-2)(n-3)$. 
We will now compute the orthogonal decomposition of \eqref{eq:Tn} using the framework described in Section \ref{sec:generalized_U}. 

To this end, recall the definition of $S_{n,\cdot}(\cdot)$ from \eqref{eq:def-Sn}. Then using \eqref{eq:Snfcount} we can rewrite \eqref{eq:Tn} as, 
\begin{align}\label{eq:nablagTSn}
    \nabla h(\tfrac{1}{2} p,\tfrac{1}{8} p^4)^{\top}\bm{T}_{n} = \frac{8p^{3}}{(n)_{2}}S_{n,2}(f) - \frac{8}{(n)_{4}}S_{n,4}(g) , 
\end{align}
where the functions $f$ and $g$ are defined as follows: 
\begin{itemize}
\item   $f(Y_{12}) = \one\left\{Y_{12}\leq p\right\} - p$. 
\item $g(\{Y_{ij}: (i, j) \in E(K_4)\}) = \sum_{i=1}^{3}\prod_{e\in E(G_i)}\one\left\{Y_e\leq p\right\} - 3p^4$, 
where $G_1,G_2,G_3$ are cycles of length four with edges $E(G_1) = \{(1,2),(2,3),(3,4),(4,1)\}$, $E(G_2) = \{(1,2),(2,4),(4,3),(3,1)\}$ and $E(G_3) = \{(1,3),(3,2),(2,4),(4,1)\}$, respectively, and $\{Y_{ij} = Y_{ji}: 1\leq i<j\leq 4\}$ are independently generated from $U[0,1]$. 
\end{itemize} 

Now, recalling \eqref{eq:SnfGR} and following \cite[Example 2]{janson1991asymptotic} we get,
\begin{align}\label{eq:Sn2tilde}
    S_{n,2}(f) = \frac{1}{2}\tilde{S}_{n,2}(f_{K_{2}}),
\end{align}
and 
\begin{align}\label{eq:Sn4tilde}
    S_{n,4}(g) = \frac{1}{4}\tilde S_{n,4}(g_{K_{2}}) + \frac{1}{6}\tilde S_{n,4}(g_{K_{3}}) + \frac{1}{2}\tilde S_{n,4}(g_{K_{1, 2}}) + \sum_{G\in \Gamma_{4}}\frac{\tilde S_{n,4}(g_{G})}{|\Aut{(G)}|} .
\end{align}
where $f_G$ and $g_G$ are the projections of $f$ and $g$ on the subspace $M_G$ (recall \eqref{eq:def-MG}). A direct counting argument gives, $\tilde{S}_{n,2}(f_{K_{2}}) = \frac{(n-4)!}{(n-2)!}\tilde S_{n,4}(f_{K_{2}})$. Hence, from \eqref{eq:Sn2tilde} and \eqref{eq:Sn4tilde}, the RHS of \eqref{eq:nablagTSn} can be rewritten as, 
\begin{align}\label{eq:gradgTnexpr}
    & \nabla h(\tfrac{1}{2} p,\tfrac{1}{8} p^4)^{\top}\bm{T}_{n} \nonumber \\ 
    & = \frac{8}{(n)_{4}} \left( \tilde S_{n,4}\left( \frac{1}{2} p^{3} f_{K_{2}} -  \frac{1}{4} g_{K_{2}} \right)  -  \frac{1}{6}\tilde S_{n,4}( g_{K_{3}}) - \frac{1}{2}\tilde S_{n,4}( g_{K_{1, 2}}) - \sum_{G\in \Gamma_{4}}\frac{\tilde S_{n,4}( g_{G})}{|\Aut{(G)}|} \right) .
\end{align} 
Using \eqref{eq:PL2-decomp} we now compute the differents projections. To begin with, note that 
$$f_{K_{2}} = (\one\left\{Y_{12}\leq p\right\}-p) \text{ and } g_{K_{2}} = 2p^{3}(\one\left\{Y_{12}\leq p\right\}-p).$$ This shows that $\frac{1}{2}p^{3} f_{K_{2}} - \frac{1}{4} g_{K_{2}}=0$. Also, by computations similar to those in \eqref{eq:fE12asEU} and \eqref{eq:fK12asEU}, it can shown that $g_{K_{3}} = 0$. Further, by  \cite[Lemma 4]{janson1991asymptotic},
\begin{align*}
    \frac{8}{(n)_{4}}\sum_{G\in \Gamma_{4}}\frac{\tilde S_{n,4}(g_{G})}{|\Aut{(G)}|} = o_P(n^{-\frac{3}{2}}).
\end{align*}
Hence, \eqref{eq:gradgTnexpr} can be simplified as follows,
\begin{align}\label{eq:hTn}
    \nabla h(\tfrac{1}{2} p,\tfrac{1}{8} p^4)^{\top}\bm{T}_{n} = \frac{4}{(n)_{4}}\tilde S_{n,4}(g_{K_{1, 2}}) + o_P(n^{-\frac{3}{2}}).
\end{align}
Once again, computing the projection of $g$ on $M_{K_{1, 2}}$ using \eqref{eq:PL2-decomp} we get, 
\begin{align}\label{eq:gK12}
g_{K_{1, 2}} = p^{2}(\one\left\{Y_{12}\leq p\right\}-p)(\one\left\{Y_{13}\leq p\right\}-p). 
\end{align}
Using \eqref{eq:hTn} and \eqref{eq:gK12} together with the distributional convergence result in \cite[Lemma 7]{janson1991asymptotic}, we have    
\begin{align*}
    n^\frac{3}{2}\nabla h(\tfrac{1}{2} p,\tfrac{1}{8} p^4)^{\top}\bm{T}_{n} = n^\frac{3}{2}\frac{4}{(n)_{4}}\tilde S_{n,4}(g_{K_{1, 2}}) + o_P(1)\dto N\left(0,32p^{6}(1-p)^{2}\right) . 
\end{align*}
This completes the proof of Lemma \ref{lemma:convgderivativeg}. \hfill $\Box$

\subsection{Proof of Proposition \ref{prop:consistency}} 
\label{sec:structureH1pf}

Recall the definition of $\hat f(G_n)$ from \eqref{eq:definitionfGn}. Then, since $\hat t(H, G_n) \overset{P} \rightarrow t(H, W)$ for every fixed graph $H$, it follows that $$\hat f(G_n) = \hat t(K_2,G_n)^4 - \hat t(C_4,G_n) \prob f(W) = t(K_2,W)^4 - t(C_4,W).$$ The quasi-randomness result of Chung, Graham, and Wilson \cite[Theorem 1]{chung1989cycle}, formulated in terms of graphons (see \cite[Section 11.8]{lovasz2012large}), implies that $f(W) \ne 0$, for any graphon $W$ that is not constant almost everywhere. Hence, under $H_1$,  
\begin{align*}
    \dfrac{ \hat f (G_n)}{4\sqrt{2}\ \hat{t}(K_{2},G_n)^{3}(1- \hat{t}(K_{2},G_n))}\prob \dfrac{f(W)}{4\sqrt{2}t(K_2,W)^3(1-t(K_2,W))}>0 . 
\end{align*}
Hence, recalling \eqref{eq:defTnstructure}, $T_n\prob\infty$, which now completes the proof.

\subsection{Distribution of $\hat f(G_n)$ under the Alternative}
\label{sec:structureGn}

In this section we apply Theorem \ref{thm:asymp-joint-dist} to derive the limiting of $\hat f(G_n)$ under the alternative.

\begin{prop}\label{prop:H0alternative} Suppose $|t(K_{2},W)^{4} - t(C_{4},W)|>0$. Then the following hold: 

\begin{itemize} 

\item If $W$ is irregular with respect to $K_2$ and $C_4$, then 
$$\sqrt n (\hat f(G_n) - f(W) ) \stackrel{D} \rightarrow N( 0 , \tau_1^2), $$ 
where $\tau_1^2$ is defined in \eqref{eq:varianceirregular}.  

\item If $W$ is irregular with respect to $K_2$ and regular with respect to $C_4$, then 
$$\sqrt n (\hat f(G_n) - f(W) ) \stackrel{D} \rightarrow N( 0 , \tau_2^2), $$ 
where $\tau_2^2$ is defined in \eqref{eq:varianceirregularregular}.

\item If $W$ is regular with respect to $K_2$ and irregular with respect to $C_4$, then 
$$\sqrt n (\hat f(G_n) - f(W) ) \stackrel{D} \rightarrow N( 0 , \tau_3^2), $$ 
where $\tau_3^2$ is defined in \eqref{eq:varianceregularirregular}.

\item If $W$ is regular with respect to $K_2$ and $C_4$, then 
$$n (\hat f(G_n) - f(W)) \stackrel{D} \rightarrow Z, $$
where the random variable $Z$ is defined in \eqref{eq:distributionregular}. 
\end{itemize}

\end{prop}

        \begin{figure}[h!]
            \centering
            \includegraphics[scale=0.825]{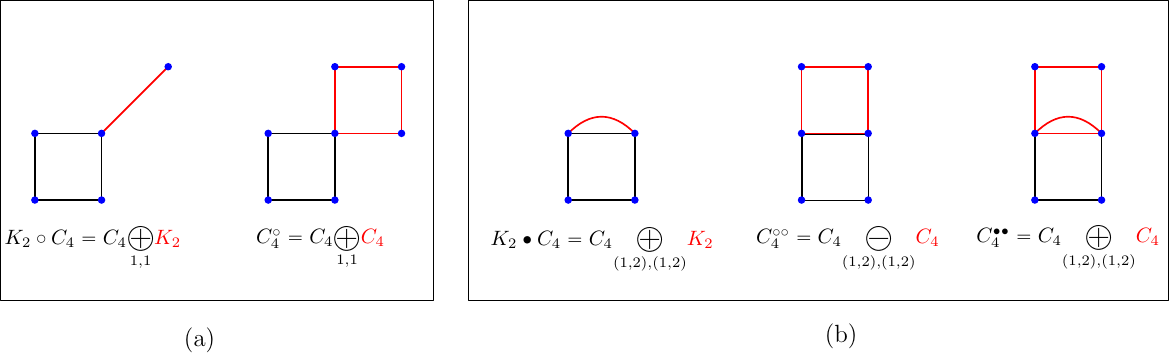}
            \caption{\small{ Graphs obtained from (a) vertex join operations and (b) edge join operations, between a copy of $K_2$ and a copy of $C_4$. }} 
            \label{fig:join42}
        \end{figure}

\begin{proof} 
Fix a graphon $W$ such that $|t(K_{2},W)^{4} - t(C_{4},W)|>0$.  We consider the 4 cases separately, depending on whether or not $W$ is $K_2$ or $C_4$-regular. 
 
\begin{itemize}

\item[{\it Case} 1:] {\it $W$ is irregular with respect to $K_{2}$ and $C_{4}$.} In this case, Corollary \ref{cor:irregjoint} gives, 
\begin{align*}
    \bm Z_{n} = \sqrt{n}
    \begin{pmatrix}
    \hat{t}(K_{2}, G_n) - t(K_{2},W)\\
    \hat{t}(C_{4}, G_n) - t(C_{4},W)
    \end{pmatrix}
    \overset{D}{\rightarrow}
    N_2\left(\bm{0}, \begin{pmatrix}
       \tau_{11} & \tau_{12}\\
        \tau_{21} & \tau_{22}
    \end{pmatrix} \right) , 
    \end{align*}
where 
\begin{itemize}

\item $\tau_{11} = t(K_{1, 2},W) - t(K_{2},W)^{2}$,

\item  $\tau_{22} = \frac{1}{4}\left[t\left(C_{4}^\circ,W\right) - t(C_{4},W)^{2}\right]$, where $C_{4}^\circ$ is the graph obtained by the vertex join of 2 copies of $C_4$ (as shown in Figure \ref{fig:join42}(a)).

\item $\tau_{12} = \tau_{21} = \frac{1}{2}\left(t(C_{4} \circ K_{2}) - t(C_{4},W)t(K_{2},W)\right)$, where $C_{4} \circ K_{2}$ is the graph obtained the vertex join of $C_4$ and $K_2$ (as shown in Figure \ref{fig:join42}(a)) . 
\end{itemize}
Note that $\sqrt{n}(\hat f(G_n) - f(W)) = \bm{1}^\top \bm Z_{n}$. Hence, by the continuous mapping theorem, 
\begin{align*}
    \sqrt{n}(\hat f(G_n) - f(W))\dto N\left(0,\sigma_{1}^2\right) , 
\end{align*}
with 
\begin{align}\label{eq:varianceirregular}
    \tau_1^2 := \tau_{11} + \tau_{22} - 2 \tau_{12}.
\end{align}

\item[{\it Case} 2:] {\it $W$ is irregular with respect to $K_{2}$ and regular with respect to $C_{4}$.} In this case, $\sqrt{n} (\hat{t}(K_{2}, W^{G_{n}}) - t(K_{2}, W))$ has a non-degenerate Gaussian limit, but $\sqrt{n} (\hat{t}(C_{4}, G_n) - t(C_{4},W))$ has a degenerate limit. In particular, from Theorem \ref{thm:asymp-joint-dist} we know that 
\begin{align*}
    \bm Z_n = \sqrt{n}
    \begin{pmatrix}
    \hat{t}(K_{2}, G_n) - t(K_{2},W)\\
    \hat{t}(C_{4}, G_n) - t(C_{4},W)
    \end{pmatrix}
    \overset{D}{\rightarrow}
    \begin{pmatrix}
    G_{1}\\
    0
    \end{pmatrix} , 
\end{align*}
where $G_{1}\sim N(0, t(K_{1, 2},W)-t(K_{2},W)^{2})$. Since $\sqrt{n} (\hat f(G_n) - f(W)) = \bm{1}^\top \bm Z_{n}$, this implies, 
\begin{align*}
    \sqrt{n}( \hat f(G_n) - f(W) )\dto N(0,\tau_2^2) , 
\end{align*}
where 
\begin{align}\label{eq:varianceirregularregular}
    \tau_2^2 := t(K_{1, 2},W)-t(K_{2},W)^{2} . 
\end{align}

\item[{\it Case} 3:] {\it $W$ is regular with respect to $K_{2}$ and irregular with respect to $C_{4}$.} In this case, $\sqrt{n} (\hat{t}(K_{2}, G_n) - t(K_{2},W))$ has a degenerate limit, but $\sqrt{n} (\hat{t}(C_{4}, G_n) - t(C_{4},W))$ has a non-degenerate Gaussian limit. Hence, applying Theorem \ref{thm:asymp-joint-dist} we have,
\begin{align*}
  \bm{Z}_n =  \sqrt{n}
    \begin{pmatrix}
    \hat{t}(K_{2}, G_n) - t(K_{2},W)\\
    \hat{t}(C_{4}, G_n) - t(C_{4},W)
    \end{pmatrix}
    \overset{D}{\rightarrow}
    \begin{pmatrix}
    0\\
    G_{2}
    \end{pmatrix} , 
\end{align*}
where $G_{2}\sim N\left(0, \frac{1}{4}\left(t(C_4^{\circ},W) - t(C_{4},W)^{2}\right)\right)$. Taking inner product of $\bm{Z}_n$ with $\bm 1$, then gives $$\sqrt{n} (\hat f(G_n) - f(W) ) \dto N(0, \tau_3^2),$$ where 
\begin{align}\label{eq:varianceregularirregular}
    \tau_3^2 := \frac{1}{4}\left(t(C_{4}^{\circ},W) - t(C_{4},W)^{2}\right) . 
\end{align}

\item[{\it Case} 4:] {\it $W$ is regular with respect to both $K_{2}$ and $C_{4}$.} Then from Theorem \ref{thm:asymp-joint-dist}, 
\begin{align}\label{eq:distributionregularjoint}
    \bm{Z}_n = n
    \begin{pmatrix}
    \hat{t}(K_{2}, G_n) - t(K_{2},W)\\
    \hat{t}(C_{4}, G_n) - t(C_{4},W)
    \end{pmatrix} 
    \overset{D}{\rightarrow} \bm Z := 
    \begin{pmatrix}
    G_{1}\\
    G_{2}
    \end{pmatrix} +
    \begin{pmatrix}
    \int\int \left(W(x, y) - \frac{1}{2}t(K_{2},W) \right) \mathrm{d}B_{x} \mathrm{d}B_{y}\\
    \int\int \left(W_{C_{4}}(x, y) - \frac{3}{4}t(C_{4},W) \right) \mathrm{d}B_{x} \mathrm{d}B_{y}
    \end{pmatrix} , 
\end{align}
where the Brownian motion $\{B_t\}_{t \in [0, 1]}$ and $(G_{1}, G_{2})^{\top}
\sim N_2( \bm 0, \Sigma)$ are independent, and the matrix $\Sigma$ is given by: 
\begin{align*}
    \Sigma = 
    \frac{1}{2}\begin{pmatrix}
    t(K_{2},W) - t(C_{2},W) & t(C_{4},W) - t(K_{2} \bullet C_{4} ,W)\\
    t(C_{4},W) - t( K_{2} \bullet C_{4},W) & t(C_{4}^{\circ \circ},W) - t(C_{4}^{\bullet \bullet},W)  
    \end{pmatrix} 
    \end{align*}
  In this case, as shown in Figure \ref{fig:join42}(b), $K_{2} \bullet C_{4}$ is the graph obtained by the strong edge join of $K_2$ and $C_4$, 
 $C_{4}^{\circ \circ}$ is the graph obtained by the weak edge join of 2 copies of $C_4$, and $C_{4}^{\bullet \bullet}$ is the graph obtained by the strong edge join of 2 copies of $C_4$. 

Therefore, 
\begin{align}\label{eq:distributionregular}
\sqrt{n}(\hat f(G_n) - f(W)) = \bm{1}^\top \bm{Z}_{n} \stackrel{D} \rightarrow \bm{1}^\top \bm{Z} : = Z, 
\end{align}
where $\bm{Z}$ as defined in \eqref{eq:distributionregularjoint}. 
\end{itemize}

\end{proof}

\section{Multiple Weiner-It\^{o} Stochastic Integrals}
\label{sec:stochasticintegral}

In this section we recall the basic properties of multiple Weiner-It\^{o} stochastic integrals as presented in \cite{stochasticintegral}. To begin with, let $\{B_t\}_{t \in [0, 1]}$ be the standard Brownian motion in $[0, 1]$. We interpret the Brownian motion as a stochastic measure on $([0, 1], \mathscr{B}([0, 1])$, 
where $\mathscr{B}([0, 1])$ is the sigma-algebra generated by open sets of $[0, 1]$. Specifically, suppose $\{B(A): A\in \mathscr{B}([0, 1])\}$ is a collection of random variables defined on a common probability space $(\Omega, \cF, \mu)$ such that 
\begin{itemize}
\item $B(A)\sim N(0, \lambda(A))$, for all $A\in \mathscr{B}([0, 1])$, where $\lambda(A)$ is the Lebesgue measure of $A$. 

\item For any finite collection of disjoint sets $A_{1},\cdots, A_{t} \in \mathscr B(\mathcal X)$, the random variables $\{B(A_1), B(A_2), \ldots, \mathcal B(A_t)\}$ are independent and 
\begin{align*}
B\left(\bigcup_{s=1}^{t}A_{s}\right) = \sum_{s=1}^{t}B(A_{s}).
\end{align*}
\end{itemize}

For $d \geq 1$, denote by $L^2([0, 1]^d)$ the space of measurable functions $f: \mathcal X^d \rightarrow \R$ such that 
$$\|f\|_2^2:= \int_{[0, 1]^d} |f(x_1, x_2, \ldots, x_d)|^2 \mathrm d x_1 \mathrm d x_2 \ldots, \mathrm d x_d < \infty.$$ 
Define $\mathcal{E}_d \subseteq L^2([0, 1]^d)$ as the set of all elementary functions having the form
\begin{equation}\label{eq:itelty}
f(t_1, t_2, \ldots, t_d) =\sum_{1 \leq i_1, i_2, \ldots, i_d \leq m} a_{i_1, i_2, \ldots, i_d} \bm 1\{ (t_1, t_2, \ldots,  t_d) \in A_{i_1}\times \cdots \times A_{i_d} \}, 
\end{equation}
where $A_1, A_2, \ldots, A_m \in [0, 1]$ are measurable sets which are pairwise disjoint and $ a_{i_1, i_2, \ldots, i_d}$ is zero if two indices are equal. 
The multiple Weiner-It\^{o} integral for functions in $\mathcal{E}_d$ is defined as follows: 

\begin{definition}\label{defn:integral_elementary} (Multiple Weiner-It\^{o} integral for elementary functions) The $d$-dimensional Weiner-It\^{o} stochastic integral, with respect to the standard Brownian motion $\{B_t\}_{t \in [0, 1]}$, for the function $f \in \mathcal{E}_d $ in  \eqref{eq:itelty} is defined as 
$$I_d(f):=\int_{[0, 1]^d} f(x_1, x_2, \ldots, x_d) \prod_{a=1}^d \mathrm d B(x_a):=\sum_{1 \leq i_1, i_2, \ldots, i_d \leq m} a_{i_1, i_2, \ldots, i_d} B(A_{i_1})\times \cdots \times B(A_{i_d}).$$ 
\end{definition} 

The multiple Weiner-It\^{o} integral for elementary functions satisfies the following two properties \citep{stochasticintegral}: 

\begin{itemize}

\item (Boundedness) For $f \in  \mathcal{E}_d$, $\E[I_d(f)^2] \leq d! \|f\|^2 < \infty$.  

\item (Linearity) For $f,g\in \mathcal{E}_d$, 
$I_d(f+g) \stackrel{a.s.}{=} I_d(f)+I_d(g)$. 

\end{itemize}
This shows that $I_d$ is a bounded linear operator from $\mathcal{E}_d $ to $L^2(\Omega, \cF, \mu)$, the collection of square-integrable random variables defined on $(\Omega, \cF, \mu)$. Since $\mathcal{E}_d$ is dense in $L^2([0, 1]^d, \mathscr{B}([0, 1]^d), \lambda^d)$ (by \cite[Theorem 2.1]{stochasticintegral}), using the BLT theorem (see \cite[Theorem I.7]{reedsimon}) $I_d$ can be uniquely extended to $L^2([0, 1]^d, \mathscr{B}([0, 1]^d), \lambda^d)$ by taking limits. (Here, $ \lambda^d$ denotes the Lebesgue measure on $[0, 1]^d$.) This leads to the following definition: 

\begin{definition}\label{defn:integral_II} (Multiple Weiner-It\^{o} integral for general $L^{2}$-functions) The $d$-dimensional Weiner-It\^{o} stochastic integral, with respect to the standard Brownian motion $\{B_t\}_{t \in [0, 1]}$, for a function $f \in L^2([0, 1]^d)$ is defined as the $L^{2}$ limit of the sequence $\{I_d(f_n)\}_{n \geq 1}$,  where $\{f_n\}_{n \geq 1}$ is a sequence such that $f_n \in \mathcal{E}_d$ with $\lim_{n \rightarrow \infty}\|f_n-f\|_2 = 0$. 
This is denoted by:  
\begin{align}\label{eq:integral}
I_d(f):=\int_{[0, 1]^d} f(x_1, x_2, \ldots, x_d) \prod_{a=1}^d \mathrm d B(x_a). 
\end{align}
\end{definition}

As in the case of elementary functions, it can be easily checked that $I_d(f)$ satisfies the following properties: 
\begin{itemize}

\item (Boundedness) For $f \in L^2([0, 1]^d)$, $\E[I_d(f)^2] \leq d! \|f\|_2^2 < \infty$.  

\item (Linearity) For $f, g\in L^2([0, 1]^d)$, 
$I_d(f+g) \stackrel{a.s.}{=} I_d(f)+I_d(g)$. 

\end{itemize}
It is also important to note that multiple Weiner-It\^{o} integrals do not behave like classical (non-stochastic) integrals with respect to product measures, since by definition diagonal sets do not contribute to the stochastic integral. Nevertheless, one can express the multiple Weiner-It\^{o} integral for a product function in terms of univariate stochastic integrals using the Wick product (cf.~\cite[Theorem 7.26]{gaussianhilbert}). In the bivariate case, with 2 functions $f, g \in L^2([0, 1]^2)$, this simplifies to  
\begin{align}\label{eq:fgstochasticintegral}
\int_{[0, 1]} \int_{[0, 1]} f(x) g(y) \mathrm d B(x) \mathrm d B(y) = \int_{[0, 1]} f(x) \mathrm d B(x)  \int_{[0, 1]} g(y) \mathrm d B(y) - \int_{[0, 1]} f(x) g(x) \mathrm d x . 
\end{align}

Another important property is that one can interchange stochastic integrals with infinite sums over an orthonormal and symmetric set of functions, as shown in the following result: 
\begin{prop}\label{prop:I2-f_E} 
Let $f \in L^2([0, 1]^d)$ and $\{\varphi_s\}_{s \geq 1}$ is an orthonormal and symmetric set of functions in  $L^2([0, 1]^d)$. Suppose there exists constants $\{\alpha_s\}_{s \geq 1}$ such that
\begin{align*}
f = \sum_{s \geq 1} \alpha_s \varphi_s,
\end{align*}
is well-defined. Then 
\begin{align}\label{eq:exchangeId}
I_d(f) \stackrel{a.s.}= \sum_{s \geq 1}  \alpha_s I_d(\varphi_s).
\end{align}

\end{prop}
\begin{proof}
    For $N \geq 1$, define the truncated the truncated version of $f$:  
    \begin{align*}
        f_{N} := \sum_{s=1}^{N}\alpha_{s}\varphi_{s}.
    \end{align*} 
    By Bessel's inequality, $\sum_{s\geq 1}\alpha_{s}^2\leq \|f\|^2_2 <\infty$.
    Then by \cite[Lemma 6.8]{laxfunctional} along with the orthonormality of $\{\varphi_{s}\}_{s\geq 1}$ gives,
    \begin{align*}
        \left\|f-f_{N}\right\|_2^2 = o(1).
    \end{align*}
    By the linearity of stochastic integrals we know that $I_{d}(f_{N}) = \sum_{s=1}^{N}\alpha_{s}I_{d}(\varphi_{s})$. To complete the proof it is now enough to show that both RHS and LHS of \eqref{eq:exchangeId} are $L^{2}$ limits of $I_{d}(f_{N})$,  as $N\rightarrow\infty$. For this, using the boundedness property of stochastic integrals note that,
    \begin{align*}
        \E\left[\left(I_{d}(f) - I_{d}(f_{N})\right)^2\right] = \E\left[I_{d}(f - f_{N})^2\right] \leq d!\left\|f-f_{N}\right\|_2^2 = o(1).
    \end{align*}
    This shows that the LHS of \eqref{eq:exchangeId} is the $L^{2}$ limit of $I_{d}(f_{N})$, as $N\rightarrow\infty$. For the RHS note that by definition the functions $\{\varphi_{s}\}_{s\geq 1}$ are symmetric in their arguments. Hence, by \cite[Theorem 7.29]{gaussianhilbert}, for all $s\geq 1$, $I_{d}(\varphi_{s}) = d!J_{d}(\varphi_{s})$, where
    the operator $J_{d}(\cdot)$ is defined as, 
    \begin{align*}
        J_{d}(g) := \int_{D_{d}}g(x_{1},x_{2},\ldots,x_{d})\mathrm{d} B_{x_{1}}\cdots\mathrm{d} B_{x_{d}},  
    \end{align*}
   for $g\in L^{2}(D_{d})$, with 
    \begin{align*}
        D_{d} := \left\{(x_{1},\ldots,x_{d})\in [0,1]^d:0< x_{1}<x_{2}<\cdots<x_{d}< 1\right\}.
    \end{align*}
    By Theorem 7.6 and Theorem 7.3 from \cite{gaussianhilbert} we know that $J_{d}$ is an isometry. Hence,
    \begin{align*}
        \E\left[\left(\sum_{s\geq 1}\alpha_{s}I_{d}(\varphi_{s}) - \sum_{s\geq 1}^{N}\alpha_{s}I_{d}(\varphi_{s})\right)^2\right] = d!^2\E\left[\left(\sum_{s\geq 1}\alpha_{s}J_{d}(\varphi_{s}) - \sum_{s\geq 1}^{N}\alpha_{s}J_{d}(\varphi_{s})\right)^2\right] = o(1) , 
    \end{align*}
    where the last equality follows by noticing that $\{\varphi_{s}\}_{s\geq 1}$ are orthonromal and $J_{d}$ is an isometry. This shows, recalling the linearity of stochastic integrals, that the RHS of \eqref{eq:exchangeId} is the $L^2$ limit of $I_{d}(f_{N})$, which completes the proof.
\end{proof}
\end{document}